\documentclass[10pt,reqno]{amsart}
\RequirePackage[OT1]{fontenc}
\RequirePackage{amsthm,amsmath}
\RequirePackage[numbers]{natbib}
\RequirePackage[colorlinks,linkcolor=blue,citecolor=blue,urlcolor=blue]{hyperref}
\usepackage{amssymb}
\usepackage{anysize}
\usepackage{hyperref}
\usepackage{extarrows}
\usepackage{multirow}
\usepackage{natbib}
\usepackage{stmaryrd}
\usepackage{color}
\usepackage{amsmath}
\usepackage{enumitem}
\usepackage{mathtools} 
\usepackage{dsfont} 
\usepackage{comment}
\usepackage{soul}
\usepackage{braket}
\usepackage{bm}

\DeclareFontFamily{U}{mathx}{}
\DeclareFontShape{U}{mathx}{m}{n}{<-> mathx10}{}
\DeclareSymbolFont{mathx}{U}{mathx}{m}{n}
\DeclareMathAccent{\widecheck}{0}{mathx}{"71}

\numberwithin{equation}{section}
\theoremstyle{plain} 
\newtheorem{theorem}{Theorem}[section]
\newtheorem{lemma}[theorem]{Lemma}

\newtheorem{proposition}[theorem]{Proposition}
\newtheorem{remark}[theorem]{Remark}
\newtheorem{assumption}[theorem]{Assumption}

\theoremstyle{definition}

\newtheorem{example}[theorem]{Example}

\renewcommand{\Re}{\mathrm{Re}\,}
\renewcommand{\Im}{\mathrm{Im}\,}

\newcommand{\hell}{\hat{\ell}}

\newcommand{\E}{{\mathbb E }}

\newcommand{\R}{{\mathbb R }}
\newcommand{\N}{{\mathbb N}}
\newcommand{\Z}{{\mathbb Z}}

\newcommand{\C}{{\mathbb C}}

\newcommand{\ii}{\mathrm{i}}
\newcommand{\ee}{\mathrm{e}}

\newcommand{\dd}{\mathrm{d}}

\newcommand{\sgn}{\mathrm{sgn}}

\newcommand{\vertiii}[1]{{\left\vert\kern-0.25ex\left\vert\kern-0.25ex\left\vert #1 
		\right\vert\kern-0.25ex\right\vert\kern-0.25ex\right\vert}}

\newcommand{\nc}{\normalcolor}

\newcommand{\bs}{\boldsymbol}

\def\Tr{\mathrm{Tr}}

\def\<{\langle}
\def\>{\rangle}

\renewcommand{\mathbf}[1]{\bs{#1}}
 
\marginsize{25mm}{25mm}{25mm}{26mm}

\allowdisplaybreaks

\title[Eigenstate Thermalisation at the edge for Wigner Matrices]{Eigenstate Thermalisation at the edge for Wigner Matrices}
\author[Cipolloni \and Erd\H{o}s \and Henheik]{Giorgio Cipolloni \qquad \quad L\'aszl\'o Erd\H{o}s\(^{*}\) \qquad \quad Joscha Henheik\(^{*}\)}
\address{G.C., Princeton Center for Theoretical Science and Department of Mathematics
	Princeton University, Princeton, NJ 08544, USA}
\email{gc4233@princeton.edu} 

\address{L.E.~and J.H., IST Austria, Am Campus 1, 3400 Klosterneuburg, Austria}
\email{lerdos@ist.ac.at}
\email{joscha.henheik@ist.ac.at}

\thanks{\(^*\)Supported by ERC Advanced Grant ``RMTBeyond'' No.~101020331}

\date{\today} 
\begin{document}
	 \begin{abstract}  We prove the Eigenstate Thermalisation Hypothesis for Wigner matrices uniformly in
		the entire spectrum, in particular near the spectral edges,
		with a bound on the fluctuation that is optimal for any observable. 
		This complements  earlier works of Cipolloni et.~al.~\cite{ETHpaper, A2}
		and Benigni et.~al.~\cite{BenigniLopatto2103.12013, 2303.11142} that	were restricted either to the bulk of
		the spectrum or to special observables.  As a main ingredient, we prove a new 
		multi-resolvent local law that optimally  accounts for the edge scaling. 
	\end{abstract}

	\maketitle

\section{Introduction}

In the physics literature, the \emph{Eigenstate Thermalisation Hypothesis (ETH)} asserts that 
each eigenfunction of a sufficiently chaotic quantum system 
is uniformly distributed in the phase space. This concept was coined by Srednicki \cite{Srednicki}
after similar ideas appeared in the seminal paper of Deutsch \cite{deutsch1991}.
While the original physics setup concerns genuine many-body systems, especially
a small system in a heat bath described by standard statistical mechanics, Deutsch 
has  also formulated a phenomenological version of  ETH for the simplest chaotic quantum system,
 the Wigner ensemble. In this form, ETH asserts that
  for any deterministic observable (matrix) $A$
and for any normalised eigenvector $\bm{u}$ of a large  $N\times N$ 
Wigner matrix, the quadratic form $\langle \bm{u}, A\bm{u}\rangle$
is very close to its statistical average, which, in the Wigner case,
 is the normalized trace $\langle A\rangle : =\frac{1}{N} \Tr A$:
 \begin{equation}\label{eth}
  |\langle \bm{u}, A\bm{u}\rangle - \langle A\rangle|\lesssim \frac{\| A\|}{\sqrt{N}}.
 \end{equation}
The $1/\sqrt{N}$ speed of convergence is optimal and it is  in agreement with 
the earlier predictions of Feingold and Peres \cite{FeinPeres}, see also \cite{EckFisch}.
For more physics background and references, see
the introduction of \cite{ETHpaper}.

In the mathematics literature the same phenomenon  is known as the 
\emph{Quantum Unique Ergodicity  (QUE)}. In precise mathematical terms it was formulated by
Rudnick and Sarnak \cite{RudnickSarnak1994}
for standard quantisations of ergodic classical dynamical systems  and proved only in some special cases \cite{Lindenstrauss, Soundararajan, HoloSound, BrooksLindenstrauss}, often as a purely limiting statement 
without optimizing the  speed of convergence.
The key point is to control the behaviour of \emph{all} eigenfunctions; a
similar result for \emph{most} eigenfunctions  (called \emph{Quantum Ergodicity}) is  much easier and
has  been earlier discovered by \u{S}nirel’man \cite{snirelman1974}, see also \cite{ColinDeVerdiere1985, Zelditch1987}.

Motivated by the  paper of Deutsch \cite{deutsch1991} and the novel
technical developments in random matrix theory,
the ETH for Wigner matrices in the form~\eqref{eth}   has been the object of 
intense study in recent years.   An important question is
 the precise dependence of the error term in the right hand side on $A$. 
 The first proof of \eqref{eth} given in \cite{ETHpaper} involved 
  the operator norm $\lVert \mathring{A}\rVert$ of the traceless part $ \mathring{A} :=A-\langle A\rangle$ 
  of $A$, but
 this estimate is far from optimal for low rank observables. For example,
if $A= |\bm{q}\rangle\langle \bm{q}|$ is  a rank--one projection 
onto a normalised vector $\bm{q}\in\C^N$, then $\langle \bm{u}, A\bm{u}\rangle = | \langle \bm{u}, \bm{q}\rangle|^2$ 
which is known to be essentially of order $1/N$ by the \emph{complete delocalisation of eigenvectors},
see \cite{ESY2009, EYY2012, KnowYin, BEKYY, BenLopDeloc}. However the result in  \cite{ETHpaper} gives only 
the suboptimal estimate  $ | \langle \bm{u}, \bm{q}\rangle|^2 \lesssim 1/\sqrt{N}$ for this special observable.

 In the Gaussian (GUE/GOE) case, the eigenvectors are uniformly Haar distributed,
hence explicit moment calculations for $\langle \bm{u}, A\bm{u}\rangle$ are possible by Weingarten calculus. 
The result \nc indicates the following optimal form of~\eqref{eth}:
\begin{equation}
\label{eth1}
	\left| \langle \bm{u}_i, A \bm{u}_j\rangle - \delta_{ij} \langle A \rangle \right| \lesssim
		\frac{\langle |\mathring{A}|^2 \rangle^{1/2}}{\sqrt{N}}.
	\end{equation}
Note that this estimate  involves the (normalised) 
 Hilbert-Schmidt norm $\langle |\mathring{A}|^2 \rangle^{1/2}$   instead of the operator norm\footnote{Note that 
$\langle |\mathring{A}|^2 \rangle^{1/2}$ 
 is substantially smaller than $\lVert \mathring{A}\rVert$ for matrices  $\mathring{A}$ of low
rank;   in fact, if $\mbox{rank}(\mathring{A})=1$, then 
 $\lVert \mathring{A}\rVert= \sqrt{N} \langle |\mathring{A}|^2 \rangle^{1/2}$, losing the entire $\sqrt{N}$ factor
 in~\eqref{eth} compared with the optimal~\eqref{eth1}.  },
and  it can also be extended to different eigenvectors $\bm{u}_i, \bm{u}_j$.
In particular, \eqref{eth1} recovers the optimal delocalisation bound for eigenvectors
as a special case.

 The optimal form of ETH \eqref{eth1} for any Wigner matrix was proved 
   for the special case when $A$ is a projection in \cite{BenigniLopatto2103.12013, 2303.11142},
and for arbitrary $A$ but only  in the bulk of the spectrum\footnote{We point out that the end of the proof
of Theorem 2.2 in  the published version of  \cite{A2} contained a small error; a 
correct and in fact simpler argument
was given in the updated arXiv:2203.01861 version of the paper.}  in \cite{A2}.  In fact, 
$\sqrt{N} [\langle \bm{u}_i, A \bm{u}_j\rangle - \delta_{ij} \langle A\rangle ]$ is asymptotically normal
with variance proportional to $\langle |\mathring{A}|^2 \rangle^{1/2}$  (see \cite{normalfluc, A2})
in the bulk,
showing that the Hilbert-Schmidt norm 
$\langle |\mathring{A}|^2 \rangle^{1/2}$ is indeed the optimal one.
In the main theorem  of the current paper 
(see Theorem~\ref{thm:ETH} below) we prove \eqref{eth1} for all observables and 
all eigenfunctions, giving  the optimal ETH for Wigner matrices
in all spectral regimes.

We remark that ETH is expected to hold for much more general random matrix ensembles. For example the approach in \cite{ETHpaper} could be  directly generalized 
to a special class of generalized Wigner matrices
in \cite{GenWigETH}. 
Furthermore, ETH in the bulk has recently been extended
to deformed random matrix ensembles \cite{iidpaper, equipart}, 
where both the leading  term $\delta_{ij} \langle A \rangle$
and the correct replacement for the traceless  part of $A$ in the right hand side of~\eqref{eth1} 
became considerably more involved, in particular they are energy dependent.
The edge regime and the optimal norm of $A$ in the error term are still open questions for these ensembles,
but we leave them to further works and for simplicity we focus on the Wigner case here.

The key tool to achieve our ETH is a new \emph{multi--resolvent local law} with traceless observables \nc 
that is optimal
at the spectral edges.
 Multi--resolvent local laws in general  refer to concentration results for alternating products of resolvents 
of a random matrix and deterministic matrices (observables). \nc  Their  proofs are typically more difficult 
at the spectral edges since,   besides correctly accounting for the traceless observables, 
their optimal form also \nc requires to exploit 
a  delicate cancellation mechanism; the smallness of the local density of eigenvalues 
must accurately  compensate for the linear instability of a nonlinear equation that governs the fluctuation.
In contrast to the previous proofs of local laws behind ETH results, here we apply a dynamical 
approach, the \emph{method characteristic flow} complemented with a Green function comparison
argument.
 While this method has already been  extensively
tested for single resolvent local laws \cite{Bourgade2021, HuangLandon, AdhiHuang, LLS, AdhiLandon, LandSos, AggaHuang}, 
only two papers concern  the multi-resolvent situation \cite{2210.12060, bourfalc},
neither of them focuses on the critical  edge behaviour. 
On top of the edge, we will need to track another key aspect of the local law; in order to obtain 
the Hilbert-Schmidt norm in \eqref{eth1}, the same norm must appear in the local law as well.
Typically, errors in the local laws involve the operator
 norm of the deterministic matrices between the resolvents,
the control in the much harder Hilbert-Schmidt sense was considered only very recently in \cite{A2}.
However, this work did not track the optimal edge behaviour. Our new local law  is simultaneously optimal
in both aspects. We will explain the strength of this new local law in the context of previous works
in Section~\ref{sec:loclaw}.

\subsection*{Notations}
By $\lceil x \rceil := \min\{ m \in \Z \colon m \ge x \}$ and $\lfloor x \rfloor := \max\{ m \in \Z \colon m \le x \}$
we denote the upper and lower integer part of a real number $x \in \R$.
We set $[k] := \{1, ... , k\}$ for $k \in \N$ and $\langle A \rangle := d^{-1} \mathrm{Tr}(A)$, $d \in \N$,
for the normalised trace of a $d \times d$-matrix $A$.
For positive quantities $A, B$ we write $A \lesssim B$ resp.~$A \gtrsim B$ and mean that $A \le C B$ resp.~$A \ge c B$ for some $N$-independent constants $c, C > 0$ that depend only on the basic control parameters of the model
in Assumption~\ref{ass:entries} below. Moreover, for $N$-dependent positive quantities $A, B$, we write $A \ll B$ whenever $A/B \to 0$ as $N \to \infty$.

 We denote vectors by bold-faced lower case Roman letters $\boldsymbol{x}, \boldsymbol{y} \in \C^{N}$, for some $N \in \N$, and define 
\begin{equation*}
	\langle \boldsymbol{x}, \boldsymbol{y} \rangle := \sum_i \bar{x}_i y_i\,, 
	\qquad A_{\boldsymbol{x} \boldsymbol{y}} := \langle \boldsymbol{x}, A \boldsymbol{y} \rangle\,.
\end{equation*}
Matrix entries are indexed by lower case Roman letters $a, b, c , ... ,i,j,k,... $ from the beginning or the middle of the alphabet and unrestricted sums over those are always understood to be over $\{ 1 , ... , N\}$. 

Finally, we will use the concept  \emph{'with very high probability'},  meaning that for any fixed $D > 0$, the probability of an $N$-dependent event is bigger than $1 - N^{-D}$ for all $N \ge N_0(D)$. Also, we will use the convention that $\xi > 0$ denotes an arbitrarily small positive exponent, independent of $N$.
 Moreover, we introduce the common notion of \emph{stochastic domination} (see, e.g., \cite{semicirclegeneral}): For two families
\begin{equation*}
	X = \left(X^{(N)}(u) \mid N \in \N, u \in U^{(N)}\right) \quad \text{and} \quad Y = \left(Y^{(N)}(u) \mid N \in \N, u \in U^{(N)}\right)
\end{equation*}
of non-negative random variables indexed by $N$, and possibly a parameter $u$, we say that $X$ is stochastically dominated by $Y$, if for all $\epsilon, D >0$ we have 
\begin{equation*}
	\sup_{u \in U^{(N)}} \mathbf{P} \left[X^{(N)}(u) > N^\epsilon Y^{(N)}(u)\right] \le N^{-D}
\end{equation*}
for large enough $N \ge N_0(\epsilon, D)$. In this case we write $X \prec Y$. If for some complex family of random variables we have $\vert X \vert \prec Y$, we also write $X = O_\prec(Y)$. 
\subsection*{Acknowledgment.} We thank Volodymyr Riabov for his help with creating Figure \ref{fig:flow}  and for pointing out the missing condition $|\sigma|< 1$ in Assumption \ref{ass:entries}.
\section{Main results} \label{sec:mainres}

We consider $N\times N$ Wigner matrices $W$, i.e.~$W$ is a random real symmetric or complex Hermitian  matrix $W=W^*$ with independent entries (up to the Hermitian symmetry) and with
 single entry distributions $w_{aa}\stackrel{\dd}{=}N^{-1/2}\chi_{\dd}$, and $w_{ab}\stackrel{\dd}{=}N^{-1/2}\chi_{\mathrm{od}}$, for $a>b$. The random variables $\chi_{\mathrm{d}},\chi_{\mathrm{od}}$ 
 satisfy the following assumptions.\footnote{By inspecting our proof, it is easy to see that  actually we do not need to assume that the off-diagonal entries of $W$ are  identically distributed. We only need that they all have the same second moments, but higher moments can be different.}
\begin{assumption}
\label{ass:entries}
The off-diagonal distribution $\chi_{\mathrm{od}}$ is a real or complex centered random variable, $\E\chi_{\mathrm{od}}=0$, with $\E|\chi_{\mathrm{od}}|^2 = 1$, and we have that $\sigma := \E \chi_{\mathrm{od}}^2$ satisfies $|\sigma| < 1$.
The diagonal distribution is a real centered random variable, $\E \chi_{\mathrm{d}} =0$. Furthermore, we assume the existence of high moments, i.e.~for any $p\in \N$ there exists $C_p > 0$ such that
\[
\E \big[|\chi_{\mathrm{d}}|^p+|\chi_{\mathrm{od}}|^p\big]\le C_p\,. 
\]
\end{assumption}
Our main result is the optimal form of the eigenstate thermalization hypothesis (ETH) for Wigner matrices uniformly in the spectrum, in particular, including the spectral edges. Its proof is given in Section~\ref{subsec:proofETH} and it is
based on a new  \emph{multi-resolvent local law}, Theorem~\ref{thm:main} below.

\begin{theorem}[Eigenstate Thermalization Hypothesis] \label{thm:ETH}
	Let $W$ be a Wigner matrix satisfying Assumption~\ref{ass:entries} with orthonormalized eigenvectors $\bm{u}_1, ... , \bm{u}_N$ and let $A \in \C^{N \times N}$ be deterministic. Then
	\begin{equation}\label{eq:eth}
		\max_{i,j \in [N]} \left| \langle \bm{u}_i, A \bm{u}_j\rangle - \delta_{ij} \langle A \rangle \right| \prec 
		\frac{\langle |\mathring{A}|^2 \rangle^{1/2}}{\sqrt{N}}
	\end{equation}
	where $\mathring{A} := A - \langle A \rangle$ denotes the traceless part of $A$. 
	\end{theorem}

\subsection{Multi-resolvent local laws}\label{sec:loclaw}

Consider the resolvent $G(z):=(W-z)^{-1}$, with $z\in\C\setminus\R$. It is well known 
 that in the limit $N\to \infty$ the resolvent becomes deterministic, with its deterministic approximation 
 $m_{\mathrm{sc}}(z)\cdot I$, where $m_{\mathrm{sc}}$ is  the Stieltjes transform of the semicircular law:
\begin{equation}
\label{eq:semicirc}
m(z):=m_{\mathrm{sc}}(z)=\int_\R\frac{1}{x-z}\rho_{\mathrm{sc}}(x)\,\dd x, \qquad\quad \rho_{\mathrm{sc}}(x):=\frac{1}{2\pi}\sqrt{[4-x^2]_+}.
\end{equation}
This  holds even in the local regime as long as $|\Im z|\gg N^{-1}$; such concentration results are commonly called 
{\it local laws.}  

The single resolvent local law, in its simplest form\footnote{Traditionally \cite{EYY2012, KnowYin, BEKYY}, 
 local laws did not consider arbitrary test matrix $A$, but only $A=I$ or special rank one projections
 leading the \emph{isotropic local laws}. General $A$ was included later, e.g.~in~\cite{slowcorr}.}, asserts  that
\begin{equation}
\label{eq:singleG}
\big|\langle (G(z)- m(z))A\rangle\big|\prec \frac{\| A \|}{N\eta}, \qquad \eta:=| \Im z|,
\end{equation}
holds for any deterministic matrix (\emph{observable}) $A$. The $1/N\eta$ error is optimal
for $A=I$
in the relevant $\eta\lesssim1 $ regime
and $N\eta \langle G(z)-m(z)\rangle$ is approximately Gaussian with variance of order one \cite{HeKnowles}.
However, for traceless observables, i.e. $\langle A\rangle =0$, 
hence $A=\mathring{A}$, the bound in~\eqref{eq:singleG} improves to
the optimal form,
\begin{equation*}
\big|\langle (G(z)- m(z))A\rangle\big| =  \big|\langle G(z)A\rangle\big| \prec \frac{  \sqrt{\rho(z)}}{N\sqrt{\eta}}
\langle |A|^2\rangle^{1/2},
 \qquad \rho(z):= \frac{1}{\pi}| \Im m( z)|.
\end{equation*}
The improvement in the $\eta$-power  together and the additional density factor $\rho(z)$ 
relevant near the spectral edges
were first observed in~\cite{ETHpaper}, while the optimal dependence on the  
Hilbert-Schmidt norm of $A$ was proved in~\cite{A2}.
Single resolvent local laws, however, are not sufficient to control the eigenfunction overlaps as in \eqref{eq:eth}.
While the local law, via the spectral decomposition of $\Im G = \frac{1}{2\ii}(G-G^*)$,
\begin{equation}
\label{eq:res}
    \langle \Im G(z)A\rangle  = \frac{1}{N} \sum_i \frac{\eta}{(\lambda_i-E)^2 +\eta^2} \langle \bm{u}_i, A \bm{u}_i\rangle,
    \qquad z=E+\ii \eta,
\end{equation}
gives an effectively local average of approximately $N\eta$ diagonal 
overlaps $\langle \bm{u}_i, A \bm{u}_i\rangle$, inferring
the size of a single overlap is not possible just from this average since $\langle \bm{u}_i, A \bm{u}_i\rangle$ may change sign as $i$ varies.

Two-resolvent local laws are much more powerful. In particular, using
\begin{equation}\label{specdec}
    \langle \Im G(z_1)A\Im G(z_2)A^*\rangle  = 
    \frac{1}{N} \sum_{i,j} \frac{\eta}{(\lambda_i-E_1)^2 +\eta^2} 
    \frac{\eta}{(\lambda_j-E_2)^2 +\eta^2}|\langle \bm{u}_i, A \bm{u}_j\rangle|^2,
    \quad z_l=E_l+\ii \eta, \;\; l=1,2,
\end{equation}
we see that for a traceless observable, $\langle A\rangle=0$,
a bound of the form 
\begin{equation}
\label{eq:2G}
 \langle \Im G(z_1)A\Im G(z_2)A^*\rangle \prec \| A\|^2
\end{equation}
 at $\eta\sim N^{-1+\xi}$, $\xi>0$,
would imply that a local average (in both indices) of $|\langle \bm{u}_i, A \bm{u}_j\rangle|^2$
is bounded by $N^{-1+2\xi}\|A\|^2$. Since $|\langle \bm{u}_i, A \bm{u}_j\rangle|^2$ is positive (unlike 
$\langle \bm{u}_i, A \bm{u}_i\rangle$ in \eqref{eq:res}), we can deduce the optimal bound 
$|\langle \bm{u}_i, A \bm{u}_j\rangle|^2\prec \frac{1}{N} \|A\|^2$ for each overlap. This argument 
in this form is valid only in the bulk; near the spectral edges the right hand side of \eqref{eq:2G} needs to be improved to
$\rho(z_1)\rho(z_2)\|A\|^2$; this was already achieved in~\cite{ETHpaper}. However, to obtain 
the optimal Hilbert-Schmidt norm  of the observable in~\eqref{eq:eth}
a second improvement to the form
\begin{equation}
\label{eq:2Gimproved}
 \langle \Im G(z_1)A\Im G(z_2)A^*\rangle \prec \rho(z_1)\rho(z_2) \langle |A|^2\rangle, \qquad  \langle A\rangle=0, 
\end{equation}
is necessary. The main achievement of the current paper is to extract both types
of improvement \emph{simultaneously.}

While Theorem~\ref{thm:ETH} requires only the upper bound~\eqref{eq:2Gimproved}
for  $\Im G A \Im G A$, along its proof other alternating products 
 of resolvents (with or without $\Im$)
 and deterministic matrices emerge. 
More precisely, setting  $G_i:=G(z_i)$ and considering 
 deterministic matrices $B_i$,  the main object of interest is
\begin{equation}
\label{eq:mainobj}
G_1B_1G_2B_2G_3\dots B_{k-1} G_k
\end{equation}
for some fixed $k$. We will call expressions of the form~\eqref{eq:mainobj}  \emph{(resolvent) chains}.
We will show a {\it multi-resolvent local law}, i.e. 
that any chain \eqref{eq:mainobj} concentrates around a deterministic object  and give an upper bound
on the fluctuation.  The  interesting regime is the local one, i.e. when $|\Im z_i|\ll 1$.
We will also consider the case when some of the $G_i$'s are replaced by their imaginary part $\Im G_i$, 
and we will  show
 that in this case the fluctuations are reduced close to the edge of the spectrum by some factor of $|\Im m(z_i)|$
which is essentially the density $\rho_{\mathrm{sc}}$ at $\Re z_i$.

It turns out \cite{ETHpaper} 
that the sizes of both the deterministic limit of~\eqref{eq:mainobj} and its fluctuation are substantially reduced
if some of the  matrices $B_i$ are traceless. 
Therefore, in the main part of the paper we  study \eqref{eq:mainobj} when all the matrices $B_i$ are  traceless,
  $\langle B_i\rangle =0$, this will also imply a local law for \eqref{eq:mainobj} for generic $B_i$'s 
  using that any  matrix $B$  can be decomposed into a constant and a traceless part 
  as $B=\langle B\rangle\cdot I+\mathring{B}$.

 We will prove local laws that are optimal simultaneously in the two different aspects mentioned above 
 in addition to account for the improvement owing to the traceless observables.
The first aspect is to trace the improvement near the  spectral edges in terms of additional $\rho$-powers;
in general the presence of each $\Im G$ provides an additional $\rho$ factor.
 Second,  instead of the technically much easier Euclidean matrix  norm (operator norm)
 of the $B_i$'s,   we need to use the more sophisticated Hilbert-Schmidt norm. 
 One additional advantage of using the Hilbert-Schmidt norm  is that it enables us to test 
 the chain in~\eqref{eq:mainobj} against  rank one matrices
  and still get optimal bounds. 
In particular, testing it  against the projection $|\bm{x}\rangle\langle \bm{y}|$
 immediately gives the so-called {\it isotropic local laws}, 
i.e. concentration for the
individual matrix elements $\langle \bm{x}, G_1B_1\dots B_{k-1} G_k\bm {y}\rangle$, for
any deterministic vectors $ \bm{x},\bm {y}$. 

Our results also hold for the case when 
the spectral parameters $z_i$'s are different,
 but we will not explore the additional potential improvements
from this fact since it is not needed for ETH.
 While in some part of the argument we track   the different values of $|\Im z_i|$
precisely (instead of overestimating them by the worst one), we will not exploit the additional gain 
from possibly different real parts $\Re z_i$; this study is left for future investigations.

  Multi-resolvent local laws for
    chains \eqref{eq:mainobj}  with traceless deterministic matrices
   have been the object of interest in several recent papers,
    however in each  of these works only one aspect of the fluctuations of \eqref{eq:mainobj} was taken into consideration: either the problem was optimal only in the bulk of the spectrum \cite{A2}, 
    hence missing $\rho$ factors were ignored, 
    or the error term was estimated using the crude operator norm of the $B_i$ \cite{ETHpaper, multiG},
    or only chains of length one ($k=1$) had an optimal error term  in both aspects~\cite{functionalCLT}.
    Our new result  (Theorem~\ref{thm:main} below)  does not have any of these restriction: 
    we give a bound on the fluctuation of \eqref{eq:mainobj} uniformly in the spectrum 
 with optimal $N$- and $\rho$-powers and with the  Hilbert-Schmidt norm on the traceless $B_i$'s.

\subsubsection{Preliminaries on the deterministic approximation} \label{sec:prelim}
Before stating our main technical result we introduce some additional notation. Given a  non-crossing partition $\pi$ of 
the set $[k]:=\set{1,\ldots,k}$ arranged in cyclic order,  the partial trace $\mathrm{pTr}_{\pi}$ 
of an ordered set of matrices $B_1, \ldots, B_{k-1}$ is defined as
\begin{equation}
	\label{eq:partrdef}
	\mathrm{pTr}_\pi(B_1,\ldots,B_{k-1}): = \prod_{S\in\pi\setminus \mathfrak{B}(k)}\left\langle\prod_{j\in S}B_j\right\rangle\prod_{j\in \mathfrak{B}(k)\setminus\set{k}} B_j,
\end{equation} 
with $\mathfrak{B}(k)\in\pi$ denoting the unique block containing $k$. Then, for generic $B_i$'s, the deterministic approximation of \eqref{eq:mainobj} is given by \cite[Theorem~3.4]{thermalisation}:
\begin{equation}
	\label{eq:Mdef}
	M_{[1,k]} = M(z_1,B_1,\ldots,B_{k-1},z_k) := \sum_{\pi\in\mathrm{NC}([k])}\mathrm{pTr}_{K(\pi)}(B_1,\ldots,B_{k-1}) \prod_{S\in\pi} m_\circ[S ],
\end{equation}
where $\mathrm{NC}([k])$ denotes the non-crossing partitions of the set $[k]$, and $K(\pi)$ denotes the Kreweras complement of $\pi$ (see \cite[Definition~2.4]{thermalisation} and \cite{Kreweras}).  Furthermore, for any subset $S\subset [k]$ we define $m[S]:=m_\mathrm{sc}[\bm{z}_S]$ as the iterated divided difference of $m_\mathrm{sc}$ evaluated in $\bm{z}_S:=\{z_i: i\in S\}$
which can also be written as
\begin{equation}\label{msc dd}
	m[S]=m_\mathrm{sc}[\bm{z}_S] =m_\mathrm{sc}[\set{z_i: i \in S}] = \int_{-2}^2\rho_\mathrm{sc}(x)\prod_{i\in S}\frac{1}{x-z_i}\dd x.
\end{equation}
We denote by $m_\circ[\cdot ]$  the free-cumulant transform of $m[\cdot]$ which is uniquely defined implicitly by the relation
\begin{equation}
	\label{eq:freecumulant}
	m[S] = \sum_{\pi\in\mathrm{NC}(S)} \prod_{S'\in\pi} m_\circ[S'],  \qquad \forall S\subset [k],
\end{equation}
e.g. $m_\circ[i,j]=m[\set{i,j}]-m[\set{i}]m[\set{j}]$. For example, for $k=2$ we have
\begin{equation}
	\begin{split}
		M(z_1,B_1,z_2) &= \langle B_1\rangle(m_\mathrm{sc}[z_1,z_2]-m_\mathrm{sc}(z_1)m_\mathrm{sc}(z_2)) + B_1 m_\mathrm{sc}(z_1)m_\mathrm{sc}(z_2) \\
		&= \frac{\langle B_1\rangle}{2\pi}\int_{-2}^2\frac{\sqrt{4-x^2}}{(x-z_1)(x-z_2)}\dd x + (B_1-\braket{B_1}) m_\mathrm{sc}(z_1)m_\mathrm{sc}(z_2).
	\end{split}
\end{equation}

 The main objects of interest  within this section are general resolvent chains 
\begin{equation}
	\mathcal{G}_1 B_1\mathcal{G}_2 B_2\dots B_{k-1}\mathcal{G}_k
\end{equation}
where 
$\mathcal{G}_i \in \{G_i, \Im G_i\}$, and we denote by $\mathfrak{I}_k\subset [k]$ 
the set of the indices for which $\mathcal{G}_i=\Im G_i$. 
Note that  some resolvents may be replaced with their imaginary parts.
 In order to generalize \eqref{eq:Mdef},  for any subset $\mathfrak{I}_k\subset [k]$  we 
 define\footnote{Calligraphic letters
 like $ \mathcal{G}, \mathcal{M}$ indicate that we may consider $\Im G$ instead of some resolvents $G$ in the chain.}
\begin{equation}
	\label{eq:Mdefim}
	\mathcal{M}_{[1,k]}= \mathcal{M}(z_1,B_1,\dots,B_{k-1},z_k;\mathfrak{I}_k):=\sum_{\pi\in\mathrm{NC}([k])}\mathrm{pTr}_{K(\pi)}(B_1,\ldots,B_{k-1}) \prod_{S\in\pi} m_\circ^{(\mathfrak{I}_k)}[S],
\end{equation}
with $m_\circ^{(\mathfrak{I}_k)}[S]$ implicitly defined as in \eqref{eq:freecumulant} with $m[S]$ replaced with $m^{(\mathfrak{I}_k)}[S]$, where
\begin{equation} \label{eq:Mdivdiff}
	m^{(\mathfrak{I}_k)}[S]= m^{(\mathfrak{I}_k)}[\set{z_i: i \in S}]: = \int_{-2}^2\rho_\mathrm{sc}(x)\left(\prod_{i\in \mathfrak{I}_k \cap S}
	\Im \frac{1}{x-z_i}\right)\left(\prod_{i\in S\setminus \mathfrak{I}_k}\frac{1}{x-z_i}\right)\dd x.
\end{equation}
We now give some bounds on the deterministic approximations in the case where all matrices in \eqref{eq:Mdefim} are traceless, $\langle B_i \rangle = 0$.\footnote{From now on we use the convention that traceless matrices are denoted by $A$, while general deterministic matrices are denoted by $B$.} 
The proof of the following lemma is presented in Appendix \ref{sec:addtech}. 
\begin{lemma}[$M$-bounds] \label{lem:Mbound}
		Fix $k \ge 1$. Consider spectral parameters $z_1, ... , z_{k+1} \in \C \setminus \R$ and traceless matrices $A_1, ... , A_k \in \C^{N \times N}$. Moreover, let 
		\begin{equation*}
		{\eta}_j := |\Im z_j|\,, \qquad m_j := m_{\rm sc}(z_j)\,, \qquad \rho_j :=\frac{1}{\pi} |\Im m_j|\,.
	\end{equation*}
	\begin{itemize}
		\item[(a)] Denoting $	\ell := \min_{j \in [k]}\big[ \eta_j (\rho_j + \mathbf{1}(j \notin \mathfrak{I}_k))\big]$ and assuming $N \ell \ge 1$, we have the average bound
		\begin{equation} \label{eq:Mbound}
			\left| \langle \mathcal{M}(z_1, A_1, ... , A_{k-1}, z_k; \mathfrak{I}_k) A_k \rangle \right| \lesssim \left(\prod_{i \in \mathfrak{I}_k} \rho_i\right) 
				N^{k/2 - 1}\prod_{j \in [k]} \langle |A_j|^2 \rangle^{1/2}\,. 
		\end{equation}
		\item[(b)] Denoting $	\ell := \min_{j \in [k+1]}\big[ \eta_j (\rho_j + \mathbf{1}(j \notin \mathfrak{I}_{k+1}))\big]$ and assuming $N \ell \ge 1$, we have the isotropic bound\footnote{The isotropic bound for $|\langle \bm x, \mathcal{M} \bm y\rangle|$ 
		in~\eqref{eq:MboundISO} is the
		same as the norm bound  $\| \mathcal{M}\|$.}
		\begin{equation} \label{eq:MboundISO}
			\left| \langle \bm x, \mathcal{M}(z_1, A_1, ... , A_{k}, z_{k+1}; \mathfrak{I}_{k+1}) \bm y \rangle \right| \lesssim \left(\prod_{i \in \mathfrak{I}_{k+1}} \rho_i \right) 
				N^{k/2}\prod_{j \in [k]} \langle |A_j|^2 \rangle^{1/2}\,.  
		\end{equation}
		for arbitrary bounded deterministic vectors $\Vert \bm x \Vert \,, \Vert \bm y \Vert \lesssim 1$.
	\end{itemize}
	\end{lemma}
Note that	\eqref{eq:Mbound} already reflects the different aspects of our local law: it correctly accounts for 
the $\rho$-powers for each $\Im G$, it involves the Hilbert-Schmidt norm of the observables
and it is not hard to see that the $N$-power is also optimal.   Note that
the isotropic bound~\eqref{eq:MboundISO} is  stated separately for convenience, \nc
although it will be a straightforward consequence of the average bound~\eqref{eq:Mbound}.

\subsubsection{Multi-resolvent local law} \label{sec:main}
As our main  input for Theorem \ref{thm:ETH}, we will prove
 the following multi-resolvent local law, optimally accounting for the decay of the density at the edge. 

 \begin{theorem}[Multi-resolvent local law with optimal edge dependence] \label{thm:main}
	Let $W$ be a Wigner matrix satisfying Assumption~\ref{ass:entries}, and fix $k \in \N$. Consider spectral parameters $z_1, \ldots , z_{k+1} \in \C \setminus \R$, the associated resolvents $G_j = G(z_j) := (W-z_j)^{-1}$ with $\mathcal{G}_j \in \{ G_j, \Im G_j\}$, and traceless matrices $A_1, \ldots , A_k \in \C^{N \times N}$. 
	Finally, let
	\begin{equation}
		\label{eq:defpar}
		{\eta}_j := |\Im z_j|\,, \qquad m_j := m_{\rm sc}(z_j)\,, \qquad \rho_j :=\frac{1}{\pi} |\Im m_j|\,, \qquad j\in[ k+1].
	\end{equation}
\begin{itemize}
\item[(a)] Denote by $\mathfrak{I}_k$ the set of indices $j \in [k]$ where $\mathcal{G}_j = \Im G_j$. 
Then, setting 
$$\ell := \min_{j \in [k]}\big[ \eta_j (\rho_j + \mathbf{1}(j \notin \mathfrak{I}_k))\big],
$$ 
we have the \emph{average law}
	\begin{equation} \label{eq:mainAV}
	\left| \langle \mathcal{G}_1 A_1  \mathcal{G}_2\ldots \mathcal{G}_k A_k \rangle - \langle \mathcal{M}_{[1,k]}A_k \rangle \right| \prec \left[\left(\prod_{i \in \mathfrak{I}_k} \rho_i \right) \wedge \max_{i \in [k]} \sqrt{\rho_i}\right] \, 
\frac{N^{k/2 - 1}}{\sqrt{N \ell}} \,  \prod_{j \in [k]} \langle |A_j|^2 \rangle^{1/2}
   \,,
\end{equation}
uniformly in spectral parameters satisfying $ \min_j N \eta_j \rho_j  \ge N^{\epsilon}$ 
and $\max_j |z_j| \le N^{1/\epsilon}$ for some $\epsilon > 0$. 
\item[(b)] Denote by $\mathfrak{I}_{k+1}$ the set of indices $j \in [k+1]$ where $\mathcal{G}_j = \Im G_j$. Then, setting $$\ell := \min_{j \in [k+1]}\big[ \eta_j (\rho_j + \mathbf{1}(j \notin \mathfrak{I}_{k+1}))\big],
$$ we have the \emph{isotropic law}
\begin{equation} \label{eq:mainISO}
	\left| \langle \bm x, \mathcal{G}_1 A_1\mathcal{G}_2 \ldots  A_k \mathcal{G}_{k+1} \bm y \rangle - \langle \bm x,  \mathcal{M}_{[1,k+1]}\bm y \rangle \right| \prec \left[\left(\prod_{i \in \mathfrak{I}_{k+1}} \rho_i \right) \wedge \max_{i\in [k+1]} \sqrt{\rho_i}\right] \,
		\frac{N^{k/2}}{\sqrt{N \ell}} \, \prod_{j \in [k]} \langle |A_j|^2 \rangle^{1/2}  \,,
\end{equation}
uniformly in bounded deterministic vectors $\Vert \bm x \Vert \,, \Vert \bm y \Vert \lesssim 1$ and spectral parameters satisfying $\min_j N \eta_j \rho_j  \ge N^{\epsilon}$ and $\max_j |z_j| \le N^{1/\epsilon}$ for some $\epsilon > 0$. 
\end{itemize}
\end{theorem}

Observe that, in the regime $N \ell \gg 1$, the error terms in \eqref{eq:mainAV} and \eqref{eq:mainISO} are smaller by an additional small $(N \ell)^{-1/2}$-factor compared to the size of the leading terms in \eqref{eq:Mbound} and \eqref{eq:MboundISO}, respectively.

\begin{remark}[Optimality]\label{rmk:opt} The bounds \eqref{eq:mainAV} and \eqref{eq:mainISO}
are optimal (up to the $N^\xi$ factor hidden in the $\prec$-relation)
in the class of bounds that involve only the parameters $N$, $\eta_i$, $\rho_i$ and the Hilbert-Schmidt
norm of $A_i$'s. This fact can be seen by computing the variance of the left hand sides in the case when $W$
is a GUE matrix. The resolvents can be written out by spectral theorem, similarly to~\eqref{specdec}, and the
variance with respect to the eigenvectors can be explicitly computed by Weingarten calculus, while the
variance with respect to the eigenvalues (that are independent of the eigenvectors) can be identified
from well-known central limit theorems for linear statistics of eigenvalues. For example, for $k=2$, $A_1=A_2=A$,
$z_1= z_2 =z$
and $\mathfrak{I}_k=\emptyset$,  in this way we obtain
\begin{equation}\label{var}
\sqrt{   \E\big|\langle GAGA\rangle - m^2 \langle A^2\rangle \big|^2}
   \sim \frac{1}{N\eta}\langle A^2\rangle + \frac{\sqrt{\rho}}{N\sqrt{\eta}} \langle A^4\rangle^{1/2}.
\end{equation}
After estimating $\langle A^4\rangle \le N \langle A^2\rangle^2$, which may saturate for certain $A$,
we see the optimality of \eqref{eq:mainAV}
for this case. The general case is a similar, albeit somewhat tedious calculation. 
\end{remark}
\nc

\begin{remark}[Interpretations]\label{rmk:MHS}
  We have two  further \nc comments on Theorem \ref{thm:main}.
	\begin{itemize} 
\item[(i)]  For $\mathfrak{I}_k = \emptyset$ and $\mathfrak{I}_{k+1} = \emptyset$ both bounds, \eqref{eq:mainAV} and \eqref{eq:mainISO}, have already been proven in \cite[Theorem~2.2 and Corollary 2.4]{A2}. In the complementary cases $\mathfrak{I}_k \neq \emptyset$ and $\mathfrak{I}_{k+1} \neq \emptyset$, we point out that the minimum $\big[ ... \wedge ... \big]$ in \eqref{eq:mainAV} and \eqref{eq:mainISO} is realized by the product  
$\prod_{i\in \mathfrak{I}}\rho_i$  since $\rho_i \lesssim 1$. 
In particular, as a rule of thumb, every index $j$ for which $\mathcal{G}_j = \Im G_j$, decreases both the size of the deterministic approximation \eqref{eq:Mbound}--\eqref{eq:MboundISO} and the size of the error \eqref{eq:mainAV}--\eqref{eq:mainISO} by a factor $\rho_j$, with $\rho_j \le 1$, compared to the case when $\mathcal{G}_j = G_j$. An exception to this rule is \eqref{eq:mainAV} for $k=1$; here
 the bounds for $\langle G A \rangle$ and $\langle \Im G A \rangle$ are identical.

\item[(ii)] 	The estimates in Theorem \ref{thm:main} remain valid if we replace 
\begin{equation} \label{eq:Mreplace}
	\begin{split}
		\langle \mathcal{M}_{[1,k]}A_k \rangle &\longrightarrow \left(\prod_{i\in \mathfrak{I}_k }
		\Im m_i \right)\left(\prod_{i\notin  \mathfrak{I}_k}m_i\right) \langle A_1... A_k \rangle \\
		\langle \bm x, \mathcal{M}_{[1,k+1]}\bm y \rangle &\longrightarrow \left(\prod_{i\in \mathfrak{I}_{k+1} }
		\Im m_i \right)\left(\prod_{i\notin  \mathfrak{I}_{k+1}}m_i\right) \langle \bm x, A_1... A_k \bm y \rangle
	\end{split}
\end{equation}
in \eqref{eq:mainAV} and \eqref{eq:mainISO}, respectively, i.e., if we consider
only the trivial  partition into singletons $\pi$ in the definition~\eqref{eq:Mdefim} of $\mathcal{M}_{[1,k]}$. 
This is simply due to the fact that all other summands 
in \eqref{eq:Mdefim} are explicitly smaller than the error terms in~\eqref{eq:mainAV}--\eqref{eq:mainISO}.
 A proof of this fact is given in Appendix \ref{sec:addtech}. 
\end{itemize}
\end{remark}

\begin{remark}[Generalisations]  We mention a few direct generalisations of Theorem \ref{thm:main} whose proofs
are omitted as they are straightforward. 

\begin{itemize}
\item[(i)] 	In Theorem \ref{thm:main} each $\mathcal{G}$ can be
 replaced by a product of $\mathcal{G}$'s and an individual $\mathcal{G}$ may
also stand for $|G|$, not only for $G$ or $\Im G$ (see \cite[Lemma 3.2]{multiG}, \cite[Lemma 3.1]{A2}, and also Lemma \ref{lem:G^2lemma} below). 
We refrain from stating these results explicitly as they are easily obtained using appropriate integral representations of general products of such
$\mathcal{G}$'s in terms of a single $\Im G$. 

\item[(ii)]  We stated the multi--resolvent local laws in Theorem \ref{thm:main} 
only for $\mathcal{G}_j\in \{G_j,\Im G_j\}$, however, inspecting the proof, one can easily see that
it also leads to a local law for 
$\mathcal{G}_j\in\{G_j, \Im G_j, G^\mathfrak{t}_j,\Im G^\mathfrak{t}_j\}$, where $G^\mathfrak{t}$ stands for
the transpose of $G$.
  In particular, this implies that the ETH in Theorem~\ref{thm:ETH} can also  be extended to
\begin{equation}\label{eq:ethbar}
		\max_{i,j \in [N]} \left| \langle \overline{\bm{u}_i}, A \bm{u}_j\rangle - \langle A \rangle \langle \overline{\bm{u}_i}, \bm{u}_j\rangle \right| \prec 
		\frac{\langle |\mathring{A}|^2 \rangle^{1/2}}{\sqrt{N}}.
	\end{equation}
	Furthermore, setting $\sigma:= \E \chi_{\mathrm{od}}^2$, 
	for $|\sigma|<1$ we have (see \cite[Theorem 2.3]{ETHpaper})
	\[
	\big|\langle \overline{\bm{u}_i}, \bm{u}_j\rangle\big|\prec \frac{C_\sigma}{\sqrt{N}}.
	\]
	In two extreme cases $\sigma=\pm 1$, 
	we have $|\langle \overline{\bm{u}_i}, \bm{u}_j\rangle|=\delta_{i,j}$ if $\sigma=1$ and $|\langle \overline{\bm{u}_i}, \bm{u}_j\rangle|=\delta_{i,N-j+1}$ 
	if $\sigma=-1$ and $\E(W_{aa}^2)=0$
 (see~\cite[Remark 2.4]{ETHpaper}). We remark that here $\bm{u}_i$ denotes the eigenvector corresponding to the eigenvalue $\lambda_i$, with the $\lambda_i$'s labeled in increasing order. 
 	\end{itemize}
	\end{remark}

\subsection{Proof of Theorem~\ref{thm:ETH}}\label{subsec:proofETH}
	
	Fix $\epsilon>0$, pick $E\in [-2,2]$  and define $\eta(E)$ implicitly by
	\[
	N\eta(E)\rho(E+\ii\eta(E))= N^\epsilon.
	\]
	Let $A$ be a traceless matrix $\langle A\rangle=0$, then by spectral decomposition~\eqref{specdec} and the well-known eigenvalue rigidity\footnote{Rigidity asserts that the increasingly ordered eigenvalues $\lambda_i$
	are very close to the $i$-th $N$-quantile $\gamma_i$ of the semicircle density $\rho_{\rm sc}$ in
	the sense $|\lambda_i-\gamma_i|\prec N^{-2/3} [i\wedge (N+1-i)]^{-1/3}$,
	i.e. each eigenvalue is strongly concentrated around the corresponding quantile essentially on the scale
	of the local eigenvalue spacing. }  (see, e.g., \cite{EYY2012}) it is easy to see that (see \cite[Lemma~1]{ETHpaper} for more details)
	\begin{equation*}
		\max_{i,j \in [N]} N\left| \langle \bm{u}_i, A \bm{u}_j\rangle  \right|^2
		\prec N^{2\epsilon}\sup_{E_1,E_2\in [-2,2]} \frac{\big|\langle \Im G(E_1+\ii \eta(E_1))A\Im G(E_2+\ii\eta(E_2)) A^*\rangle\big|}{\rho(E_1+\ii\eta(E_1))\rho(E_2+\ii\eta(E_2))} \prec N^{2\epsilon}\langle |A|^2\rangle\,. 
	\end{equation*}
	We point out that in the last inequality we used  \eqref{eq:mainAV} for $k=2$ and $\mathfrak{I}_2=[2]$:
	\[
	\big|\langle \Im G_1 A \Im G_2 A^*\rangle- \Im m_1\Im m_2\langle |A|^2\rangle\big|\prec  \frac{\rho_1\rho_2}{\sqrt{N\ell}}\langle |A|^2\rangle\,.
	\]
	The fact that this bound holds simultaneously for all $E_1=\Re z_1\in [-2,2]$ and $ E_2=\Re z_2\in [-2,2]$ follows by 
	a simple grid argument together with 
	the Lipschitz regularity of the resolvent  (with Lipschitz constant
	of order $N$ at spectral parameters with imaginary part bigger than $1/N$).
	This completes the proof of Theorem \ref{thm:ETH}. \qed
	
	The rest of the paper is devoted to
	the proof of the multi-resolvent local law, Theorem~\ref{thm:main}.

\section{Multi--resolvent local law: Proof of Theorem \ref{thm:main}} \label{sec:proof}
In this section we prove 
 the \emph{multi-resolvent local laws}
in Theorem \ref{thm:main} via the following three steps: 
\begin{itemize}
\item[\bf 1.] \textbf{Global law.} Prove  a multi-resolvent  \emph{global law}, i.e. for spectral parameters ``far away" from the spectrum, $\min_j \mathrm{dist}(z_j, [-2,2]) \ge  \delta \nc $ for some small $\delta > 0$ (see Proposition \ref{prop:initial}). 
\item[\bf 2.] \textbf{Characteristic flow.} Propagate the global law to a \emph{local law} by considering the evolution of the Wigner matrix $W$ along the Ornstein-Uhlenbeck flow, 
thereby introducing an almost order one Gaussian component (see Proposition \ref{prop:zig}). The spectral parameters evolve from the global regime to the local regime according to the \emph{characteristic (semicircular) flow}. 
The simultaneous effect of these two evolutions is a key cancellation of two large terms. \nc
\item[\bf 3.] \textbf{Green function comparison.} Remove the Gaussian component by a Green function comparison (GFT) argument (see Proposition \ref{prop:zag}).
\end{itemize}

As the first step, we have the following global law. Its proof, which is analogous to the proofs presented in \cite[Appendix~B]{multiG} and \cite[Appendix~A]{A2}, is given in Appendix \ref{app:globallaw} for completeness. 
We point out that all these proofs do not use the system of master inequalities and the bootstrapped error analysis
that form the technical backbone of~\cite{multiG, A2}, they use simple norm bounds on the resolvents. 
In particular, 
Proposition \ref{prop:initial} holds for general deterministic matrices since the traceless condition plays no role in this case. 
\begin{proposition}[Step 1: Global law]
\label{prop:initial} 
Let $W$ be a Wigner matrix satisfying Assumption~\ref{ass:entries}, and fix any $k \in \N$ and $\delta>0$.  Consider spectral parameters $z_1, ... , z_{k+1} \in \C \setminus \R$, the associated resolvents $G_j = G(z_j) := (W-z_j)^{-1}$, with $\mathcal{G}_j \in \{ G_j , \Im G_j \}$, and deterministic matrices $B_1, ... , B_k \in \C^{N \times N}$. Denote $\eta_{i}:=|\Im z_{i}|$ and $\rho_{i}:=\pi^{-1}|\Im m_{\rm sc}(z_{i})|$. 
Then, uniformly in deterministic matrices $B_i$ and in spectral parameters satisfying $\mathrm{dist}(z_j, [-2, 2])\ge \delta$, the following holds.
\begin{itemize}

\item[(a)] Let $\mathfrak{I}_k$ be the set of indices  $j \in [k]$ where $\mathcal{G}_{j}=\Im G_{j}$, and
define
${\ell}:= \min_{j \in [k]}\big[\eta_{j}(\rho_{j}+\bm1(j\notin\mathfrak{I}_{k}))\big]$. Then we have the averaged bound
\begin{equation}
\label{eq:maininAV}
\left| \langle \mathcal{G}_1 B_1 ... \mathcal{G}_k B_k \rangle - \langle \mathcal{M}_{[1,k]}B_k \rangle \right| \prec \left[\left(\prod_{i \in \mathfrak{I}_k} \rho_{i} \right) \wedge \max_{i \in [k]} \sqrt{\rho_{i}} \right] \frac{N^{k/2-1}}{\sqrt{N \ell}} \prod_{j \in [k]} \langle |B_j|^{2} \rangle^{\frac{1}{2}}
 \,. 
	\end{equation}
	\item[(b)] Let $\mathfrak{I}_{k+1}$ be the set of indices  $j \in [k+1]$ where $\mathcal{G}_{j}=\Im G_{j}$, and
	define
	${\ell}:= \min_{j \in [k+1]}\big[\eta_{j}(\rho_{j}+\bm1(j\notin\mathfrak{I}_{k+1}))\big]$. 
	Then, for deterministic unit vectors ${\bm x}, {\bm y}$, we have the isotropic bound
	\begin{equation}
\label{eq:maininISO}
\left| (\mathcal{G}_1 B_1 ... B_k \mathcal{G}_{k+1})_{{\bm x}{\bm y}} - (\mathcal{M}_{[1,k+1]})_{{\bm x}{\bm y}} \right| \prec  \left[\left(\prod_{i \in \mathfrak{I}_{k+1}} \rho_{i} \right) \wedge \max_{i \in [k+1]} \sqrt{\rho_{i}} \right]	\frac{N^{k/2}}{\sqrt{N \ell}} \prod_{j \in [k]} \langle |B_j|^{2} \rangle^{\frac{1}{2}} \,.
\end{equation}
	\end{itemize}
	\nc
\end{proposition}

In the next Proposition~\ref{prop:zig}, using Proposition \ref{prop:initial} 
as an input, we derive Theorem~\ref{thm:main} for Wigner matrices which have an order one Gaussian component.  
For this purpose we consider the evolution of the Wigner matrix $W$ along the Ornstein-Uhlenbeck flow
\begin{equation}
\label{eq:OUOUOU}
\dd W_t=-\frac{1}{2}W_t\dd t+\frac{\dd B_t}{\sqrt{N}},  \qquad W_0=W,
\end{equation}
with $B_t$ being real symmetric or complex Hermitian Brownian motion\footnote{Strictly speaking, we use this Brownian motion only when $\sigma : =  \E \chi^2_{\mathrm{od}}  $ is real and
$\E  \chi^2_{\mathrm{d}} = 1+ \sigma$, otherwise we need a small modification,
see later in Section~\ref{sec:opAedge}.} 
with entries having $t$ times the same first two moments of $W$, and define its resolvent $G_t(z):=(W_t-z)^{-1}$ with $z\in\C\setminus\R$. Even if not stated explicitly we will always consider this flow only for short  times, i.e. for $0\le t\le T$,  where the maximal time $T$ is smaller than 
$\gamma$, for some small constant $\gamma>0$. Note that along the flow \eqref{eq:OUOUOU} the first two moments of $W_t$ are preserved, and so the self-consistent density of states of $W_t$ is unchanged;
it remains the standard semicircle law. We now want to compute the deterministic approximation of product of resolvents and deterministic matrices with trace zero,
\begin{equation}
\label{eq:quantnochar}
\mathcal{G}_t(z_1)A_1\mathcal{G}_t(z_2) A_2\mathcal{G}_t(z_3)A_3\dots, \qquad\quad \langle A_i\rangle =0,
\end{equation}
and have a very precise estimate of the error term.

In fact, we  also let the spectral parameters evolve with time with a carefully chosen equation
that  conveniently cancels some leading 
error terms in the time evolution of \eqref{eq:quantnochar}. 
\begin{figure}
	\centering
	\includegraphics[scale=1.3]{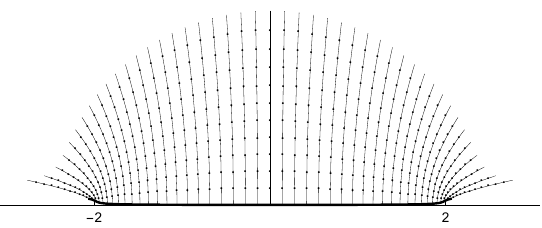}
	\caption{Several trajectories for solutions of \eqref{eq:chardef} are depicted. We chose ten reference times, indicated by dots, showing that the rate of change along the flow strongly depends on $\rho$. 
The solid black line is the graph of $E \mapsto \eta(E)$ with $\eta(E)$ implicitly defined via $\eta(E) \rho(E + \ii \eta(E)) = \mathrm{const}.$ for a  small positive constant. A similar picture also appeared in \cite[Figure 1]{Bourgade2021}.}
	\label{fig:flow}
\end{figure}
The corresponding equation is
the characteristic equation for the semicircular flow, i.e.~given by the first order ODE (see Figure \ref{fig:flow}):
\begin{equation}
\label{eq:chardef}
\partial_t z_{i,t}=-m(z_{i,t})-\frac{z_{i,t}}{2}\,. 
\end{equation}
Define $\eta_{i,t}:=|\Im z_{i,t}|$ and $\rho_{i,t}:=\pi^{-1}|\Im m(z_{i,t})|$. Note that along the characteristics we have
\begin{equation}
\label{eq:characteristics}
\partial_t m(z_{i,t})=-\partial_z m(z_{i,t}) \left(m(z_{i,t})+\frac{z_{i,t}}{2}\right)=-\partial_z m(z_{i,t})\left(-\frac{1}{2m(z_{i,t})}+\frac{m(z_{i,t})}{2}\right)=\frac{m(z_{i,t})}{2},
\end{equation}
where in the last two equalities we used the defining equation $m(z)^2 + zm(z)+1=0$ of the Stieltjes transform of the semicircular law. In particular, taking the imaginary part of~\eqref{eq:characteristics}
we get $\rho_{i,s}\sim \rho_{i,t}$ for any $0\le s\le t$, 
while the behavior of the $\eta_{i,t}$ depends on the regime: in the bulk $\eta_{i,t}$ decreases
 linearly in time with a speed of order one, close to the edge $\eta_{i,t}$ decreases still linearly, but with a 
 speed depending on  $\rho$, i.e. it is slower near the edges.  
 By standard ODE theory we obtain the following lemma:
 \begin{lemma}
 \label{lem:propchar}
Fix an $N$--independent $\gamma>0$, fix $0<T< \gamma \nc$, and pick $z\in \C\setminus \R$.
Then there exists an initial condition $z_0$ such that the solution $z_t$ of  \eqref{eq:chardef} with this initial condition $z_0$ satisfies $z_T=z$. Furthermore, there exists a constant $C>0$ such that \nc $\mathrm{dist}(z_0,[-2,2]) \ge C\nc T$.
 \end{lemma}

 The spectral parameters evolving by
  \eqref{eq:OUOUOU}  will have the property 
  that 
  \begin{equation}
\mathcal{G}_t(z_{1,t})A_1\dots A_{k-1}\mathcal{G}_t(z_{k,t})-\mathcal{M}_{[1,k],t}\approx \mathcal{G}_0(z_{1,0})A_1\dots A_{k-1}\mathcal{G}_0(z_{k,0})-\mathcal{M}_{[1,k],0},
\end{equation}
with $\mathcal{M}_{[1,k],t}:=\mathcal{M}(z_{1,t},A_1,\dots, A_{k-1},z_{k,t})$, for any $0\le t \le T$. 
Note that the deterministic approximation $\mathcal{M}_{[1,k],t}$ depends on time only through 
the time dependence of the spectral parameters.
 The deterministic approximation of \eqref{eq:quantnochar} with fixed spectral parameters
is unchanged along the whole flow \eqref{eq:OUOUOU} since the Wigner semicircular density 
is preserved under the OU flow~\eqref{eq:OUOUOU}.

\begin{proposition}[Step 2: Characteristic flow] \label{prop:zig}
	Fix $\epsilon,\gamma>0$, $0\le T\le \gamma$, $K\in\N$. Consider
	 $z_{1,0},\dots,z_{K+1,0}\in\C \setminus \R$  as initial conditions of the solution $z_{j,t}$  of \eqref{eq:chardef}
	 for  $0\le t\le T$,
define $G_{j,t}:=G_t(z_{j,t})$ 
 and let $\mathcal{G}_{j,t}\in \{G_{j,t}, \Im G_{j,t}\}$. 
 Let $\Vert \bm x \Vert \,, \Vert \bm y \Vert \lesssim1$ be bounded deterministic vectors.

\begin{itemize}

\item[(a)] For any $k\le K$ let $\mathfrak{I}_k$ be the set of indices  $j \in [k]$ where $\mathcal{G}_{j,t}=\Im G_{j,t}$, and
define
${\ell}_{t} := \min_{j \in [k]}\big[\eta_{j,t}(\rho_{j,t}+\bm1(j\notin\mathfrak{I}_{k}))\big]$, 
the time dependent analogue\footnote{We point out that the index $j$ realizing the minimum may change along the time evolution. Additionally, by \eqref{eq:characteristics} and the text below it, we note that if $\min_i N\eta_i\rho_i\ge N^\epsilon$ then $\min_i N\eta_{i,t}\rho_{i,t}\ge N^\epsilon$ for any $0\le t\le T$.}  of $\ell$. 
 Then, assuming that 
\begin{equation}
\label{eq:inass0}
\left| \langle \mathcal{G}_{1,0} A_1 ... \mathcal{G}_{k,0} A_k \rangle - \langle \mathcal{M}_{[1,k],0}A_k \rangle \right|\prec \left[\left(\prod_{i \in \mathfrak{I}_k} \rho_{i,0} \right) \wedge \max_{i \in [k]} \sqrt{\rho_{i,0}} \right]	\frac{N^{k/2 - 1}}{\sqrt{N \ell_{0}}} \prod_{j \in [k]} \langle |A_j|^2 \rangle^{1/2} \,
\end{equation}
holds uniformly for any $k\le K$, any choice of  $A_1,\dots,A_k$ traceless deterministic matrices
and any choice of $z_{i,0}$'s such that 
$N\eta_{i,0}\rho_{i,0}\ge N^\epsilon$ and $ |z_{i,0}| \le N^{1/\epsilon}$,
 then
we have
\begin{equation}
\label{eq:flowgimg}
\left| \langle \mathcal{G}_{1,T} A_1 ... \mathcal{G}_{k,T} A_k \rangle - \langle \mathcal{M}_{[1,k],T}A_k \rangle \right|\prec \left[\left(\prod_{i \in \mathfrak{I}_k} \rho_{i,T} \right) \wedge \max_{i \in [k]} \sqrt{\rho_{i,T}} \right]
	\frac{N^{k/2 - 1}}{\sqrt{N \ell_{T}}} \prod_{j \in [k]} \langle |A_j|^2 \rangle^{1/2}\,,
\end{equation}
for any $k\le K$, again 
uniformly in traceless matrices $A_i$ and
 in spectral parameters satisfying $N \eta_{i, T} \rho_{i, T}  \ge N^{\epsilon}$, $|z_{i,T}|\le N^{1/\epsilon}$.

\item[(b)] Let $\mathfrak{I}_{k+1}$ be the set of indices  $j \in [k+1]$ where $\mathcal{G}_{j,t}=\Im G_{j,t}$, and
 define ${\ell}_{j,t} := \min_{j \in [k+1]} \big[\eta_{j,t}(\rho_{j,t}+\bm1(j\notin\mathfrak{I}_{k+1}))\big]$. Then, assuming that
\begin{equation}
	\label{eq:inass0ISO}
	\left| \langle \bm x, \mathcal{G}_{1,0} A_1 ... A_k \mathcal{G}_{k+1, 0}\bm y\rangle - \langle \bm x, \mathcal{M}_{[1,k+1],0} \bm y \rangle \right|\prec \left[\left(\prod_{i \in \mathfrak{I}_{k+1}} \rho_{i,0} \right) \wedge \max_{i \in [k+1]} \sqrt{\rho_{i,0}} \right]
		\frac{N^{k/2 }}{\sqrt{N \ell_{0}}} \prod_{j \in [k]} \langle |A_j|^2 \rangle^{1/2}  \, 
\end{equation}
holds for any $k\le K$, uniformly in $A$'s and in the spectral parameters as in part (a), and in deterministic vectors,  then
we have 
\begin{equation}
	\label{eq:flowgimgISO}
	\left| \langle \bm x, \mathcal{G}_{1,T} A_1 ... A_k \mathcal{G}_{k+1, T}\bm y\rangle - \langle \bm x, \mathcal{M}_{[1,k+1],T} \bm y \rangle \right|\prec \left[\left(\prod_{i \in \mathfrak{I}_{k+1}} \rho_{i,T} \right) \wedge \max_{i \in [k+1]} \sqrt{\rho_{i,T}} \right]	\frac{N^{k/2 }}{\sqrt{N \ell_{T}}} \prod_{j \in [k]} \langle |A_j|^2 \rangle^{1/2} \,,
\end{equation}
for any $k\le K$, again
uniformly  in $A$'s and in spectral parameters as in part (a), and in deterministic vectors ${\bm x}, {\bm y}$.

\end{itemize}

\end{proposition}

Proposition~\ref{prop:zig} is proven in Section~\ref{sec:opAedge}.
As the third and final step, we show that the additional Gaussian component introduced in Proposition~\ref{prop:zig} 
can be removed using a Green function comparison (GFT) argument. 
The proof of this proposition is presented in Section~\ref{sec:GFT}.

\begin{proposition}[Step 3: Green function comparison] \label{prop:zag} 
Let $H^{(\bm v)}$ and $H^{(\bm w)}$ be two $N \times N$ Wigner matrices with matrix elements given by the random variables $v_{ab}$ and $w_{ab}$, respectively, both satisfying Assumption \ref{ass:entries} and having matching moments up to third order,\footnote{This condition can easily be relaxed to being matching up to an error of size $N^{-2}$ as done, e.g., in \cite[Theorem~16.1]{EYbook}.} i.e.
\begin{equation} \label{eq:momentmatch}
	\E \bar{v}_{ab}^u v_{ab}^{s-u} = \E \bar{w}_{ab}^u w_{ab}^{s-u}\,, \quad s \in \{0,1,2,3\}\,, \quad u \in \{0,...,s\}\,. 
\end{equation}
Fix $K \in \N$ and consider spectral parameters $z_1, ... , z_{K+1} \in \C \setminus \R$ satisfying $ \min_j N \eta_j \rho_j  \ge N^{\epsilon}$ and $\max_j |z_j| \le N^{1/\epsilon}$ for some $\epsilon > 0$ and the associated resolvents $G_j^{(\#)} = G^{(\#)}(z_j) := (H^{(\#)}-z_j)^{-1}$ with $\mathcal{G}^{(\#)}_j \in \{ G^{(\#)}_j, \Im G^{(\#)}_j\}$ and $\# = \bm v, \bm w$. Pick   traceless matrices $A_1, ... , A_K \in \C^{N \times N}$. 

Assume that, for $H^{(\bm v)}$, 
we have the following bounds (writing $\mathcal{G}_j \equiv \mathcal{G}_j^{(\bm v)}$ for brevity). 
\begin{itemize}
	\item[(a)] For any $k \le K$, consider any subset of cardinality $k$ of the $K+1$ spectral parameters and, similarly, consider any subset of cardinality $k$ of the $K$ traceless matrices. Relabel both of them by $[k]$, and denote the set of indices $j \in [k]$ by $\mathfrak{I}_k$ where $\mathcal{G}_j = \Im G_j$. Setting $\ell := \min_{j \in [k]}\big[ \eta_j (\rho_j + \mathbf{1}(j \notin \mathfrak{I}_k))\big]$ we have that
	\begin{equation} \label{eq:zagmultiG}
		\left| \langle \mathcal{G}_1 A_1 ... \mathcal{G}_k A_k \rangle - \langle \mathcal{M}_{[1,k]}A_k \rangle \right| \prec \left[\left(\prod_{i \in \mathfrak{I}_k} \rho_i \right) \wedge \max_{i \in [k]} \sqrt{\rho_i}\right] \, 
		\frac{N^{k/2 - 1}}{\sqrt{N \ell}} \,  \prod_{j \in [k]} \langle |A_j|^2 \rangle^{1/2}\,,
	\end{equation}
	uniformly in all choices of subsets of $z$'s and $A$'s. 
			\item[(b)] For any $k \le K$, consider any subset of cardinality $k+1$ of the $K+1$ spectral parameters and, similarly, consider any subset of cardinality $k$ of the $K$ traceless matrices. Relabel them by $[k+1]$ and $[k]$, respectively, and denote the set of indices $j \in [k+1]$ by $\mathfrak{I}_{k+1}$ where $\mathcal{G}_j = \Im G_j$. Setting $\ell := \min_{j \in [k+1]}\big[ \eta_j (\rho_j + \mathbf{1}(j \notin \mathfrak{I}_{k+1}))\big]$ we have that
	\begin{equation} \label{eq:zagmultiGISO}
		\begin{split}
			\left| \langle \bm x, \mathcal{G}_1 A_1 ...  A_k \mathcal{G}_{k+1} \bm y \rangle - \langle \bm x,  \mathcal{M}_{[1,k+1]}\bm y \rangle \right| \prec \left[\left(\prod_{i \in \mathfrak{I}_{k+1}} \rho_i \right) \wedge \max_{i\in [k+1]} \sqrt{\rho_i}\right] \,
			\frac{N^{k/2}}{\sqrt{N \ell}} \, \prod_{j \in [k]} \langle |A_j|^2 \rangle^{1/2}  \,,
		\end{split}
	\end{equation}
	uniformly in all choices of subsets of $z$'s and $A$'s as in part (a) and in bounded deterministic vectors $\Vert \bm x \Vert \,, \Vert \bm y \Vert \lesssim 1$.
\end{itemize}

Then, \eqref{eq:zagmultiG}--\eqref{eq:zagmultiGISO} also hold for the ensemble $H^{(\bm{w})}$, uniformly all choices of subsets of $z$'s and $A$'s and in bounded deterministic vectors. 
\end{proposition}

We are now ready to finally conclude the proof of Theorem \ref{thm:main}. Fix $T>0$, and fix $z_1,\dots,z_{k+1}\in\C \setminus \R$ such that  $\min_i N\eta_i \rho_i  \ge N^\epsilon$, and let $z_{i,0}$ be the initial conditions of the characteristics \eqref{eq:chardef} chosen so that $z_{i,T}=z_i$ (this is possible thanks to Lemma~\ref{lem:propchar}).  Then, the assumption \eqref{eq:inass0} of Proposition~\ref{prop:zig} is satisfied  for those $z_{i,0}$ by Proposition~\ref{prop:initial} with $\delta=CT$, where $C>0$ is the constant from Lemma~\ref{lem:propchar}. We can thus use Proposition~\ref{prop:zig} to show that \eqref{eq:flowgimg} and \eqref{eq:flowgimgISO} hold. Finally, the Gaussian component added in Proposition~\ref{prop:zig} is removed using Proposition~\ref{prop:zag} with the aid of a complex version of the standard moment-matching lemma \cite[Lemma~16.2]{EYbook},
see Lemma \ref{lem:momentmatch} in Appendix \ref{app:moma} for more details. \qed

\nc

\section{Characteristic flow: Proof of Proposition~\ref{prop:zig}} \label{sec:opAedge}

In this section we present the proof of Proposition~\ref{prop:zig}. The argument is structured as follows:
\begin{itemize}
\item[(i)] In Section \ref{subsec:pureIM} we begin by proving the average part, Proposition \ref{prop:zig}~(a), in the case when $\mathcal{G}_{j,t}=\Im G_{j,t}$ for each $j\in [k]$, i.e., we prove \eqref{eq:flowgimg} for chains containing only $\Im G$'s. 
Along the flow \eqref{eq:OUOUOU} new resolvents without imaginary part arise, so the pure $\Im G$ structure
cannot be directly maintained. However, 
 we can use the integral representation (see, e.g. \cite[Eq.~(3.14)]{multiG}),
\begin{equation} \label{eq:intrepIM}
\prod_{j=1}^mG(z_j ) = \frac{1}{\pi} \int_\R \Im G (x + \ii \eta) \prod_{j=1}^m \frac{1}{x - z_j + \sgn(\Im z_j) \ii \eta} \dd x, 
\end{equation}
(that is
valid for any $0 < \eta < \min_j \Im z_j$ or $\max_j \Im z_j < - \eta < 0$)
to express each $G$  in terms of $\Im G$, thus the flow for purely $\Im G$ chains will be self-contained.

\item[(ii)] Afterwards, in the very short
Section \ref{subsec:pureIMISO}, we prove the isotropic part, Proposition~\ref{prop:zig}~(b) again
first in the case when $\mathcal{G}_{j,t}=\Im G_{j,t}$ for each $j\in [k+1]$. Due to the Hilbert-Schmidt error terms, 
the isotropic bound \eqref{eq:flowgimgISO} will directly follow from \eqref{eq:flowgimg} proven in 
Section \ref{subsec:pureIM}.
\item[(iii)] Finally, using the integral representation \eqref{eq:intrepIM}  in the special case $m=1$,
we derive the general case of mixed chains from the purely $\Im G$'s case in Section \ref{subsec:mixed}. 
\end{itemize}
Without loss of generality, to keep the presentation simpler, throughout  this section
we will assume that 
  $\sigma : =  \E \chi^2_{\mathrm{od}}$ is real  and $ \E \chi^2_{\mathrm{d}}=1+\sigma$
  (recall that $ \chi_{\mathrm{d}}, \chi_{\mathrm{od}}$ are the distribution of the diagonal
  and off-diagonal matrix elements of $W$, respectively). At the end, in Section~\ref{sec:general}, we will explain how to lift these  two restrictions. 

We recall our choice of the characteristics
\begin{equation} \label{eq:characrecall}
	\partial_t z_{i,t}=-m(z_{i,t})-\frac{z_{i,t}}{2}.
\end{equation}
Additionally, we record the following trivially checkable integration rules for any $\alpha \ge 1$:
\begin{equation}
\label{eq:newintrule}
 \int_0^t\frac{1}{\eta_{i,s}^\alpha}\,\dd s\lesssim \frac{\log N}{\eta_{i,t}^{\alpha-1}\rho_{i,t}} \qquad \text{and} \qquad  \int_0^t\frac{1}{\eta_{s}^\alpha}\,\dd s\lesssim \frac{\log N}{\eta_{t}^{\alpha-2}\hat{\ell}_{t}}  \quad \text{with} \quad \eta_t:=\min_i\eta_{i,t}\,, \quad   \hat{\ell}_t := \min_i \eta_{i,t} \rho_{i,t}\,. 
\end{equation}
Note that, in general, $\hat{\ell}$ differs from $\ell$, introduced in Theorem \ref{thm:main}. However, in case that every resolvent $\mathcal{G}$ in a given chain is an $\Im G$, i.e.~$\mathfrak{I}$ is the full set of indices, then it holds that $\hat{\ell} = \ell$. The notation 'hat' will  be consistently used to indicate that a chain contains only $\Im G$'s (see \eqref{eq:deflongchaingsim}--\eqref{eq:deflongchaingsimbar} below). 

Using the short--hand notation $G_{i,t}:=(W_t-z_{i,t})^{-1}$ with $W_t$ being the solution of \eqref{eq:OUOUOU}, 
we now compute the derivative (recall \eqref{eq:Mdefim})
\begin{equation} \label{eq:timederiv}
\dd \langle (\Im G_{1,t} A_1 ... \Im G_{k,t} -\mathcal{M}(z_{1,t}, A_1, ... , z_{k,t}; [k]) )A_k \rangle = ... 
\end{equation}
along the characteristics with the aid of It\^{o}'s formula.  We point out that the following derivation of the flow holds for any deterministic matrices $A_i$, i.e.  in this derivation we do not assume that $\langle A_i\rangle=0$. 
We will assume again that $A_i$ are traceless only later starting from the beginning of Section~\ref{subsec:pureIM}.\nc

The evolution for \eqref{eq:timederiv} (see \eqref{eq:flowkaim} below) is obtained by multilinearity from the analogous formula for the time derivative of a resolvent chain without any imaginary parts. So first we compute
	\begin{equation}
	\begin{split}
		\label{eq:flowka}
		\dd \langle (G_{[1,k],t}-M_{{[1,k]},t})A_k\rangle&=\frac{1}{\sqrt{N}}\sum_{a,b=1}^N \partial_{ab}  \langle {G}_{[1,k],t}A_k\rangle\dd B_{ab,t}+\frac{k}{2}\langle {G}_{[1,k],t}A_k\rangle\dd t +\sum_{i,j=1\atop i< j}^k\langle {G}_{[i,j],t}\rangle\langle {G}_{[j,i],t}\rangle\dd t \\
		&\quad+\sum_{i=1}^k \langle G_{i,t}-m_{i,t}\rangle \langle {G}^{(i)}_{[1,k],t}A_k\rangle\dd t -\partial_t \langle M_{[1,k],t}A_k\rangle \dd t +\frac{\sigma}{N}\sum_{i,j=1\atop i\le j}^k\langle G_{[i,j],t}G_{[j,i],t}^\mathfrak{t}\rangle\dd t \,,
	\end{split}
\end{equation}
 where $\partial_{ab}:=\partial_{w_{ab}}$ denotes the direction derivative in the direction
  $w_{ab}=w_{ab}(t):=(W_t)_{ab}$. Here we introduced  the notation
\begin{equation*}
	{{G}}_{[{i}, {j}],t}:=\begin{cases}
		G_{i,t}A_i\dots A_{j-1}G_{j,t} & \mathrm{if}\quad  i<j \\
		G_{i,t} & \mathrm{if}\quad i=j \\
		G_{i,t}A_{i,t}\dots \Im G_{k,t}A_k G_{1,t}A_1\dots A_{j-1} G_{j,t} &\mathrm{if}\quad i>j\,,
	\end{cases}
\end{equation*}
and analogously for the deterministic approximation $M_{[i,j],t}$ (cf.~\eqref{eq:Mdef}). Furthermore,
we define ${{G}}_{[{i}, {j}],t}^{(l)}$ exactly as ${{G}}_{[{i},{j}],t}$ but with the $l$--th factor $G_{l,t}$ being
 substituted by $G_{l,t}^2$. For the last term in \eqref{eq:flowka} we used the convention that 
$\langle G_{[i,j],t}^\mathfrak{t}G_{[j,i],t}\rangle=\langle G_{i,t}^\mathfrak{t}G_{i,t}A_{i+1}G_{[i+1,i],t}\rangle$ for $j=i$.

In order to write the derivative \eqref{eq:timederiv} in a manageable form, we need to introduce some further short--hand notations. Set
\begin{equation}
\label{eq:deflongchaingsim}
\widehat{{G}}_{[\hat{i}, \hat{j}],t}:=\begin{cases}
\Im G_{i,t}A_i\dots A_{j-1}\Im G_{j,t} & \mathrm{if}\quad  i<j \\
\Im G_{i,t} & \mathrm{if}\quad i=j \\
\Im G_{i,t}A_{i,t}\dots \Im G_{k,t}A_k\Im G_{1,t}A_1\dots A_{j-1}\Im G_{j,t} &\mathrm{if}\quad i>j,
\end{cases}
\end{equation}
and define $\widehat{{G}}_{[\hat{i}, \hat{j}],t}^{(l)}$  exactly as $\widehat{{G}}_{[\hat{i},\hat{j}],t}$ except  the $l$--th 
factor $\Im G_{l,t}$ is substituted with $G_{l,t}\Im G_{l,t}$. 
Similarly, $\widehat{{G}}_{[\hat{i}, \hat{j}],t}^{(l^*)}$ is defined as $\widehat{{G}}_{[\hat{i}, \hat{j}],t}$ but with the $l$--th 
 $\Im G_{l,t}$ is substituted by $G^*_{l,t}\Im G_{l,t}$. 
Furthermore, we also define 
\begin{equation}
\label{eq:deflongchaingsimbar}
\widehat{{G}}_{[\hat{i}, j],t}:=\begin{cases}
\Im G_{i,t}A_i\dots A_{j-1} G_{j,t} & \mathrm{if}\quad  i<j \\
G_{i,t} & \mathrm{if}\quad i=j \\
\Im G_{i,t}A_{i,t}\dots \Im G_{k,t}A_k\Im G_{1,t}A_1\dots A_{j-1} G_{j,t} &\mathrm{if}\quad i>j;
\end{cases}
\end{equation}
note the absent hat  on the $j$ index
indicates that the last resolvent $G_{j,t}$ is without imaginary part.
We also define $\widehat{G}_{[i^*, \hat{j}],t}$ replacing $\Im G_{i,t}$ with $G_{i,t}^*$ in \eqref{eq:deflongchaingsim}
and similarly
 $\widehat{G}_{[i^*,j],t}$ is defined by replacing $\Im G_{i,t}$ with $G_{i,t}^*$ and $\Im G_{j,t}$ with $G_{j,t}$ in \eqref{eq:deflongchaingsim}.
  In particular, the  `decorations'  of $i$ and $j$ indicate, whether $G_{i,t}$ and $G_{j,t}$ are
   really taken as plain resolvents (no decoration) or as  adjoints (star)  or with imaginary part (hat). 
   We point out that throughout this entire section  'hat' on $G$ indicates that the chain contains only $\Im G_i$ 
   unless specified as in \eqref{eq:deflongchaingsimbar}. Finally, we use  similar notations for the 
   corresponding deterministic approximation $\widehat{M}_{[i^\#,j^\#],t}$ whose 'undecorated' version was
    defined in \eqref{eq:Mdef}. 
 Here $\#$ indicates one of the possible `decorations', i.e. star, hat or no decoration
 and the corresponding change entails modifying the factor $(x-z_i)^{-1}$ in \eqref{msc dd}
to $(x-\bar z_i)^{-1}$ in case of star, and to $\Im (x-\bar z_i)^{-1}$ in case of hat 
(as in \eqref{eq:Mdefim}--\eqref{eq:Mdivdiff}).

 The time derivative of the deterministic term in \eqref{eq:timederiv} is obtained by the following lemma, the proof of which is given in Appendix \ref{sec:addtech}. 
 \begin{lemma} \label{lem:Mcancel}
 	For any $k\ge 1$ we have
 	\begin{align}\label{partialt}
 		\partial_t \langle \widehat{M}_{[\hat{1},\hat{k}],t}A_{k}\rangle=\frac{k}{2}\langle \widehat{M}_{[\hat{1},\hat{k}],t}A_{k}\rangle+\sum_{i,j=1\atop i< j}^k\langle \widehat{M}_{[\hat{i}, j],t}\rangle\langle \widehat{{M}}_{[\hat{j}, i],t}\rangle+\sum_{i,j=1\atop i< j}^k\langle \widehat{M}_{[i^*, \hat{j}],t}\rangle\langle \widehat{{M}}_{[j^*, \hat{i}],t}\rangle \\ 
 		+\sum_{i,j=1\atop i< j}^k\langle \widehat{M}_{[\hat{i}, \hat{j}],t}\rangle\langle \widehat{{M}}_{[j^*, i],t}\rangle+\sum_{i,j=1\atop i< j}^k\langle \widehat{M}_{[i^*,j],t}\rangle\langle \widehat{{M}}_{[\hat{j}, \hat{i}],t}\rangle.
 		\nonumber
 	\end{align}
 \end{lemma}

Hence, by It\^{o}'s formula, for any $k\ge 1$, the evolution of $\widehat{G}_{[\hat{1},\hat{k}],t}$ is given by (for brevity we omit the $\dd t$ differentials)  
\begin{equation}
	\begin{split}
		\label{eq:flowkaim}
		&\dd \langle (\widehat{{G}}_{[\hat{1}, \hat{k}],t}-\widehat{M}_{[\hat{1},\hat{k}],t})A_k\rangle \\
		&=\frac{1}{\sqrt{N}}\sum_{a,b=1}^N \partial_{ab}  \langle \widehat{{G}}_{[\hat{1}, \hat{k}],t}A_k\rangle\dd B_{ab}+\frac{k}{2}\langle (\widehat{{G}}_{[\hat{1}, \hat{k}],t}-\widehat{M}_{[\hat{1}, \hat{k}],t})A_k\rangle
		+\Omega_1 + \Omega_2 + \Omega_3 +\Omega_4+ \Omega_\sigma \\
		&\quad+\sum_{i=1}^k \langle G_{i,t}-m_{i,t}\rangle \langle \widehat{{G}}^{(i)}_{[\hat{1},\hat{k}],t}A_k\rangle+\sum_{i=1}^k \langle G_{i,t}^*-\overline{m_{i,t}}\rangle \langle \widehat{{G}}_{[\hat{1}, \hat{k}],t}^{(i^*)}A_k\rangle+
		 \langle \widehat{{G}}_{[\hat{1}, \hat{k}],t}A_k\rangle\sum_{i=1}^k \frac{\langle \Im G_{i,t}-\Im m_{i,t}\rangle}{\Im z_{i,t}}\, ,	\end{split}
\end{equation}
where 
\begin{equation}
	\small
	\begin{split}
 \Omega_1: & =  \sum_{i,j=1\atop i< j}^k\left[\langle \widehat{{G}}_{[\hat{i}, j],t}-\widehat{M}_{[\hat{i},j],t}\rangle\langle \widehat{{M}}_{[\hat{j},i],t}\rangle+\langle \widehat{{M}}_{[\hat{i}, j],t}\rangle\langle \widehat{{G}}_{[\hat{j},i],t}-\widehat{M}_{[\hat{j}, i],t}\rangle+  \langle \widehat{{G}}_{[\hat{i}, j],t}-\widehat{M}_{[\hat{i},j],t}\rangle\langle \widehat{{G}}_{[\hat{j},i],t}-\widehat{M}_{[\hat{j}, i],t}\rangle \right],\\
 \Omega_2: &= \sum_{i,j=1\atop i< j}^k\left[\langle \widehat{{G}}_{[i^*, \hat{j}],t}-\widehat{M}_{[i^*, \hat{j}],t}\rangle\langle \widehat{{M}}_{[j^*,\hat{i}],t}\rangle+\langle \widehat{{M}}_{[i^*, \hat{j}],t}\rangle\langle \widehat{{G}}_{[j^*,\hat{i}],t}-\widehat{M}_{[j^*,\hat{i}],t}\rangle+ \langle \widehat{{G}}_{[i^*, \hat{j}],t}-\widehat{M}_{[i^*, \hat{j}],t}\rangle\langle \widehat{{G}}_{[j^*,\hat{i}],t}-\widehat{M}_{[j^*,\hat{i}],t}\rangle\right], \\
 \Omega_3:&= \sum_{i,j=1\atop i< j}^k\left[\langle \widehat{{G}}_{[\hat{i}, \hat{j}],t}-\widehat{M}_{[\hat{i},\hat{j}],t}\rangle\langle \widehat{{M}}_{[j^*,i],t}\rangle+\langle \widehat{{M}}_{[\hat{i}, \hat{j}],t}\rangle\langle \widehat{{G}}_{[j^*,i],t}-\widehat{M}_{[j^*, i],t}\rangle+ \langle \widehat{{G}}_{[\hat{i}, \hat{j}],t}-\widehat{M}_{[\hat{i},\hat{j}],t}\rangle\langle \widehat{{G}}_{[j^*,i],t}-\widehat{M}_{[j^*, i],t}\rangle\right],\\
 \Omega_4:&=\sum_{i,j=1\atop i< j}^k\left[\langle \widehat{{G}}_{[i^*, j],t}-\widehat{M}_{[i^*, j],t}\rangle\langle \widehat{{M}}_{[\hat{j},\hat{i}],t}\rangle+\langle \widehat{{M}}_{[i^*, j],t}\rangle\langle \widehat{{G}}_{[\hat{j},\hat{i}],t}-\widehat{M}_{[\hat{j},\hat{i}],t}\rangle+ \langle \widehat{{G}}_{[i^*, j],t}-\widehat{M}_{[i^*, j],t}\rangle\langle \widehat{{G}}_{[\hat{j},\hat{i}],t}-\widehat{M}_{[\hat{j},\hat{i}],t}\right], \\
 \Omega_\sigma:&=\frac{\sigma}{N}\sum_{i,j=1\atop i\le j}^k\left[\langle G_{[\widehat{i},j],t}G_{[\widehat{j},i],t}^\mathfrak{t}\rangle+\langle G_{[i^*,\widehat{j}],t}G_{[j^*,\widehat{i}],t}^\mathfrak{t}\rangle+\langle G_{[\widehat{i},\widehat{j}],t}G_{[j^*,i],t}^\mathfrak{t}\rangle+\langle G_{[i^*,j],t}G_{[\widehat{j},\widehat{i}],t}^\mathfrak{t}\rangle\right]\,.
 \end{split}
\end{equation}

Observe that the flow \eqref{eq:flowkaim} for imaginary parts $\Im G$ contains much more terms compared to a flow for plain resolvents $G$ (see \eqref{eq:flowka}). This is a simple consequence of the fact that each time an $\Im G$ is differentiated it creates two terms, i.e. $\partial_{ab} \Im G=G\Delta^{ab}\Im G+\Im G\Delta^{ab}G^*$, with $\Delta^{ab}$ being a matrix consisting of all zeroes except for the $(a,b)$--entry which is equal to one. 
Furthermore, the novel last term in \eqref{eq:flowkaim} comes from applying a Ward identity, $GG^*= \Im G/\Im z$.
We now write out the random part $\dd \langle \widehat{{G}}_{[\hat{1}, \hat{k}],t}A_k\rangle$ of
the flow \eqref{eq:flowkaim} for the simpler cases $k=1$ and $k=2$  to show its main structure.
Here we used that $\widehat{M}_{\hat{1},t} = \Im m_{1,t}$ with $m_{i}: =m(z_{i,t})$.

\begin{example}
	For $k=1$ we have the evolution
	\begin{equation}
	\begin{split}
		\label{eq:eqk1im}
		\dd \langle \Im G A\rangle&=\sum_{a,b=1}^N\partial_{ab} \langle \Im G A\rangle\frac{\dd B_{ab}}{\sqrt{N}}+\left(\frac{1}{2}+\frac{\langle \Im G-\Im m\rangle}{\Im z_t}\right)\langle \Im G A\rangle+\langle G-m\rangle \langle \Im G A G\rangle \\
		&\quad+\overline{\langle G-m\rangle} \langle \Im G A G^*\rangle+\frac{\sigma}{N}\langle \Im G A GG^\mathfrak{t}\rangle+\frac{\sigma}{N}\langle (G^*)^\mathfrak{t}G^*A \Im G A \rangle+\frac{\sigma}{N}\langle \Im G^\mathfrak{t} G^*A G\rangle\,, 
		\end{split}
	\end{equation}
	and for $k=2$ we get (for keeping the formula somewhat short,  we assume $\sigma=0$) 
	\begin{equation}
		\label{eq:imgflowsdetsub}
		\begin{split}
			\dd \langle \Im G_1 A_1 \Im G_2 A_2\rangle &=\sum_{a,b=1}^N\partial_{ab} \langle \Im G_1A_1 \Im G_2 A_2\rangle\frac{\dd B_{ab}}{\sqrt{N}}+\langle \Im G_1 A_1\Im G_2 A_2\rangle \\
			&\quad+\left(\frac{\langle \Im G_1 -\Im m_1\rangle}{\Im z_{1,t}}+\frac{\langle \Im G_2 -\Im m_2\rangle}{\Im z_{2,t}}\right)\langle \Im G_1 A_1\Im G_2 A_2\rangle+\langle G_2^*A_2G_1\rangle \langle \Im G _1A_1\Im G_2\rangle \\
			&\quad+\langle G_1^* A_1 G_2 \rangle\langle \Im G_2 A_2 \Im G_1\rangle+\langle \Im G_1 A_1 G_2\rangle\langle\Im G_2A_2G_1\rangle +\langle G_2^*A_2\Im G_1\rangle\langle G_1^*A_1\Im G_2\rangle \\
			&\quad+\langle G_1-m_1\rangle \langle \Im G_1 A_1\Im G_2 A_2 G_1\rangle +\langle G_2-m_2\rangle \langle \Im G_2 A_2\Im G_1 A_1 G_2\rangle \\
			&\quad+\langle G_1^*-\overline{m_1}\rangle \langle \Im G_1 A_1\Im G_2 A_2 G_1^*\rangle+\langle G_2^*-\overline{m_2}\rangle \langle \Im G_2 A_2\Im G_1 A_1 G_2^*\rangle.
		\end{split}
	\end{equation}
	Note that \eqref{eq:eqk1im}--\eqref{eq:imgflowsdetsub} combined with \eqref{partialt} give \eqref{eq:flowkaim} for the special cases $k=1,2$.
\end{example}

 \subsection{Proof of Proposition \ref{prop:zig}~(a) for pure $\Im G$-chains} \label{subsec:pureIM}

The goal of this section is to prove 
\begin{equation}
	\label{eq:goaloptaIms}
	\langle \widehat{G}_{[\widehat{1},\widehat{k}],T}A_k \rangle - \langle \widehat{M}_{[\widehat{1},\widehat{k}],T}A_k \rangle =	\langle \widehat{G}_{[\widehat{1},\widehat{k}],0}A_k \rangle - \langle \widehat{M}_{[\widehat{1},\widehat{k}],0}A_k \rangle+
\mathcal{O}_\prec \left(\Big(\prod_{i\in [k]} \rho_{i,T}\Big)\frac{N^{k/2 - 1}}{\sqrt{N \ell_{T}}} \prod_{j \in [k]} \langle |A_j|^2 \rangle^{1/2}\right),
\end{equation}
uniformly in the spectrum and in the choice of traceless matrices $A_i$. We may assume that all the $A_i$'s are Hermitian; the general case follows by multilinearity.

\subsubsection{Master inequalities}  \label{subsec:master}
  For the purpose of proving \eqref{eq:goaloptaIms}, recall the notation $\hat{\ell}_t = \min \eta_{i,t} \rho_{i,t}$ from \eqref{eq:newintrule} and define 
\begin{subequations} 	\label{eq:defphi}
\begin{equation}
\Phi_1(t) :=\frac{N \sqrt{\hat{\ell}_t}}{\rho_t \langle |A|^2 \rangle^{1/2}}\big|\langle G_tA\rangle\big|;
\end{equation}
and for  $k\ge 2$
\begin{equation}
	\Phi_k(t):= \frac{\sqrt{N \hat{\ell}_{t}}}{N^{k/2 - 1} \, \Big(\prod_{i\in [k]} \rho_{i,t}\Big) \prod_{j \in [k]} \langle |A_j|^2 \rangle^{1/2}} \big| \langle (\widehat{G}_{[\widehat{1},\widehat{k}],t}- \widehat{M}_{[\widehat{1},\widehat{k}],t})A_k \rangle\big|\,. 
\end{equation}
\end{subequations}
Note that we defined $\Phi_1(t)$ in a slightly different way than $\Phi_k(t)$ for $k\ge 2$, this is a consequence of the fact that for $k=1$ we have $| \langle GA\rangle |\sim |\langle \Im GA\rangle|$, i.e. for this special case the imaginary part does not reduce the fluctuation unlike for 
 longer chains (see also Remark \ref{rmk:MHS}~(ii)).
The prefactors in \eqref{eq:defphi} are chosen such that we expect $\Phi_k(t)$ to be an essentially order one quantity, see~\eqref{eq:goaloptaIms}. The goal  is to show exactly this, i.e.  that $\Phi_k(t)\prec 1$, uniformly in time $t\le T$
for any $k\ge 1$. Note that by \eqref{eq:inass0} it follows
 \begin{equation}
 \label{eq:initialphi}
 \Phi_k(0)\prec 1,
 \end{equation}
 for any $k\ge 1$. 
  
To prove $\Phi_k(t)\prec 1$, we will derive
 a series of {\it master inequalities} for these quantities with the following structure.
We assume that
\begin{equation} \label{phiphi}
 \Phi_k(t)\prec\phi_k
\end{equation}
 holds for some deterministic control parameter $\phi_k$, \emph{uniformly} in $0\le t\le T$, in 
 spectral parameters satisfying  $N\hell_t\ge N^\epsilon$ and in traceless deterministic matrices $A_j$ (we stress that  $\phi_k$'s depend neither on time, nor on the spectral parameters $z_{i,t}$, nor on the matrices $A_j$). Given this input, we will then show that  $\Phi_k(t)$'s also satisfy a better upper bound in terms of $\phi$'s. 
 Iterating this procedure we will arrive at the final bound $\Phi_k(t) \prec 1$. Furthermore, without loss of generality, we assume that $\phi_k\ge 1$. \nc
 
 \begin{proposition}[Master inequalities]
 \label{pro:masterinIM}
Fix $k\in \N$. Assume that $\Phi_l(t)\prec \phi_l$ for any $1\le l\le k+1 \nc$ uniformly in $t\in [0,T]$,
in spectral parameters with $N\hell_t\ge N^\epsilon$ and in traceless deterministic matrices $A_j$.
Set $\phi_0 :=1$. Then we have the \emph{master inequalities}
\begin{equation}
	\begin{split}
		\label{eq:AVmasterITERATE}
		\Phi_k(t)\prec 1+\frac{1}{(N\hat{\ell}_T)^{1/4}}\sum_{l=1}^k \tilde{\phi}_l\tilde{\phi}_{k-l} \,,
	\end{split}
\end{equation}
uniformly (in the sense explained below \eqref{phiphi}\nc) 
in $t\in [0,T]$, 
where we introduced the shorthand notation
\begin{equation} \label{tildephi}
\tilde{\phi}_l := \phi_l + \mathbf{1}(l \ \mathrm{is \;\; odd}) \sqrt{\phi_{l+1} \phi_{l-1}}\,. 
\end{equation}
\end{proposition}



Using the master inequalities, we conclude this section with the proof of \eqref{eq:flowgimg} for pure $\Im G$ chains.

\begin{proof}[Proof of Proposition~\ref{prop:zig}~(a) for pure $\Im G$ chains]
	
	We now consider the master inequalities \eqref{eq:AVmasterITERATE} for $t=T$, with $T$ the time defined in the statement of Proposition~\ref{prop:zig}. 
	
We use  a two-step induction. The base case consists of the cases $k=1,2$: 
	\begin{equation}
			\label{eq:mast12phi}
			\Phi_1(t) \prec 1+\frac{\phi_1+\sqrt{\phi_2}}{(N\hat{\ell}_T)^{1/4}}, \qquad\qquad\quad\Phi_2(t)\prec  1+\frac{\phi_1^2+\phi_2}{(N\hat{\ell}_T)^{1/4}},
	\end{equation}
	uniformly in $t\in [0, T]$. \nc
	
The following abstract iteration lemma shows how to use the master inequalities
for improving the bound on $\Phi$.

	 \begin{lemma}[Iteration]
		\label{lem:iteration} Let $X=X_N(\hell)$ be an $N$-dependent random variable depending 
		also on the parameter $\hell$.
		Fix $\epsilon,\delta>0$. Suppose that for any $l\in \N$  and any $x>0$ 
		 the fact that $X\prec x$ uniformly for 
		$\hat{\ell}\ge N^{-1+l\epsilon}$
		 implies
		\begin{equation}
			\label{eq:iterationrep}
			X\prec A+\frac{x}{B}+x^{1-\alpha}C^\alpha,
		\end{equation}
		uniformly for $\hell\ge N^{-1+(l+l')\epsilon}$, for some constants $l'\in \N$, $B\ge N^\delta>0$, $A,C>0$, and $\alpha\in (0,1)$, and suppose we also know that $X\prec N^D$ uniformly\footnote{We remark that $D,\delta,\alpha$ are $N$--independent constants, all the other quantities may depend on $N$.} 
		 in $\hell\ge N^{-1+\epsilon}$. Then
		\[
		X\prec A+C,
		\]
		uniformly for $\hell\ge N^{-1+(1+ \kappa  l')\epsilon}$, for some $\kappa=\kappa(\alpha,D,\delta)$. 
	\end{lemma}
	\begin{proof} The proof is a simple iteration of \eqref{eq:iterationrep} $\kappa$ times; it is 
	immediate to see that $\kappa$ depends only on $\alpha,D,\delta$.
		\end{proof} 
	 Notice that 
	using Lemma~\ref{lem:iteration} reduces the domain of parameters $\eta_i,\rho_i$ for which the master inequalities \eqref{eq:AVmasterITERATE} hold, e.g. from $\hell_{t}\ge N^{-1+l\epsilon}$ to
	$\hell_{t}\ge N^{-1+(l+l')\epsilon}$, and so on. However, this can happen only finitely many times,
	 and so it does not affect the estimates in the sense of
	stochastic domination that always allows for a small $N$-power tolerance
	that can be achieved by adjusting $\epsilon$ small enough.
	 For simplicity, we ignore
	this subtlety here,  
	 see \cite[Sections~4.1--4.3]{iidpaper} for a more detailed explanation.

	Using iteration from Lemma~\ref{lem:iteration} we obtain
		\begin{equation*}
			\Phi_1(t \nc)\prec 1+\frac{\sqrt{\phi_2}}{(N\hell_T)^{1/4}} \quad \text{and} \quad
			\Phi_2( t\nc)\prec  1+ \frac{\phi_1^2}{(N\hell_T)^{1/4}}
	\end{equation*}
	uniformly in $t\in [0,T]$. 
	  We can thus summarize what we proved so far as follows (all the following statements hold uniformly in $t\in [0,T]$):
	\begin{equation}
	\label{eq:impli}
	\begin{cases}
	\Phi_1(t)\prec \phi_1 \\
	\Phi_2(t)\prec \phi_2 
	\end{cases} \quad\Longrightarrow \quad \begin{cases} \Phi_1(t)\prec 1+\frac{\sqrt{\phi_2}}{(N\hell_T)^{1/4}} \\
	\Phi_2(t)\prec  1+ \frac{\phi_1^2}{(N\hell_T)^{1/4}} \\
	\end{cases}
	\end{equation}
	for any deterministic control parameters $\phi_1,\phi_2$. 
	
	Next, we use \eqref{eq:impli} again replacing the control parameters $\phi_1, \phi_2$ in the input with the new
	\begin{equation}
	\label{eq:tildephi}
	\phi'_1:=1+\frac{(\phi_2)^{1/2}}{(N\hell_T)^{1/4}}, \qquad\quad \phi'_2:=1+ \frac{\phi_1^2}{(N\hell_T)^{1/4}}.
	\end{equation}
	We thus obtain
	\begin{equation}
	\label{eq:conclh}
	\begin{cases}
	\Phi_1(t)\prec 1+\frac{(\phi'_2)^{1/2}}{(N\hell_T)^{1/4}}\lesssim 1+\frac{\phi_1}{(N\hell_T)^{3/8}} \\[1mm]
	\Phi_2(t)\prec  1+ \frac{(\phi'_1)^2}{(N\hell_T)^{1/4}}\lesssim 1+\frac{\phi_2}{(N\hell_T)^{3/4}}
	\end{cases}
	\end{equation}
	where in the second inequalities we used the definitions \eqref{eq:tildephi}.

Finally, using that $\Phi_i(t)\prec \phi_i$ by assumption and applying Lemma~\ref{lem:iteration} once again, we obtain \nc
	\begin{equation*}
		\Phi_1(t\nc )\prec 1 \quad \text{and} \quad \Phi_2(t \nc)\prec 1\,.
	\end{equation*}
	uniformly in $t\in [0,T]$. \nc
	To prove the same relation for $\Phi_l(t \nc)$ with $l\ge 3$,  we use  a step-two induction.  Fix 
	an even $k\ge 4$	and 
	 assume as our induction hypothesis that  $\Phi_l(t \nc)\prec 1$ 
	for any $1\le l\le k-2$, uniformly in $t\in [0,T]$. We now prove that $\Phi_l(t\nc )\prec 1$ also holds for $l=k-1, k$,
	uniformly in $t\in [0,T]$.
	From \eqref{eq:AVmasterITERATE}, using the induction hypothesis 
	$\Phi_l(t\nc )\prec \phi_l:=1$ for $1\le l\le k-2$, we have
	\begin{equation*}
			\Phi_{k-1}(t)\prec 1+\frac{\phi_{k-1}+\sqrt{\phi_k}}{(N\hell_T)^{1/4}}\,,  \qquad\qquad\quad \Phi_k(t)\prec 1+\frac{\phi_k+\phi_{k-1}+\sqrt{\phi_k}}{(N\hell_T)^{1/4}}  \,
	\end{equation*}
 uniformly in $t\in [0, T]$. \nc
	Then using  iteration from Lemma \ref{lem:iteration}, we obtain
	\begin{equation*}
	\Phi_{k-1}(t \nc)\prec 1+\frac{\sqrt{\phi_{k}}}{(N\hell_T)^{1/4}}   \quad \text{and} \quad 
	\Phi_k( t\nc )\prec 1+ \frac{ \phi_{k-1} }{(N\hell_T)^{1/4}} \,
\end{equation*}
  uniformly in $t\in [0, T]$. \nc
Proceeding similarly to \eqref{eq:impli}--\eqref{eq:conclh}, we thus obtain \nc
	\begin{equation*}
		\Phi_{k-1}(t\nc)\prec 1 \qquad \text{and} \qquad \Phi_k( t \nc)\prec 1\,.
	\end{equation*}
This concludes the induction step and hence, by setting $t=T$, \nc the proof of Proposition \ref{prop:zig}~(a)
 modulo the proof of Proposition \ref{pro:masterinIM}, which will be done next. 
\end{proof}

 \subsubsection{Proof of Proposition \ref{pro:masterinIM}} \label{subsec:pfmasterAV}
 As a preparation for the proof of the master inequalities  (Proposition~\ref{pro:masterinIM}), we recall that $t \mapsto \eta_{i,t}$ is decreasing and $\rho_{i,s} \sim \rho_{i,t}$ for any $0 \le s \le t\lesssim1$ (see \eqref{eq:characteristics}, \eqref{eq:characrecall}, and the paragraphs around).

\begin{proof}[Proof of Proposition~\ref{pro:masterinIM}] We begin with the case $k=1$. Hence, for $A_1=A$, we start by rewriting the flow  \eqref{eq:eqk1im} with $\Im G$ replaced by $G= G_t(z_t)$ (recall \eqref{eq:defphi}):
	\begin{equation} \label{eq:eq1}
		\dd \langle GA\rangle=\sum_{a,b=1}^N\partial_{ab}\langle GA\rangle\frac{\dd B_{ab}}{\sqrt{N}}+\frac{1}{2}\langle GA\rangle \dd t+\langle G-m\rangle\langle G^2A\rangle \dd t+\frac{\sigma}{N}\langle GAGG^\mathfrak{t}\rangle\dd t \,.
	\end{equation}
We point out that the additional term $\frac{1}{2}\langle GA\rangle$ in
the rhs.~of \eqref{eq:eq1} can be incorporated  into the lhs.
by differentiating $e^{-t/2}\langle GA\rangle$; the extra exponential factor is irrelevant since $e^{t/2}\sim 1$
for our times $t\lesssim 1$. Note that the same argument applies to the term
\begin{equation*}
\frac{k}{2}\langle (\widehat{{G}}_{[\hat{1}, \hat{k}],t}-\widehat{M}_{[\hat{1}, \hat{k}],t})A_k\rangle
\end{equation*}
appearing in \eqref{eq:flowkaim} for general $k$. We are now ready to obtain the master inequality for $\Phi_1(t)$.

Assume $\Phi_k(t)\prec \phi_k$ for $k=1,2$, in the sense of uniformity
explained after \eqref{phiphi} (recall that $\Phi_1(0)\prec 1$ by \eqref{eq:initialphi}), and we will prove  improved bounds on $\Phi_1(t)$. We first consider the third
summand in \eqref{eq:eq1}. Here, we use the integral representation (see also \cite[Lemma~5.1]{iidpaper})
\begin{equation} \label{eq:niceintrep}
G^2(z) = \frac{1}{2 \pi \ii} \oint_\Gamma \frac{G(w)}{(w-z)^2} \dd w\,,
\end{equation}
which simply follows from residue calculus. Here, $\Gamma$ is a tiny circle of radius $|\Im z|/2$ around $z \in \C \setminus \R$, which ensures that $|\Im w| |\Im m(w)| \sim |\Im z| |\Im m(z)|$ as follows by elementary continuity properties of $m(w)$. In this way, applying \eqref{eq:niceintrep} for every fixed time $s \le t$ and using the fact that the deterministic approximation of $\langle G^2 A \rangle$ vanishes as $\langle A\rangle =0$, we obtain (with the $G_s:= G_s(z_s)$,
$m_s:=m(z_s)$ notation) 
$$
\big|\langle G_s^2 A \rangle\big| \prec \frac{1}{\eta_s} \frac{\rho_s \langle|A|^2 \rangle^{1/2}}{N \sqrt{\hat{\ell}_s}} \phi_1\,. 
$$
Hence, in combination with the single resolvent local law $|\langle G_s-m_s\rangle|\prec 1/(N\eta_s)$, we find
\begin{equation}
	\label{psi11}
	\frac{N\sqrt{\hell_t}}{\rho_t \langle|A|^2 \rangle^{1/2}} \int_0^t \langle G_s-m_s\rangle
	 \langle G^2_s A \rangle\, \dd s\prec 		
	\frac{N\sqrt{\hell_t}}{\rho_t \langle|A|^2 \rangle^{1/2}} \int_0^t  \phi_1\; \frac{\rho_s\langle|A|^2 \rangle^{1/2}}{N^2\eta_s^2\hell_s^{1/2}}\,\dd s\lesssim \frac{\phi_1}{N\hell_t}\log N.
\end{equation}
In the last step we used  the integration estimate \eqref{eq:newintrule} and the fact  that along the characteristics $\hell_s\gtrsim \hell_t$ for $0\le s\le t$.
The prefactor $N \sqrt{\hell_t}/(\rho_t \langle|A|^2 \rangle^{1/2})$ is included in anticipation 
of the same prefactor in the definition of $\Phi_1$ in \eqref{eq:defphi}.

Then we proceed with the estimate of the quadratic variation of the martingale term in \eqref{eq:eq1}:
\[
\begin{split}
\frac{1}{N}\sum_{a,b=1}^N \big[\big|\partial_{ab}\langle G_sA\rangle|^2+\sigma\partial_{ab}\langle G_sA\rangle\overline{\partial_{ba}\langle G_sA\rangle}\big]\dd t&\lesssim\frac{1}{N^3}\sum_{a,b=1}^N \big|(G_sAG_s)_{ab}|^2\dd t \\
&= \frac{1}{N^2}\langle G_s A G_sG_s^* A G_s^*\rangle\dd t= \frac{1}{N^2\eta_t^2}\langle \Im G_s A \Im G_s A\rangle\dd t,
\end{split}
\]
where we used that $\dd [B_{ab},\overline{B_{cd}}]=\delta_{ac}\delta_{bd}+ \sigma\delta_{ad}\delta_{bc}$ 
and the Ward identity $GG^*=\frac{\Im G}{\Im z}$. Then, we write
\[
\langle \Im G_s A \Im G_s A\rangle 
 =\langle \widehat{M}_{[\widehat{1},\widehat{2}],s}A\rangle
 + \Big( \langle \Im G_s A \Im G_s A\rangle -\langle \widehat{M}_{[\widehat{1},\widehat{2}],s}A\rangle \Big)
  \prec \rho_s^2 \langle |A|^2 \rangle+\frac{\rho_s^2 \langle |A|^2 \rangle}{\sqrt{N\hell_s}} \phi_2\,.
\]
Here we used  that the deterministic approximation $\langle \widehat{M}_{[\widehat{1},\widehat{2}],s}A\rangle$
is bounded by $\rho_s^2 \langle |A|^2 \rangle$ and we used 
 \eqref{eq:defphi} together with  $\Phi_2(s)\prec \phi_2$. 
For the time integration of the quadratic variation term, with the appropriate prefactor, we obtain 
\begin{equation}
	\label{eq:qv1}
	\begin{split}
	\frac{N\sqrt{\hell_t}}{\rho_t \langle |A|^2 \rangle^{1/2}} & \left(\int_0^t  \frac{\langle \Im G_s A \Im G_s A\rangle}{N^2\eta_s^2}  \,\dd s\right)^{1/2} \\
	&\prec \frac{N\sqrt{\hell_t}}{\rho_t \langle |A|^2\rangle^{1/2}}\left(\int_0^t  \frac{\rho_s^2 \langle |A|^2 \rangle}{N^2\eta_s^2}\left(1 +\frac{\phi_2}{(N\hell_s)^{1/2}}\right)  \,\dd s\right)^{1/2}\lesssim 1
	+\frac{\sqrt{\phi_2}}{(N\hell_t)^{1/4}}\,.
	\end{split}
\end{equation}
Here in the last inequality we used that along the characteristics $\hell_s\gtrsim \hell_t$ 
for $0\le s\le t$ and the integration rule~\eqref{eq:newintrule}.
Using the Burkholder-Davis-Gundy (BDG) inequality we conclude 
exactly the same estimate \eqref{eq:qv1} for the stochastic term in \eqref{eq:eq1}
in high probability as in quadratic variation.

Next, we estimate the last term in the rhs. of \eqref{eq:eq1}:
\begin{equation}
\begin{split}
\label{eq:newk=1}
\frac{N\sqrt{\hell_t}}{\rho_t\langle |A|^2\rangle^{1/2}}\int_0^t \frac{|\sigma|}{N}\big|\langle G_sAG_sG_s^\mathfrak{t}\rangle\big|\,\dd s&\le \frac{N\sqrt{\hell_t}}{\rho_t\langle |A|^2\rangle^{1/2}}\int_0^t \frac{1}{N\eta_s^{3/2}}\langle \Im G_sA\Im G_s A\rangle^{1/2} \langle \Im G_s \rangle^{1/2}\,\dd s \\
&\prec \frac{N\sqrt{\hell_t}}{\rho_t\langle |A|^2\rangle^{1/2}}\int_0^t \frac{\rho^{1/2}_s}{N\eta_s^{3/2}}\left(\langle |A|^2\rangle \rho_s^2+\frac{\langle |A|^2\rangle \rho_s^2\phi_2}{\sqrt{N\hell_s}}\right)^{1/2}\,\dd s \\
&\lesssim 1+\frac{\sqrt{\phi_2}}{(N\hell_t)^{1/4}},
\end{split}
\end{equation}
where in the first inequality we used Schwarz inequality together with several  Ward identities, and in the second inequality the single resolvent local law $|\langle G_s-m_s\rangle|\prec 1/(N\eta_s)$ to show that $\langle\Im G_s\rangle\prec \rho_s$ (recall that we consider the regime $N\eta_s\rho_s\ge N^\epsilon$, so $1/(N\eta_s)\le \rho_s$).

Putting all these estimates together, and using that $\Phi_1(0)\prec 1$ by \eqref{eq:initialphi} to bound 
the initial condition after integration, we obtain the first master inequality
\begin{equation}
	\Phi_1(t)\prec 1+\frac{\phi_1}{N\hell_t}+\frac{\sqrt{\phi_2}}{(N\hell_t)^{1/4}},
\end{equation}
again in the sense of uniformity explained after \eqref{phiphi}. This gives \eqref{eq:AVmasterITERATE} using $N\widehat{\ell}_t\ge 1$ and $\hell_T\lesssim \hell_t$\footnote{This follows from the fact that by \eqref{eq:chardef} we have $\rho_{i,s}\sim \rho_{i,t}$ and $\eta_{i,t}\lesssim \eta_{i,s}$, for $s\le t$.}. \nc

 For the proof of the master inequalities~\eqref{eq:AVmasterITERATE} 
 with $k\ge 2$, a fundamental input for the estimates of the various terms in \eqref{eq:flowkaim} is
  the following  \emph{$G^2$-Lemma}. Recall that even if we are interested only in 
pure $\Im G$ chains, their evolution equation~\eqref{eq:flowkaim} necessarily contains mixed chains as well. 
The $G^2$-Lemma  turns them back to pure $\Im G$ chains. 
 It expresses how to estimate
 \emph{not} strictly alternating purely $\Im G$ chains in terms of strictly alternating purely $\Im G$ chains based upon
the integral representation~\eqref{eq:intrepIM}.  Note that this formula involves the non-analytic function
$\Im G$ hence simple and flexible contour deformations are prohibited,
contrary to the $k=1$ case, where we did not care about preserving $\Im G$'s 
and  the contour integral \eqref{eq:niceintrep} with the analytic $G(z)$ was applicable.

For brevity we will state the  $G^2$-Lemma for spectral parameters $z_1, ... , z_k$ without time dependence, 
but eventually we will use them  for $z_{1,t}, ... , z_{k,t}$ at any fixed time along the flow.  The proof 
 is given in Section \ref{subsec:proofG^2} below.

\begin{lemma}[$G^2$-Lemma] \label{lem:G^2lemma} Fix $k\ge 2$.  Let $i,j \in [k]$ with $j-i \ge 1$
and assume that $\Phi_l\prec \phi_l$
holds uniformly (in the sense explained after \eqref{phiphi})
 for some control parameters $\phi_l \ge 1$ for $l=1,2,\ldots , k$.
Then, for all versions of $\widehat{G}_{[i^\#,j^\#]}$ and $\widehat{M}_{[i^\#,j^\#]}$, i.e. for any choice of $\#$ indicating star (adjoint), hat (imaginary part) 
or simply no `decoration', we have the following:\footnote{Note that we 
use the $\prec$-notation to purely deterministic quantities. The reason is that it conveniently absorbs 
irrelevant $|\log \eta| \lesssim (\log N)$-factors coming from slightly singular integrals, see Footnote~\ref{log}.}
\begin{equation}
\label{eq:redbm}
\left| \big\langle \widehat{M}_{[i^\#, j^\#]} \big\rangle \right| \prec  \left(\frac{\rho_i \rho_j}{\eta_i \eta_j}\right)^{1/2} \Big(\prod_{n =i+1}^{j-1} \rho_n \Big) \, N^{\frac{j-i}{2}-1} \, \Big(\prod_{m = i}^{j-1} \langle |A_m|^2 \rangle^{1/2}\Big)
\end{equation}
and (the  decorations at the indices $i$ and $j$  on $\widehat{G}$ and on $\widehat{M}$  have to be matching)
\begin{equation}
\label{eq:redbg}
\begin{split}
\left| \big\langle \widehat{G}_{[i^\#,j^\#]} - \widehat{M}_{[i^\#,j^\#]}\big\rangle \right| 
\prec  \, \left(\frac{\rho_i \rho_j}{\eta_i \eta_j} \right)^{1/2} \, \Big(\prod_{n =i+1}^{j-1} \rho_n \Big)  \, \frac{ N^{\frac{j-i}{2}-1} }{\sqrt{N \hell}} \, \Big(\prod_{m = i}^{j-1} \langle |A_m|^2 \rangle^{1/2}\Big)  \tilde{\phi}_{j-i}  \,,
\end{split}
\end{equation}
where we used the notation $\tilde{\phi}_{j-i} = \phi_{j-i} + \mathbf{1}( j-i \ \mathrm{odd} )\sqrt{\phi_{j-i-1} \phi_{j-i+1}}$ 
(as in \eqref{tildephi}).

Moreover, it holds that (now $\#$ indicates star (adjoint) or no `decoration')
\begin{equation}
\label{eq:schwarzeasy}
\left| \big\langle \widehat{{G}}_{[\hat{1}, \hat{k}]}^{(i^\#)}A_k\big\rangle \right| \prec\left( \frac{\rho_i}{\eta_i}\right)^{1/2} \big| \big\langle \Im G_i \big(A_i \Im G_{i+1} ... A_{i-1} \big)\Im G_i \big(A_i \Im G_{i+1} ... A_{i-1}\big)^*\big\rangle\big|^{1/2} \,. 
\end{equation}
\end{lemma}

Since all resolvent chains and their $M$-approximations are multi-linear in the $A$'s,  by a simple
scaling we may assume, without loss of generality,  that $\langle |A_j|^2 \rangle = 1$ for all $j \in [k]$. This 
shortens some formulas.

We start our estimates on $\Phi_k(t)$ with  bounding the quadratic variation of the martingale term in \eqref{eq:flowkaim}:
\begin{equation}
	\begin{split}
	\label{eq:expexprqv}
&\frac{1}{N}\sum_{a,b=1}^N \big[\big|\partial_{ab}  \langle \widehat{{G}}_{[\hat{1}, \hat{k}]}A_k\rangle\big|^2 +\sigma \partial_{ab}  \langle \widehat{{G}}_{[\hat{1}, \hat{k}]}A_k\rangle\overline{\partial_{ba}  \langle \widehat{{G}}_{[\hat{1}, \hat{k}]}A_k\rangle}\big] \\
&\qquad\qquad\qquad\qquad\quad\lesssim \frac{1}{N^2} \sum_{i=1}^{k} \big\langle \big(G_i A_i \Im G_{i+1} ... A_{i-1} G_i \big) \big(G_i A_i \Im G_{i+1} ... A_{i-1} G_i \big)^*\big\rangle \\
&\qquad\qquad\qquad\qquad\quad=  \sum_{i=1}^{k}\frac{ \big\langle \Im G_i \big(A_i \Im G_{i+1} ... A_{i-1} \big)\Im G_i \big(A_i \Im G_{i+1} ... A_{i-1}\big)^*\big\rangle }{N^2 \eta_i^2} \,,
	\end{split}
\end{equation}
where we omitted the time dependence.  Notice that the quadratic variation in \eqref{eq:expexprqv} naturally contains chains of length $2k$. In order to get a closed system of inequalities containing only chains up to length $k$ we rely on the following 
	\emph{reduction inequality}; its proof is given in Appendix \ref{sec:addtech}.
		\begin{lemma}[Reduction inequality]
		\label{lem:redinphi}
		Fix $k\ge 2$, and assume that $\Phi_l(t)\prec \phi_l$ holds uniformly
		 in $t\in [0,T]$ for $0\le  l\le 2k$. Then, uniformly in $t\in[0, T]$ we have
		 		\begin{equation}
			\label{eq:redinphi}
			\Phi_{2k}(t)\prec \begin{cases}
				(N\hell_t)^{1/2}+\frac{1}{(N \hell_t)^{1/2}}\phi_k^2 \quad & k \,\, \mathrm{even} \\
				(N\hell_t)^{1/2}+\phi_{k-1}+\phi_{k+1}+
				\frac{1}{(N \hell_t)^{1/2}}\phi_{k+1}\phi_{k-1} \quad & k \,\, \mathrm{odd}.
			\end{cases}
		\end{equation}
	\end{lemma}
	In the remainder of the proof we will always use \eqref{eq:redinphi} in the following simplified form
	\begin{equation}
	\label{eq:newredin}
	\Phi_{2k}(t)\prec (N\widehat{\ell}_t)^{1/2}+\tilde{\phi}_k^2,
	\end{equation}
with $\tilde{\phi}_k$ being defined in \eqref{tildephi}. Note that \eqref{eq:newredin} follows from \eqref{eq:redinphi} using that $\phi_l\ge 1$. Furthermore, we remark that \eqref{eq:newredin} holds for $k$ being even and odd.

\nc

Adding the prefactor from the definition of $\Phi_{2k}(s)$, 
we find that
\begin{equation}
\label{eq:boundneediter}
\begin{split}
&\frac{\sqrt{N \hell_{t}}}{N^{k/2 - 1} \, \Big(\prod_{i\in [k]} \rho_{i,t}\Big) 
}\left(\int_0^t  \frac{ \big\langle \Im G_{i,s} \big(... \big)\Im G_{i,s} \big(...\big)^*\big\rangle }{N^2 \eta_{i,s}^2} \,\dd s\right)^{1/2} \\
\prec&\, \frac{\sqrt{N \hell_{t}}}{N^{k/2 - 1} \, \Big(\prod_{i\in [k]} \rho_{i,t}\Big) } \left(\int_0^t  \frac{ N^{k-2}\Big(\prod_{i\in [k]} \rho_{i,s}\Big)^2 }{N\eta_s^2}\left(1+\frac{\Phi_{2k}(s)\nc}{(N\hell_s)^{1/2}}\right)  \,\dd s\right)^{1/2} 
\lesssim  \, 1+\frac{\tilde{\phi_k}\nc}{(N\hell_t)^{1/4}}\,,
\end{split}
\end{equation}
analogously to \eqref{eq:qv1}, where we again used that along the characteristics $\hell_s\gtrsim \hell_t$ for $0\le s\le t$ and the integration rule~\eqref{eq:newintrule}. Additionally, in the last inequality we used \eqref{eq:newredin} for $\Phi_{2k}(s)$. \nc
Then, using the BDG inequality we conclude the same
 estimate in high probability for the martingale term in \eqref{eq:flowkaim}.

Next, we bound the first two terms in the last line of  \eqref{eq:flowkaim}. We have
\begin{equation} \label{eq:boundlastline}
	\begin{split}
	& \qquad \frac{\sqrt{N \hell_{t}}}{N^{k/2 - 1} \, \Big(\prod_{i\in [k]} \rho_{i,t}\Big) }\int_0^t \big| \langle G_{i,s}-m_{i,s}\rangle \langle \widehat{{G}}^{(i)}_{[\hat{1},\hat{k}],s}A_k\rangle \big|  \, \dd s \\
	&\prec \frac{\sqrt{N \hell_{t}}}{N^{k/2 - 1} \, \Big(\prod_{i\in [k]} \rho_{i,t}\Big) }\int_0^t \frac{\rho_{i,s}^{1/2}}{N\eta_{i,s}^{3/2}} N^{(k-1)/2}\Big(\prod_{i\in [k]} \rho_{i,s}\Big)\left(1+\frac{\Phi_{2k}(s)\nc}{(N\hell_s)^{1/2}}\right)^{1/2}  \, \dd s\lesssim 1+\frac{\tilde{\phi}_k\nc}{(N\hell_{t})^{1/4}}\,,
	\end{split}
\end{equation}
where we used the bound in \eqref{eq:schwarzeasy} together with a usual single resolvent local law $| \langle G_{i,s}-m_{i,s}\rangle | \prec (N \eta_{i,s})^{-1}$ and applied a similar reasoning as for \eqref{eq:boundneediter}, and in the last inequality we used \eqref{eq:newredin}. \nc

Then, we estimate the terms in $\Omega_\sigma$ of \eqref{eq:flowkaim}. For $j\ne i$ we have
\begin{equation}
\begin{split}
	& \qquad \frac{\sqrt{N \hell_{t}}}{N^{k/2 - 1} \, \Big(\prod_{i\in [k]} \rho_{i,t}\Big)}\int_0^t \frac{1}{N}\big| \langle G_{[\widehat{i},j],s}G_{[\widehat{j},i],s}^\mathfrak{t}\rangle \big|  \, \dd s \\
		&\le \frac{\sqrt{N \hell_{t}}}{N^{k/2 - 1} \, \Big(\prod_{i\in [k]} \rho_{i,t}\Big) }\int_0^t \frac{1}{N} \langle G_{[\widehat{i},j],s}G_{[\widehat{i},j],s}^*\rangle^{1/2} \langle G_{[\widehat{j},i],s}^*G_{[\widehat{j},i],s}\rangle^{1/2} \, \dd s \\
	&\prec \sqrt{N \hell_{t}}\int_0^t \frac{1}{N\eta_{i,s}\eta_{j,s}} \left(1+\frac{\Phi_{2(j-i)}(s)\nc}{\sqrt{N\hell_t}}\right)^{1/2}\left(1+\frac{\Phi_{2(k-j+i)}(s)\nc}{\sqrt{N\hell_t}}\right)^{1/2} \, \dd s \\
	&\lesssim \frac{1}{\sqrt{N\hell_t}}+\frac{\tilde{\phi}_{j-i}\nc}{(N\hell_{t})^{3/4}}+\frac{\tilde{\phi}_{k-j+i} \nc}{(N\hell_{t})^{3/4}}+\frac{\tilde{\phi}_{j-i}\tilde{\phi}_{k-j+i} \nc}{N\hell_{t}}\,, \\
	&\lesssim \frac{1}{\sqrt{N\hell_t}}+\frac{\tilde{\phi}_{j-i}\tilde{\phi}_{k-j+i} \nc}{(N\hell_{t})^{3/4}},
	\end{split}
	\end{equation}
	where in the first inequality we used Schwarz,in the second inequality the Ward identity (see \eqref{eq:newk=1} for similar computations in a simpler case), in the third inequality we used the reduction inequality \eqref{eq:newredin} for $j-i\ge 2$ and that $\tilde{\phi}_1=\phi_1+\sqrt{\phi_2}$ for $j-i=1$, and in the last inequality we used that $\tilde{\phi}_l\ge 1$ for any $l\ge 0$. \nc Similarly, for $j=i$ we get a bound $1+\tilde{\phi}_k\nc/(N\hell_t)^{1/4}$. To combine these two cases in a simpler bound we just estimate
	\begin{equation} \label{eq:sigmafinal}
\frac{\sqrt{N \hell_{t}}}{N^{k/2 - 1} \, \Big(\prod_{i\in [k]} \rho_{i,t}\Big)}\int_0^t \frac{1}{N}\big| \langle G_{[\widehat{i},j],s}G_{[\widehat{j},i],s}^\mathfrak{t}\rangle \big|  \, \dd s\lesssim 1+\frac{\tilde{\phi}_{j-i}\tilde{\phi}_{k-j+i} \nc}{(N\hell_{t})^{1/4}} \,.
	\end{equation}

We are now left with the terms $\Omega_1,\Omega_2,\Omega_3,\Omega_4$ of \eqref{eq:flowkaim}. We write out 
the estimates for $\Omega_1$ as all the other $\Omega_a$, $a=2,3,4$, are completely analogous.
 Using \eqref{eq:redbm}--\eqref{eq:redbg} for $i<j$ we estimate
\begin{equation}
	\begin{split}
		\label{eq:critbound}
		&\frac{\sqrt{N \hell_{t}}}{N^{k/2 - 1} \, \Big(\prod_{i\in [k]} \rho_{i,t}\Big) }\int_0^t \big| \langle \widehat{{G}}_{[\hat{i}, j],s}-\widehat{M}_{[\hat{i},j],s}\rangle\langle \widehat{{M}}_{[\hat{j},i],s}\rangle \big| \, \dd s \\
		&\prec\frac{\sqrt{N \hell_{t}}}{N^{k/2 - 1} \, \Big(\prod_{i\in [k]} \rho_{i,t}\Big) } \int_0^t \frac{N^{(j-i)/2-1}}{\sqrt{N\hell_s}}  \Big(\prod_{n\in [i+1,j-1]} \rho_{n,s}\Big) \, \frac{\rho_{i,s} \rho_{j,s}}{\eta_{i,s} \eta_{j,s}} \,  \Big(\prod_{n\in [i,j]^c} \rho_{n,s}\Big) N^{(k-j+i)/2-1} \tilde{\phi}_{j-i}\,\dd s \\
		&\lesssim \frac{\sqrt{N \hell_{t}}}{N^{k/2 - 1} \, \Big(\prod_{i\in [k]} \rho_{i,t}\Big) } \int_0^t \frac{\tilde{\phi}_{j-i}}{N \eta_s^2} \frac{N^{k/2-1}}{\sqrt{N\hell_s}}  \Big(\prod_{i\in [k]} \rho_{i,s}\Big)\,\dd s \lesssim \frac{\tilde{\phi}_{j-i}}{N \hell_t},
	\end{split}
\end{equation}
where $[i,j]^c:= [1,i-1]\cup [j+1,k]$. Similarly, we bound
\begin{equation} \label{eq:critbound2}
\begin{split}
\frac{\sqrt{N \hell_{t}}}{N^{k/2 - 1} \, \Big(\prod_{i\in [k]} \rho_{i,t}\Big) }  \int_0^t \big| \langle \widehat{{G}}_{[\hat{i}, j],s}-\widehat{M}_{[\hat{i},j],s}\rangle\langle \widehat{{G}}_{[\hat{j},i],s}-\widehat{M}_{[\hat{j}, i],s}\rangle \big| \,\dd s \prec \frac{ \tilde{\phi}_{j-i} \tilde{\phi}_{k-j+i}}{(N\hell_t)^{3/2}}\,.
\end{split}
\end{equation}

Finally, we estimate the last term in the last line of the rhs.~of \eqref{eq:flowkaim} as 
\begin{equation} \label{eq:2ndtermlhs}
\begin{split}
\frac{\sqrt{N \hell_{t}}}{N^{k/2 - 1} \, \Big(\prod_{i\in [k]} \rho_{i,t}\Big) } \int_{0}^{t} \left| \langle \widehat{{G}}_{[\hat{1}, \hat{k}],s}A_k\rangle  \frac{\langle \Im G_{i,s}-\Im m_{i,s}\rangle}{\eta_{i,s}} \right| \, \dd s \prec \frac{1}{\sqrt{N \hell_t}} + \frac{\phi_k}{N \hell_t}\,, 
\end{split}
\end{equation}
where we again used the usual single resolvent local law, the integration rule \eqref{eq:newintrule} and
$$
\big| \langle \widehat{{G}}_{[\hat{1}, \hat{k}],s}A_k\rangle \big| \prec N^{k/2 - 1} \, \Big(\prod_{i\in [k]} \rho_{i,s}\Big)  \left(  1 + \frac{\phi_k}{\sqrt{N\hell_s}}\right)\,. 
$$

Putting all these estimates \eqref{eq:boundneediter}--\eqref{eq:2ndtermlhs} together, we thus obtain

\begin{equation}
	\Phi_k(t)\prec 1+ \frac{1}{N \hat{\ell}_t}\sum_{l=1}^{k} \tilde{\phi}_l +\frac{1}{(N\hat{\ell}_{t})^{3/2}}\sum_{l=1}^{k-1}  \tilde{\phi}_l  \tilde{\phi}_{k-l}+\frac{|\sigma|}{(N\hat{\ell}_t)^{1/4}}\sum_{l=1}^k \tilde{\phi}_l\tilde{\phi}_{k-l}.
	\end{equation}
	Finally, using that $\tilde{\phi}_l\ge 1$, $|\sigma|\le 1$, and that $N\widehat{\ell}_t>1$, $\hell_T\lesssim \hell_t$, we thus conclude \nc \eqref{eq:AVmasterITERATE}.  This finishes the proof of Proposition \ref{pro:masterinIM}, modulo
the proof  of Lemma \ref{lem:G^2lemma} that will be done next. 
\end{proof}

 \subsubsection{Proof of Lemma \ref{lem:G^2lemma}} \label{subsec:proofG^2}

As a preparation for our proof, we observe that the estimate \eqref{eq:Mbound} (modulo logarithmic corrections in $\ell$) even holds true if the condition $N \ell \ge 1$ with 
$$
\ell =  \min_i [\eta_i (\rho_i + \mathbf{1}(i \notin \mathfrak{I}_k))]= \eta_{i_{\min}} (\rho_{i_{\min}} + \mathbf{1}({i_{\min}} \notin \mathfrak{I}_k))]
$$ is violated, but the \emph{second smallest} 
$$
\ell_2   := \min_{i \neq i_{\min}} [\eta_i (\rho_i + \mathbf{1}(i \notin \mathfrak{I}_k))]
$$
 satisfies $N \ell_2 \ge 1$. More precisely, under this weaker assumption, we still have that
\begin{equation} \label{eq:Mbound2ndsmallest}
	\left| \langle \mathcal{M}(z_1, A_1, ... , A_{k-1}, z_k; \mathfrak{I}_k) A_k \rangle \right| \lesssim \big(1 + \mathbf{1}(i_{\min} \notin \mathfrak{I}_k) |\log \ell|\big) \left(\prod_{i \in \mathfrak{I}_k} \rho_i\right) 
	N^{k/2 - 1}\prod_{j \in [k]} \langle |A_j|^2 \rangle^{1/2}\,. 
\end{equation}
This simply follows by realizing that 
the key estimate within the proof of \eqref{eq:Mbound}, namely 
\eqref{eq:intrepbound} in Appendix~\ref{sec:addtech}, can alternatively be estimated as
\begin{equation*} 
	\left| m^{(\mathfrak{I}_k)}[S] \right| \lesssim \big( 1 + \mathbf{1}(i_{\min} \notin \mathfrak{I}_k, i_{\min} \in S) |\log \ell| \big)\frac{\prod_{i \in S \cap \mathfrak{I}_k} \rho_i}{\ell_2^{|S| - 1}},
\end{equation*}
and following the steps leading to the proof of Lemma \ref{lem:Mbound}~(a).\footnote{\label{log} The logarithmic corrections are stemming from the estimate $\int_\R \frac{\rho(x)}{|x-z|} \dd x \lesssim 1+\big|  \log |\Im z| \big|$ (cf.~\eqref{eq:Mdivdiff}).} 
		We now turn to the actual proof of Lemma~\ref{lem:G^2lemma} and 
		again assume that, by simple scaling, $\langle |A_m|^2 \rangle = 1$ for all $m \in [k]$.

We start with the proof of \eqref{eq:redbm} for both $\#$'s indicating no decoration and 
assuming, for definiteness,  that $\eta_i=\Im z_i >0$ and  $\eta_j=\Im z_j > 0$; all other cases can be treated similarly and are hence omitted. 
 In this case, we use the integral representation \cite[Eq.~(3.15)]{multiG} (which simply follows from \eqref{eq:Mdefim}--\eqref{eq:Mdivdiff} using multilinearity)\footnote{Alternatively, this can also be obtained using
  \eqref{eq:intrepIM} for $m=2$: The resolvent chain, which is approximated by $\langle \widehat{M}_{[i,j]} \rangle$ contains a $G_j G_i$-factor after cyclicity of the trace.
   Applying \eqref{eq:intrepIM} for $m=2$ to this part of the chain and using a \emph{meta argument} like in Appendix~\ref{sec:addproofsec4}, we can also conclude \eqref{eq:intrepMbasic}.} 
\begin{equation} \label{eq:intrepMbasic}
	\big\langle \widehat{M}_{[i, j]} \big\rangle = \frac{1}{\pi } \int_\R \frac{\big\langle\widehat{M}(x+\ii \zeta, A_i, z_{i+1}, ... , z_{j-1}) A_{j-1}\big\rangle}{(x-z_i + \ii \zeta)(x - z_j + \ii \zeta)} \dd x
\end{equation}
with $\zeta := (\eta_i \wedge \eta_j)/2$.  To estimate the $x$-integration in~\eqref{eq:intrepMbasic},
we will apply the following basic lemma, which shall frequently be used in the sequel. 
Its proof is omitted as it is a simple Hölder's inequality and elementary calculus using
basic properties of $\rho(z)$. 
\begin{lemma} \label{lem:xrestore}
	Under the setting and assumptions described above, for any $\alpha \in [0,1]$, we have that  
	\begin{equation} \label{eq:xdeprestore}
		\frac{1}{\zeta^{ \alpha}}\int_{\R } 
		\frac{\big(\rho(x + \ii \zeta)\big)^{1 - \alpha} }{| x-z_i + \ii \zeta| |x - z_j + \ii \zeta|} \dd x \prec \frac{1}{(\eta_i \eta_j)^{1/2}}\left(\frac{\rho_i \rho_j}{\big(({\eta}_i \rho_i)
		 ({\eta}_j \rho_j)\big)^{\alpha}}\right)^{1/2} \,. 
	\end{equation}
\end{lemma}
Therefore, plugging in \eqref{eq:Mbound2ndsmallest} with $\mathbf{1}(...) = 0$ for the numerator in \eqref{eq:intrepMbasic} and then 
using \eqref{eq:xdeprestore},
we obtain 
\begin{equation} \label{eq:intrepM}
	\begin{split}
		\left| \big\langle \widehat{M}_{[i, j]} \big\rangle \right| \lesssim \left(\prod_{n = i+1}^{j-1} \rho_n\right) 
		N^{(j-i)/2 - 1} \int_\R \frac{\rho(x + \ii \zeta)}{| x-z_i + \ii \zeta| |x - z_j + \ii \zeta|} \dd x \prec \left(\frac{\rho_i \rho_j }{\eta_i \eta_j}\right)^{1/2} \left(\prod_{n = i+1}^{j-1} \rho_n\right) 
		N^{(j-i)/2 - 1} \,,
	\end{split}
\end{equation}
completing the proof of \eqref{eq:redbm}.

We now turn to the proof of \eqref{eq:redbg}, again focusing on the case where both $\#$'s indicate no decoration and assuming that $ \eta_i=\nc \Im z_i >0$ and  $\eta_j=\Im z_j > 0$.  
As the first step, we apply the integral representations \eqref{eq:intrepIM} and \eqref{eq:intrepMbasic} (see  \cite[Eqs.~(3.14) and (3.15)]{multiG}) to find
\begin{equation} \label{eq:intrepG-M}
	\begin{split}
		\left| \big\langle \widehat{G}_{[i, j]} - \widehat{M}_{[i, j]} \big\rangle \right| \lesssim  \int_\R \frac{\big| \big\langle \big(\Im G(x + \ii \zeta) A_i \widehat{G}_{[\, \widehat{i+1}, \, \widehat{j-1}\, ]} - \widehat{M}(x+\ii \zeta, A_i, ...)\big) A_{j-1}\big\rangle\big|}{| x-z_i + \ii \zeta| |x - z_j + \ii \zeta|} \dd x 
	\end{split}
\end{equation}
with $\zeta = (\eta_i \wedge \eta_j)/2$ and split the integral into an \emph{above the scale} and a \emph{below the scale} part. This concept refers to spectral regimes  $x\in \R$ where the typical eigenvalue spacing $\rho(x + \ii \zeta)/N$ is larger or smaller
than the given $\zeta$. 
 More precisely, we fix an arbitrarily small $\xi > 0$ and decompose 
 $\R$ into\footnote{To be precise, in the integral \eqref{eq:intrepG-M}
 we first need to cut-off the regime where $|x| \ge N^{100}$, say, and estimate this contribution by a simple norm bound using that the spectrum of the Wigner matrix is contained in $[-2-\epsilon, 2+ \epsilon]$ with very high probability \cite{EYY2012}. Such technicality about the irrelevant, very far out $x$-regime will henceforth be ignored.}
\begin{equation} \label{eq:abovebelow}
\big\{ x: N \rho(x+\ii \zeta) \zeta \ge N^\xi \big\}\, \dot{\cup}  \, \big\{ x:   N \rho(x+\ii \zeta) \zeta < N^\xi \big\} =: I_{\mathrm{above}} \, \dot{\cup}\, I_{\mathrm{below}}\,.
\end{equation}

For the \emph{above the scale} part, we use that $\Phi_{j-i} \prec \phi_{j-i}$ and estimate this part of the integral 
\eqref{eq:intrepG-M} by
\begin{equation} \label{eq:above1}
	\int_{I_{\rm above}} \frac{\rho(x + \ii \zeta)}{| x-z_i + \ii \zeta| |x - z_j + \ii \zeta| }\left(\prod_{n = i+1}^{j-1} \rho_n\right) 
	\frac{N^{(k-i)/2 - 1}}{\sqrt{N \hell(x)}}  \phi_{j-i}\dd x \,,
\end{equation}
where we emphasized that now $\hell(x) = \zeta \rho(x+\ii \zeta) \wedge \min_{n \in [i+1,j-1]} \eta_n \rho_n$  depends on the integration variable $x$ since the integrated chain in \eqref{eq:intrepG-M} contains a resolvent 
at spectral parameter $x+\ii\zeta$.
Next, we further split $I_{\rm above} $ into two parts $I_{\rm above} =  I_{{\rm above}, =} \, \dot{\cup} \,  I_{{\rm above}, <}$ with
\begin{equation} \label{eq:above=<}
	I_{{\rm above}, =} := \big\{ x :  \hell(x) = \rho(x + \ii \zeta) \zeta\big\} \quad \text{and} \quad  I_{{\rm above}, <} := \big\{ x :  \hell(x) < \rho(x + \ii \zeta) \zeta\big\},
\end{equation} 
depending on whether the minimum is attained at the special spectral argument $x+\ii\zeta$ or not, 
and estimate each of them separately. In this way, we obtain the contribution from $I_{{\rm above}, =}$ to \eqref{eq:above1} to equal 
\begin{equation} \label{eq:above2}
	\frac{1}{\sqrt{N}} \left[\frac{1}{\zeta^{1/2}}\int_{I_{{\rm above},=}} \frac{\big(\rho(x + \ii \zeta)\big)^{1/2}}{| x-z_i + \ii \zeta| |x - z_j + \ii \zeta| }  \dd x\right]\rho_{i+1} \ldots\rho_{j-1} N^{(j-i)/2 - 1}  \phi_{j-i} \,. 
\end{equation}
By means of Lemma \ref{lem:xrestore} with $\alpha = 1/2$ applied to the integral in $\big[\cdots\big]$, this can be bounded as 
\begin{equation} \label{eq:above3}
	\begin{split}
		\frac{1}{(\eta_i \eta_j)^{1/2}}	\frac{1}{\sqrt{N}} \sum_{s=1}^{j-i} \frac{ \sqrt{\rho_i} \,  \rho_{i+1}
		 \ldots  \rho_{j-1} \sqrt{\rho_j}}{\sqrt{({\eta}_i \rho_i)^{1/2}} \sqrt{({\eta}_j \rho_j )^{1/2}}}N^{(j-i)/2 - 1}  \phi_{j-i} \le  \left(\frac{\rho_i \rho_j }{\eta_i\eta_j}\right)^{1/2}\left(\prod_{n = i+1}^{j-1} \rho_n\right) 
		\frac{N^{(j-i)/2 - 1}}{\sqrt{N \hell}} \phi_{j-i}   \,. 
	\end{split}
\end{equation}
For $I_{{\rm above}, <}$ the argument is completely analogous, yielding exactly the same bound as in \eqref{eq:above3}. 
This completes the bound for the  \emph{above the scale} part.

For the \emph{below the scale} part, we estimate the two terms in the numerator 
 in \eqref{eq:intrepG-M} separately; in this regime the local law is anyway not effective 
 in the sense that $G-M$ is not smaller than $G$. 
  For the $\widehat{M}$-term, we recall the bound \eqref{eq:Mbound2ndsmallest}, and estimate
\begin{equation} \label{eq:belowM}
	\begin{split}
		&\int_{I_{\rm below}} \frac{\big| \big\langle  \widehat{M}(x+\ii \zeta, A_i, ...)A_{j-1}\big\rangle\big|}{| x-z_i + \ii \zeta| |x - z_j + \ii \zeta|} \dd x 
		\lesssim \, N^{(j-i)/2 - 1}\rho_{i+1}\ldots  \rho_{j-1}\left[ \int_{I_{\rm below}}\frac{\rho(x + \ii \zeta) }{| x-z_i + \ii \zeta| |x - z_j + \ii \zeta|} \dd x\right]  \\[2mm]
		\prec &\frac{1}{(\eta_i \eta_j)^{1/2}}\frac{1}{N} \frac{\sqrt{\rho_i} \, \rho_{i+1}\ldots \rho_{j-1} \, \sqrt{\rho_j}}{({\eta}_i \rho_i)^{1/2} ({\eta}_j \rho_j)^{1/2}}N^{(j-i)/2 - 1}
		\lesssim \left(\frac{\rho_i \rho_j }{\eta_i\eta_j}\right)^{1/2}\left(\prod_{n = i+1}^{j-1} \rho_n\right) 
		\frac{N^{(j-i)/2 - 1}}{N \hell} \,. 
	\end{split}
\end{equation}
To go from the second to the third line, we used that $\rho(x + \ii \zeta) \zeta \prec N^{-1}$ for $x \in I_{\rm below}$ (recall that $\xi> 0$ in the definition~\eqref{eq:abovebelow}
 may be chosen arbitrarily small) and employed Lemma \ref{lem:xrestore} with $\alpha = 1$. In the ultimate step, we utilized 
$\eta_i \rho_i \wedge \eta_j \rho_j \ge \hell$ together with $N \hell \gtrsim 1$.
 This concludes the discussion of the $\widehat{{M}}$-term.

Next, we turn to the $\widehat{{G}}$-term in \eqref{eq:intrepG-M} in the regime $x \in I_{\rm below}$ and
first  focus on the case where $j-i$ is even. Here, we employ a Schwarz inequality in order to be able to exploit 
\begin{equation} \label{eq:monotone}
	\begin{split}
		\big| \big\langle \Im G(x + \ii \zeta) &A_i \widehat{G}_{[\, \widehat{i+1}, \, \widehat{j-1}\, ]} A_{j-1}\big\rangle\big| \\
		\le \, & \frac{\zeta_x}{\zeta} \big| \big\langle \Im G(x + \ii \zeta_x) (A_i ... \Im G_{r-1} A_{r-1}) \Im G_{r} (A_i ... \Im G_{r-1} A_{r-1})^*\big\rangle \big|^{1/2} \\
		& \quad \times\big| \big\langle \Im G_{r} (A_{r} ... \Im G_{j-1} A_{j-1}) \Im G(x + \ii \zeta_x) (A_{r} ... \Im G_{j-1} A_{j-1}) ^*\big\rangle \big|^{1/2} 
	\end{split}
\end{equation}
where $\zeta_x > \zeta $ is implicitly defined via $N \rho(x + \ii \zeta_x) \zeta_x = N^\xi$ and we denoted $r := (i+j)/2$. After application of a Schwarz inequality, we find this part of \eqref{eq:intrepG-M} to be bounded by 
\begin{equation} \label{eq:below1}
	\begin{split}
		\left(\int_{I_{\rm below}} \frac{\zeta_x}{\zeta}\frac{\big| \big\langle \Im G(x + \ii \zeta_x) (A_i ... \Im G_{r-1} A_{r-1}) \Im G_{\frac{j+i}{2}} (A_i ... \Im G_{r-1} A_{r-1})^*\big\rangle \big|}{| x-z_i + \ii \zeta| |x - z_j + \ii \zeta|} \dd x\right)^{1/2} \\
		\times \left(\int_{I_{\rm below}} \frac{\zeta_x}{\zeta}\frac{ \big| \big\langle \Im G_{r} (A_{r} ... \Im G_{j-1} A_{j-1}) \Im G(x + \ii \zeta_x) (A_{r} ... \Im G_{j-1} A_{j-1})^*\big\rangle \big| }{| x-z_i + \ii \zeta| |x - z_j + \ii \zeta|} \dd x\right)^{1/2}
	\end{split}
\end{equation}

Adding and subtracting the respective $\widehat{{M}}$-terms for both resolvent chains in \eqref{eq:below1}, we are left with two terms for each integral. For concreteness, we estimate the one in the first line in \eqref{eq:below1}, the second line the same. The first line in \eqref{eq:below1} is bounded by (the square root of)
\begin{equation*}
	\begin{split}
		\frac{1}{\zeta}\int_{I_{\rm below}} \dd x\frac{\zeta_x \rho(x + \ii \zeta_x)}{| x-z_i + \ii \zeta| |x - z_j + \ii \zeta| } &\left(\prod_{n = i+1}^{r-1} \rho_n\right)^2 \rho_{r} \, N^{\frac{j-i}{2}-1}( 1+\phi_{j-i})\\
		&\prec( 1+\phi_{j-i})\left(\frac{ \rho_i \rho_j}{\eta_i\eta_j}\right)^{1/2}  \left(\prod_{n = i+1}^{r-1} \rho_n\right)^2 \rho_{r} \, \frac{N^{\frac{j-i}{2}-1}}{N \hell}  \,. 
	\end{split}
\end{equation*}
Here, we used that $N \rho(x + \ii \zeta_x) \zeta_x = N^\xi$ for arbitrarily small $\xi > 0$ and employed
 Lemma \ref{lem:xrestore} (with $\alpha=1$) in estimates analogous to \eqref{eq:above3} and \eqref{eq:belowM}.
  Combining this with the identical estimate for the second line of \eqref{eq:below1} and using $N \hell \ge 1$ and $\phi_{j-i} \ge 1$, we finally deduce that 
\begin{equation} \label{eq:below2}
	\eqref{eq:below1} \prec \phi_{j-i}\left(\frac{ \rho_i \rho_j}{\eta_i\eta_j}\right)^{1/2} \left(\prod_{n = i+1}^{j-1} \rho_n\right) 
	\frac{N^{(j-i)/2 - 1}}{\sqrt{N \hell}} \,. 
\end{equation}

For $j-i$ being odd, only the monotonicity argument \eqref{eq:monotone} is different: 
\begin{equation*} 
	\begin{split}
		\big| \big\langle \Im G(x + \ii \zeta) &A_i \widehat{G}_{[\, \widehat{i+1}, \, \widehat{j-1}\, ]} A_{j-1}\big\rangle\big| \\
		\le \, & \frac{\zeta_x}{\zeta} \big| \big\langle \Im G(x + \ii \zeta_x) (A_i ... \Im G_{r-1} A_{r-1}) \Im G_{r} (A_i ... \Im G_{r-1} A_{r-1})^*\big\rangle \big|^{1/2} \\
		& \quad \times\big| \big\langle \Im G_{r} (A_{r+1} ... \Im G_{j-1} A_{j-1}) \Im G(x + \ii \zeta_x) (A_{r+1} ... \Im G_{j-1} A_{j-1})^*\big\rangle \big|^{1/2} \,,
	\end{split}
\end{equation*}
where we now denoted $r:= (i+j+1)/2$. 
This asymmetry in the lengths of the resolvent chains now leads to the term $\sqrt{\phi_{j-i+1} \phi_{j-i-1}}$ in \eqref{eq:redbg}, the rest of the argument is identical.

Finally, we turn to the proof of \eqref{eq:schwarzeasy}. Again, we focus on the case where $\#$ indicates no decoration. By application of a Schwarz inequality, we find
\begin{equation} \label{eq:schwarzeasyproof}
	\begin{split}
		\left| \big\langle \widehat{{G}}_{[\hat{1}, \hat{k}]}^{(i)}A_k\big\rangle \right| &= \left|\big\langle \Im G_1 A_1 ... \Im G_{i-1} A_{i-1} \Im G_i G_i A_i ... \Im G_k A_k\big\rangle 	\right| \\
		&\le |\langle G_i G_i^* \rangle|^{1/2} \left| \big\langle \Im G_i(A_i\Im G_{i+1}\dots A_{i-1})\Im G_i(A_i\Im G_{i+1}\dots A_{i-1})^* \big\rangle \right|^{1/2}	\\
		&\prec\left( \frac{\rho_i}{\eta_i}\right)^{1/2} \left| \big\langle \Im G_i(A_i\Im G_{i+1}\dots A_{i-1})\Im G_i(A_i\Im G_{i+1}\dots A_{i-1})^* \big\rangle \right|^{1/2}\,,
	\end{split}
\end{equation}
where in the last step we used the Ward identity $GG^* = \Im G/\eta$ together with the usual single resolvent local law applied to $\Im G_i$. This concludes the proof of Lemma \ref{lem:G^2lemma}
which was the last missing piece for the proof of Proposition \ref{prop:zig}~(a) for pure $\Im G$ chains. \qed

\subsection{Proof of Proposition \ref{prop:zig}~(b) for pure $\Im G$-chains} \label{subsec:pureIMISO}
In this section, we briefly explain how to derive Proposition \ref{prop:zig}~(b) from Proposition \ref{prop:zig}~(a). For fixed spectral parameters and bounded deterministic vectors $\Vert \bm x \Vert \,, \Vert \bm y \Vert \lesssim 1$, we have 
\begin{equation} \label{eq:isofromav}
\left|\langle \bm x, (\widehat{{G}}_{[\hat{1}, \widehat{k+1}]}-\widehat{M}_{[\hat{1},\widehat{k+1}]})\bm y\rangle\right| \lesssim \left|\big\langle  (\widehat{{G}}_{[\hat{1}, \widehat{k+1}]}-\widehat{M}_{[\hat{1},\widehat{k+1}]})A_{k+1}\big\rangle\right| + \left|\big\langle  \widehat{{G}}_{[\hat{1}, \widehat{k+1}]}-\widehat{M}_{[\hat{1},\widehat{k+1}]}\big\rangle\right| 
\end{equation}
with the special choice $A_{k+1} := N \bm y \bm x^* - \langle \bm x , \bm y \rangle$. Next, using that $\langle |A_{k+1}|^2 \rangle^{1/2} \lesssim N^{1/2}$ we find from Proposition \ref{prop:zig}~(a) for pure $\Im G$ chains the first term in \eqref{eq:isofromav} to be bounded as 
\begin{equation*}
\left|\big\langle  (\widehat{{G}}_{[\hat{1}, \widehat{k+1}]}-\widehat{M}_{[\hat{1},\widehat{k+1}]})A_{k+1}\big\rangle\right| \prec \Big(\prod_{i\in [k+1]} \rho_{i}\Big)\frac{N^{k/2 }}{\sqrt{N \hell}} \prod_{j \in [k]} \langle |A_j|^2 \rangle^{1/2}\,. 
\end{equation*}
For the second term, we apply \eqref{eq:redbg} from Lemma \ref{lem:G^2lemma} (note that by Proposition \ref{prop:zig}~(a) for pure $\Im G$ chains, we have $\Phi_k\prec \phi_k:= 1$ and hence also $\tilde\phi_k=1$) and obtain
\begin{equation*}
\left|\big\langle  \widehat{{G}}_{[\hat{1}, \widehat{k+1}]}-\widehat{M}_{[\hat{1},\widehat{k+1}]}\big\rangle\right|  \prec \left(\frac{\rho_1 \rho_{k+1}}{\eta_1 \eta_{k+1}}\right)^{1/2}\Big(\prod_{i = 2}^{k} \rho_{i}\Big)\frac{N^{k/2-1 }}{\sqrt{N \hell}} \prod_{j \in [k]} \langle |A_j|^2 \rangle^{1/2} \le  \Big(\prod_{i\in [k+1]} \rho_{i}\Big)\frac{N^{k/2 }}{\sqrt{N \hell}} \prod_{j \in [k]} \langle |A_j|^2 \rangle^{1/2}\,,
\end{equation*}
where in the last step we used $\eta_1 \rho_1 \wedge \eta_{k+1} \rho_{k+1} \ge \hell$ and $N \hell \ge 1$. This concludes the proof of Proposition~\ref{prop:zig}~(b) for pure $\Im G$ chains. 

\subsection{Proof of Proposition \ref{prop:zig}~(b)  for mixed chains} \label{subsec:mixed}
We consider mixed resolvent chains
$$
\mathcal{G}_1 A_1 ... \mathcal{G}_k A_k
$$
with  $\mathcal{G}_j \in \{ G_j, \Im G_j\}$ and traceless matrices $A_1, ... , A_k \in \C^{N \times N}$, and explain how the respective bounds in \eqref{eq:mainAV}--\eqref{eq:mainISO} are obtained from the multi-resolvent local law for pure $\Im G$-chains derived in Sections~\ref{subsec:pureIM}--\ref{subsec:pureIMISO}. We will henceforth focus on the average case, the isotropic bounds can immediately be obtained from those by following Section \ref{subsec:pureIMISO}. 

 Recalling 
$$
 \ell = \min_{j \in [k]}\big[ \eta_j (\rho_j + \mathbf{1}(j \notin \mathfrak{I}_k))\big]
 $$
 where $\mathfrak{I}_k$  denotes the
 set of indices $j \in [k]$ where $\mathcal{G}_j = \Im G_j$, the goal of this section is to prove that
	\begin{equation} \label{eq:mainAVproof}
		\left| \langle \mathcal{G}_1 A_1 ... \mathcal{G}_k A_k \rangle - \langle \mathcal{M}_{[1,k]}A_k \rangle \right| \prec \left[\left(\prod_{i \in \mathfrak{I}_k} \rho_i \right) \wedge \max_{i \in [k]} \sqrt{\rho_i}\right] \, 
		\frac{N^{k/2 - 1}}{\sqrt{N \ell}} \,  \prod_{j \in [k]} \langle |A_j|^2 \rangle^{1/2}\,.
	\end{equation}

In order to do so, we iteratively apply the integral representation \eqref{eq:intrepIM} with $m=1$ for every $\mathcal{G}_j$ such that $j \notin \mathfrak{I}_k$. In Section \ref{subsec:I=notempty}, this procedure will immediately yield the claimed bound \eqref{eq:mainAVproof} for $\mathfrak{I}_k \neq \emptyset$ (recall from Remark~\ref{rmk:MHS}~(ii), that in this case the minimum in \eqref{eq:mainAVproof} is always realized by the product). In the complementary case, $\mathfrak{I}_k = \emptyset$, which has already been studied in \cite{A2}, the outcome of iteratively applying \eqref{eq:intrepIM} is the natural continuation of the pattern obtained for $\mathfrak{I}_k \neq \emptyset$.
However,  in this way we only find the weaker bound, where in \eqref{eq:mainAVproof} the minimum $\big[... \wedge ...\big]$ replaced by one. The improvement to include the small factor  $\max_{i \in [k]} \sqrt{\rho_i}$ requires a short separate argument, which we provide in Section~\ref{subsec:I=empty}. 

\subsubsection{The case $\mathfrak{I}_k \neq \emptyset$} \label{subsec:I=notempty}For concreteness, we consider the case
 where $\mathfrak{I}_k = [k-1]$, i.e.~$\mathcal{G}_k = G_k$ with $\Im z_k > 0$ w.l.o.g.~and all other $\mathcal{G}$'s are $\Im G$'s. Then, using the integral representation \eqref{eq:intrepIM} with $m=1$ and $\eta = \zeta = \Im z_k/2$, and its analog for the deterministic approximation (see \cite[Eqs.~(3.14) and (3.15)]{multiG}
 and \eqref{eq:intrepG-M} above), we find that
	\begin{equation*} 
		\begin{split}
	&\left| \langle \Im{G}_1 A_1 ... G_k A_k \rangle - \langle \mathcal{M}(z_1, A_1, ... , z_k; [k-1])A_k \rangle \right| \\
	& \qquad \lesssim \int_\R \frac{	\left| \langle \Im{G}_1 A_1 ... \Im G(x + \ii \zeta) A_k \rangle - \langle \mathcal{M}(z_1, A_1, ... , x+ \ii \zeta; [k])A_k \rangle \right|}{|x - z_k + \ii \zeta|} \, \dd x 
		\end{split}
\end{equation*}

We then follow the steps in the proof of Lemma \ref{lem:G^2lemma} starting from \eqref{eq:intrepG-M} in order to estimate the integral. In particular, we split the integration region into $I_{\rm above}$ and $I_{\rm below}$, just as in \eqref{eq:abovebelow}. In the treatment of these regimes, the two main differences compared to the proof of Lemma \ref{lem:G^2lemma} are the following: 
\begin{itemize}
\item[(i)] We use the $M$-bound in \eqref{eq:Mbound2ndsmallest} with logarithmic corrections, which can be absorbed into $\prec$.
\item[(ii)] Lemma \ref{lem:xrestore} gets replaced by the bound
$$
\int_\R \frac{\big(\rho(x+\ii \zeta)\big)^\alpha}{|x-z_k + \ii \zeta|} \, \dd x \prec 1 \qquad \text{for all} \qquad \alpha > 0\,,
$$
which can easily be seen using that $\Im z_k \ge N^{-1}$ and $\rho(w)$ decays polynomially as $|w| \to \infty$.
\end{itemize}
 For example, instead of \eqref{eq:above2} we estimate (recall that $I_{\rm above}$ is further split into $I_{\rm above, =}$ and $I_{\rm above, <}$ in \eqref{eq:above=<})
\begin{equation*}
	\begin{split}
	\frac{1}{\sqrt{N}} \left[\frac{1}{\zeta^{1/2}}\int_{I_{{\rm above},=}} \hspace{-2mm}\frac{\big(\rho(x + \ii \zeta)\big)^{1/2}}{| x-z_k + \ii \zeta|}  \dd x\right]\rho_{1} \ldots \rho_{k-1} N^{k/2 - 1} \prec \left(\prod_{i \in [k-1]} \rho_i \right) \, 
\frac{N^{k/2 - 1}}{\sqrt{N \ell}} \,,
	\end{split}
\end{equation*}
neglecting the product of Hilbert-Schmidt norms. We point out that, compared to the estimates in the pure $\Im G$-case, now $\ell := \min_{j \in [k]}\big[ \eta_j (\rho_j + \mathbf{1}(j \neq k))\big]$ and $\rho_k$ disappeared from the rhs. Therefore, as a result, we find the claimed bound \eqref{eq:mainAVproof} for $\mathfrak{I}_k = [k-1]$. All other cases with $\mathfrak{I}_k \neq \emptyset$ follow by iteratively applying this strategy.  This completes the proof of Proposition \ref{prop:zig}~(b) if
$\mathfrak{I}_k \ne \emptyset$.

\subsubsection{The case $\mathfrak{I}_k = \emptyset$} \label{subsec:I=empty}As mentioned above,  
in order to obtain the improvement by $\max_{i \in [k]} \sqrt{\rho_i}$, we now give a separate argument. We thereby closely follow the steps in Section \ref{subsec:pureIM} and point out only the main differences. In particular, we now use the flow \eqref{eq:flowka}, together with the following lemma proven in Appendix~\ref{sec:addproofsec4}, instead of \eqref{eq:flowkaim}. Here, similarly to \eqref{eq:flowka}, the absence of hats indicates that \emph{none} of the resolvents $\mathcal{G}$ in the chain approximated by $M$ is an
$\Im G$. 
	\begin{lemma}
		\label{lem:cancM}
		We have
		\begin{equation*}
			\partial_t \langle M_{[1,k],t}A_k\rangle=\frac{k}{2}\langle M_{[1,k],t}A_k\rangle+\sum_{i,j=1, \atop i<j}^k\langle M_{[i,j],t}\rangle \langle M_{[j,i],t}\rangle. 
		\end{equation*}
	\end{lemma}
 Moreover, using the shorthand notations 
$$\eta_t := \min_{i \in [k]} \eta_{i,t} \quad \text{and} \quad \rho_t:= \max_{i \in [k]} \rho_{i,t}\,,$$
 we introduce the new normalized differences
\begin{equation}
	{\Psi}_k(t):= \frac{\sqrt{N \eta_{t}}}{N^{k/2 - 1} \, \sqrt{\rho_t} \prod_{j \in [k]} \langle |A_j|^2 \rangle^{1/2}} \big| \langle ({G}_{[{1},{k}],t}- {M}_{[{1},{k}],t})A_k \rangle\big|
\end{equation}
for every $k \in \N$. The $\Psi_k$'s introduced here are the no-$\Im G$-analogs of the $\Phi_k$'s defined in \eqref{eq:defphi}, i.e.~all hats are removed and we replaced $\hat{\ell}_t \to \eta_t$ as well as $\prod_i \rho_{i,t} \to \sqrt{\rho_t}$. 

In the following, we will derive master inequalities for the $\Psi_k$'s, analogously to Proposition \ref{pro:masterinIM}. However, compared to the proof in Section \ref{subsec:pureIM}, we now have two major simplifications: 
\begin{itemize}
\item[(i)] Since the bound \eqref{eq:mainAVproof} for $\mathfrak{I}_k \neq \emptyset$ is already proven, the 
contribution of the
quadratic variation term in \eqref{eq:flowka}, which automatically carries two $\Im G$'s, is easily estimated as (again assuming $\langle |A_j|^2 \rangle = 1$ for all $j \in [k]$ henceforth)
\begin{equation*}
	\begin{split}
		\frac{\sqrt{N \eta_{t}}}{N^{k/2 - 1} \,  \sqrt{\rho_{t}} }&\left(\int_0^t  \frac{ \big\langle \Im G_{i,s} \big(A_i G_{i+1,s}... A_{i-1} \big)\Im G_{i,s}  \big(A_i G_{i+1,s}... A_{i-1} \big)^*\big\rangle }{N^2 \eta_{i,s}^2} \,\dd s\right)^{1/2} \\
		\prec&\, \frac{\sqrt{N \eta_{t}}}{N^{k/2 - 1} \,  \sqrt{\rho_{t}} } \left(\int_0^t  \frac{ N^{k-2} \rho_{i,s}^2 }{N\eta_{i,s}^2}  \,\dd s\right)^{1/2} 
		\lesssim  \, \sqrt{\frac{\rho_{i,t} \eta_t}{\rho_t \eta_{i,t}}} \le 1\,,
	\end{split}
\end{equation*}
analogously to \eqref{eq:boundneediter}. Note that in the first step, we did not use the overestimate $1/\eta_{i,s} \le 1/ \eta_s$ inside the integral as done in \eqref{eq:boundneediter}. The same reasoning applies to the analog of the first two terms in the \nc last line of \eqref{eq:flowkaim} and the terms contained in $\Omega_\sigma$ (cf.~the estimates in \eqref{eq:boundlastline}--\eqref{eq:sigmafinal})\nc. We point out that, in this section, the already proven bounds for resolvent chains containing at least one $\Im G$ make the usage of reduction inequalities as in Lemma \ref{lem:redinphi} obsolete. 

\item[(ii)] For treating the analogues of $\Omega_1,\Omega_2,\Omega_3,\Omega_4$ in \eqref{eq:flowkaim}, it is not necessary to "restore" $\Im G$'s via the integral representation \eqref{eq:intrepIM} as in the proof of the $G^2$-Lemma \ref{lem:G^2lemma}. Instead, in the course of proving an analog of Lemma \ref{lem:G^2lemma} (again suppressing the time dependence of the $z$'s as well as $\eta$ and $\rho$) it is sufficient to apply resolvent identities for $|z_i- z_j| \ge \eta$ and the integral representation 
\begin{equation*} 
	G(z_i) G(z_j) = \frac{1}{2 \pi \ii} \int_\Gamma \frac{G(w)}{(w-z_i)(w-z_j)} \dd w\,,
\end{equation*}
for $|z_i - z_j| \le \eta$. In this case $z_i$ and $z_j$ are necessarily on the same halfplane ($\Im z_i \Im z_j > 0$) and, just as in \eqref{eq:niceintrep}, $\Gamma$ is a tiny contour encircling $z_i, z_j \in \C \setminus \R$ in such a way that $\mathrm{dist}(\Gamma, \{z_i,z_j\}) \sim \eta$, which ensures that $ |\Im m(w)| \lesssim  \max_{i \in [k]} \rho_i$ 
on $\Gamma$ as follows by elementary continuity properties of $m(w)$. 

As a consequence, for fixed $k \in \N$, we find, assuming $\Psi_l\prec \psi_l$ for some control parameters $\psi_l \ge 1$ for $l=1,2,\ldots , k$ in
 the usual sense of uniformity 
explained below \eqref{phiphi}, that 
\begin{equation*}
	\left| \big\langle M_{[i, j]} \big\rangle \right| \prec \frac{1}{\eta} \, N^{\frac{j-i}{2}-1} \, \Big(\prod_{m = i}^{j-1} \langle |A_m|^2 \rangle^{1/2}\Big)\,,
\end{equation*}
as an analog of \eqref{eq:redbm} and 
\begin{equation*}
	\begin{split}
		\left| \big\langle {G}_{[i,j]} - {M}_{[i,j]}\big\rangle \right| 
		\prec  \, \frac{1}{\eta} \,  N^{\frac{j-i}{2}-1} \sqrt{\frac{\rho  }{{N \eta}}} \, \Big(\prod_{m = i}^{j-1} \langle |A_m|^2 \rangle^{1/2}\Big)  {\psi}_{j-i}  \,,
	\end{split}
\end{equation*}
as an analog of \eqref{eq:redbg}, for all $i,j \in [k]$ with $j-i \ge 1$. 
\end{itemize}
Overall, using the above two simplifications and following the arguments in \eqref{eq:boundneediter}--\eqref{eq:2ndtermlhs}, we arrive at the following new set of master inequalities. 
 \begin{proposition}[Master inequalities II]
	\label{pro:masterin} 
	Fix $k\in \N$. Assume that $\Psi_l(t)\prec \psi_l$ for any $1\le l\le k$ uniformly in $t\in [0,T]$, in spectral parameters with $N\hell_t\ge N^\epsilon$ (recall $\hell_t = \min_{i \in [k]} \eta_{i,t} \rho_{i,t}$ from \eqref{eq:newintrule}; not to be confused with the $\ell$ used around \eqref{eq:mainAVproof}!) and in traceless deterministic matrices $A_j$ for some $\psi_l\ge 1$ and set $\psi_0 := 1$. Then we have the \emph{master inequalities} 
		\begin{equation}
		\begin{split}
			\label{eq:AVmasterITERATE2}
			\Psi_k(t)\prec 1+\frac{1}{(N\hell_{T})^{1/4}}\sum_{l=1}^{k}  {\psi}_l  {\psi}_{k-l} 
			\end{split}
	\end{equation}
	uniformly (in the sense explained below \eqref{phiphi}) 
	in $t\in [0,T]$.
\end{proposition}

Using that $N \hell_T \ge N^\epsilon \nc$ and iteration (Lemma \ref{lem:iteration}), analogously to Section \ref{subsec:master}, we can immediately deduce that 
$
\Psi_k(T) \prec 1
$
where $T$ is the time defined in the statement of Proposition \ref{prop:zig}. This concludes the 
proof of Proposition \ref{prop:zig}~(b) for the remaining  case
$\mathfrak{I}_k = \emptyset$. 



\subsection{Modifications for general $\sigma=\E \chi_{\mathrm{od}}^2$. }\label{sec:general}
 The proof of  Proposition~\ref{prop:zig} presented so far  assumed for simplicity that 
 $\sigma =  \E \chi^2_{\mathrm{od}} $ is real  and $\E \chi^2_{\mathrm{d}} = 1+\sigma$.
 We now explain how to prove the general case, when these two restrictions are lifted. 
 The only changes concern the choice of the initial condition and of the evolution $B_t$ in the flow \eqref{eq:OUOUOU}.

 If $\sigma$ is not real,  we modify the evolution in  \eqref{eq:OUOUOU} in such a way the entries of $B_t$ are $\sqrt{t}$ times a standard complex Gaussian, and we modify the initial condition in \eqref{eq:OUOUOU} 
 from $W_0=W$ to $W_0=\widetilde{W}_T$, with another Wigner matrix $\widetilde{W}_T$ prepared such that
\begin{equation}
\label{eq:simpl1}
e^{-T/2}\widetilde{W}_T+\sqrt{1-e^{-T}}U\stackrel{\dd}{=}W. 
\end{equation}
Here  $U$ is a GUE matrix,
 which is independent of $\widetilde{W}_T$ (here we used that $|\sigma|< 1$). 
 We point out that the limiting eigenvalue density of $\widetilde{W}_T$ does not change along the flow \eqref{eq:OUOUOU} as a consequence of the fact that $\E |(W_t)_{ab}|^2$, for $a>b$, is preserved, and only
\[
\E (W_t)_{ab}^2=e^{-t} \E (\widetilde{W}_T)_{ab}^2, \qquad\quad \E (W_t)_{aa}^2=e^{-t/2}\E (\widetilde{W}_T)_{aa}^2
+ \frac{1}{N} \sqrt{1-e^{-t}}\,, \qquad t\in [0, T]\,,
\]
change. The fact that  $\E (W_t)_{ab}^2$ and $\E (W_t)_{aa}^2$ 
do change along the flow contributes to a change of order $1/N$
in the averaged Stieltjes transform of $W_t$; such change is easily seen to 
be negligible for the precision of the local laws we are considering here. If $\sigma\in\R$ but
$\E \chi_{\mathrm{d}}^2 \ne 1+\sigma$,
 similarly to \eqref{eq:simpl1}, we choose $B_t$ so that its entries have variance $t$ times the variance of $W$ for the off--diagonal entries and $\E (B_t)_{aa}^2=(1+\sigma)t$, and we can prepare yet another Wigner matrix $\widehat{W}_T$ such that
\begin{equation}
\label{eq:simpl2}
e^{-T/2}\widehat{W}_T+\sqrt{1-e^{-T}}\widehat{U}\stackrel{\dd}{=}W,
\end{equation}
with $\widehat{U}$ being independent of $\widehat{W}_T$ and having the same entries distribution as $W$ except for the diagonal entries having variance $\E \widehat{U}_{aa}^2=\frac{1}{N}(1+\sigma)$. The second moments of $(\widehat{W}_t)_{ab}$ are preserved and only the diagonal changes 
\[
\E(\widehat{W}_t)_{aa}^2=e^{-t/2}\E (\widehat{W}_T)_{aa}^2+\frac{1}{N}\sqrt{1-e^{-t}}(1+\sigma);
\]
hence the limiting eigenvalue distribution is still given by the semicircular law.

\nc

\section{Green function comparison: Proof of Proposition \ref{prop:zag}} \label{sec:GFT}

In this section, we remove the Gaussian component introduced in Propositions \ref{prop:zig} by a Green function comparison (GFT) argument, i.e.~we prove Proposition \ref{prop:zag}. For simplicity, we will write the detailed proof only in the case of no imaginary parts, i.e.~$\mathfrak{I}_k = \emptyset$ and $\mathfrak{I}_{k+1} = \emptyset$ in the average and isotropic case, respectively. The minor modifications needed for handling the other cases will be briefly discussed in Section \ref{subsec:withIM} below.

Before entering the proof, we point out that  typical GFT arguments (starting from \cite{TVActa}) are used to
compare the distribution of a genuinely fluctuating observable under two different matrix ensembles
whose single entry distributions have matching  first few moments. Technically, a family of 
interpolating ensembles is constructed which may be finite (e.g. Lindeberg replacement strategy)
or continuous (e.g. along an Ornstein-Uhlenbeck flow) and the change of the  distribution 
in question is closely monitored along the interpolation.
In this standard setup for GFT, however, local laws 
 serve as \emph{a priori} bounds 
obtained by independent methods and they assumed to  hold for all interpolating ensembles in between.
In other words,
concentration--type information about  resolvents $G(z)$ with $\Im z$ well above the eigenvalue spacing
are turned into information on the distribution of $G(z)$ with $\Im z$ at, or even slightly below
the eigenvalue spacing. 
Our application of GFT is different in spirit, since we aim to prove local laws
for one ensemble knowing them for the other one. Thus GFT needs to be done   \emph{self-consistently}
with monitoring a carefully designed quantity that satisfies
a Gronwall-type inequality along the interpolation. 

We remark that more than ten years ago Knowles and Yin  in \cite{KnowYin2}  used GFT in 
a similar spirit to prove  single resolvent local law for ensembles
where the deterministic approximation $M$ to $G$ is not a multiple of identity matrix
(for example deformed Wigner matrices). Later a much more direct and generally applicable
alternative method based
upon the matrix Dyson equation \cite{AEK1, AEKS} has been developed
to prove such local laws without GFT.  Our current dynamical approach revives
the idea of a self-consistent  GFT, since it naturally serves as a counterpart of the characteristic flow
to remove the Gaussian component added along that flow. In fact, the approach of \cite{KnowYin2} 
also used a tandem of gradual reduction of $\eta=\Im z$ (called \emph{bootstrapping} steps) and a self-consistent GFT 
(called \emph{interpolation} steps),  see Fig.~1.1 in~\cite{KnowYin2}. 
However, the bootstrapping step
in~\cite{KnowYin2} was much less effective than the characteristic flow which does the $\eta$-reduction in 
one step even for a much more complex multi-resolvent chain.
   In the GFT step, we  use the  simple entry-by-entry Lindeberg replacement strategy that is better adjustable to 
   our complicated resolvent chains
instead  of a special continuous interpolation as in~\cite{KnowYin2}, but  the core of both techniques  is
a self-consistent Gronwall argument.
The main technical  challenge in our proof  is that the error in one step of the Lindeberg replacement
is not always sufficiently small, but by carefully monitoring the 
errors in each step,  we  gain from  summing them up explicitly. We will explain this 
mechanism in Example~\ref{ex:ave}.

Now we turn to the actual proof. 
 Recalling the notations 
 \begin{equation} \label{eq:parameters}
 	\eta := \min_i |\Im z_i| \quad \text{and} \quad \rho:= \pi^{-1} \max_i |\Im m_i|\,,
 \end{equation}
 we begin by  distinguishing the \emph{averaged} and \emph{isotropic} control quantities 
 \begin{align} \label{eq:Psikav}
 	\Psi_k^{\rm av} &  := \frac{\sqrt{N \eta}}{N^{k/2-1} \sqrt{\rho}} \big| \big\langle \big(G_1 A_1 ... G_k - M_{[1,k]}\big)A_k \big\rangle\big| \\ \label{eq:Psikiso}
 	\Psi_k^{\rm iso} (\bm x, \bm y) &:= \frac{\sqrt{N\eta}}{N^{k/2} \sqrt{\rho}}  \big|  \big(G_1 A_1 ...  A_k G_{k+1} - M_{[1,k+1]}\big)_{\bm{x}\bm{y}} \big|\,, 
 \end{align}
 where $\bm x , \bm y \in \C^N$ are unit deterministic vectors and the traceless matrices $A_i \in \C^{N \times N}$ are assumed to have normalized Hilbert-Schmidt norms, $\langle |A_i|^2 \rangle^{1/2} = 1$. Recall that, in \eqref{eq:Psikav}--\eqref{eq:Psikiso}, we only consider chains without $\Im G$'s, the more general cases will be discussed later in Section~\ref{subsec:withIM}. Finally, we point out that our notation in \eqref{eq:Psikav}--\eqref{eq:Psikiso} already suppressed the dependence on the spectral parameters and deterministic matrices, since the sets of these are considered fixed along the argument. In the following, we will often say that an estimate on $\Psi$ holds \emph{uniformly}, by which we will always mean uniformity in all unit deterministic vectors and all choices of subsets of spectral parameters and deterministic matrices as explained in Proposition \ref{prop:zag}~(a). 
 
 Now, the goal of this section is to prove Proposition \ref{prop:zag}. More precisely, we will show that, if the optimal multi-resolvent local laws  
 \begin{equation} \label{eq:multiG}
 	\Psi_k^{\rm av} + \Psi_k^{\rm iso} \prec 1, \quad \text{for all fixed} \quad k \in \N,
 \end{equation}
 hold \emph{uniformly} for a Wigner matrix with some given single entry distributions, then they also hold
  for every other Wigner matrix with different single entry distributions, again \emph{uniformly}, provided that 
  the \emph{first three moments} of the  entries of these two ensembles 
 \emph{match}. A fundamental input for our proof is that the corresponding single resolvent local laws hold for \emph{every} Wigner matrix ensemble~\cite{EYY2012, KnowYin, BEKYY}, i.e.~the following Green function comparison argument is not needed for them. 
 \begin{theorem}
 For fixed $\epsilon > 0$, we have
\begin{equation} \label{eq:singleGoptimal}
	|\langle G -m \rangle| \prec \frac{1}{N \eta} \,, \qquad \big|\big(G-m\big)_{\bm{x} \bm{y}}\big| \prec \sqrt{\frac{\rho}{N \eta}} + \frac{1}{N \eta}
\end{equation}
uniformly in unit deterministic vectors $\bm x, \bm y$ and at spectral parameter $z \in \C \setminus \R$ with $\eta=|\Im z| \ge N^{-1+\epsilon}$ and $\Re z \in \R$, where $\rho = \pi^{-1} |\Im m(z)|$. 
 \end{theorem}
 For convenience, these single resolvent laws will be expressed in the compact form
 \begin{equation*}
 	\Psi_0^{\rm av} + \Psi_0^{\rm iso} \prec 1\,,
 \end{equation*}
 which extends \eqref{eq:Psikav}--\eqref{eq:Psikiso} when no  traceless matrices $A$ are
 present (see, e.g., \cite{multiG, A2}). 
 
 Before starting the proof, we recall some notation which has already been used in the statement of Proposition~\ref{prop:zag}. 
 We will distinguish between the two ensembles compared in the GFT argument by using different letters, $v_{ab}$ and $w_{ab}$, for their matrix elements, and we shall occasionally use the notation $H^{(\bm{v})}$ and $H^{(\bm{w})}$ to indicate the difference. Alternatively, one could denote the matrix elements by a universal letter $h_{ab}$ and distinguish the two ensembles in the underlying measure, especially in the expectations $\E_{\bm{v}}$ and $\E_{\bm{w}}$. However, since the proof of Proposition \ref{prop:zag} works by replacing the matrix elements one-by-one in $N(N+1)/2$ steps, we use the first notation, analogously to \cite[Section 16]{EYbook}.

 \subsection{Preliminaries}
 The principal idea of the proof is as follows: First, we fix a bijective ordering 
 \begin{equation}
 	\phi : \{(i,j) \in[N]^2 : i \le j \} \to [\gamma(N)]\,, \qquad \gamma(N):= \frac{N(N+1)}{2}
 \end{equation}
 on the index set of independent entries of a Wigner matrix. Then, according to the induced ordering, 
 the matrix elements are swapped one-by-one from the distribution $v_{ab}$ to $w_{ab}$ in $\gamma(N) \sim N^2$ steps. 
 In particular, at step $\gamma \in \{0\}\cup[\gamma(N)]$ in this replacement procedure, the resulting
  matrix $H^{(\gamma)}$ has entries which are distributed according to $w_{ij}$ whenever
   $\phi\big((i,j)\big) \le \gamma$ and according to $v_{ij}$ whenever $\phi\big((i,j)\big) > \gamma$, 
   i.e.~$H^{(0)} = H^{(\bm{v})}$ and $H^{(\gamma(N))} = H^{(\bm{w})}$. This one-by-one replacement 
   of the matrix elements naturally requires understanding the \emph{isotropic} law \eqref{eq:zagmultiGISO}, 
   as already indicated in \eqref{eq:Psikiso}. 
 
 In order to derive \eqref{eq:multiG} also for $H^{(\bm{w})}$, we compute high moments of $\Psi_k^{\rm av/iso}$ for $H^{(\gamma)}$ and $H^{(\gamma-1)}$ for general $\gamma \in [\gamma(N)]$ and compare the results. Given sufficiently good one-step bounds, a telescopic argument will yield the estimate \eqref{eq:multiG} also for $H^{(\bm{w})}$. These "sufficiently good" one-step bounds are essentially required to accommodate the large number $O(N^2)$ of necessary replacements in order to arrive at $H^{(\gamma(N))}$. 
 A key feature of our proof, in contrast to previous applications of the replacement strategy, is 
 that the error will not always  be $o(N^{-2})$ in each step, but their cumulative size after summation 
  is still $o(1)$.
 
 The proof of Proposition \ref{prop:zag} is divided in two main parts: At first, in \textbf{Part (a)}, Section \ref{subsec:iso}, we show the isotropic part of \eqref{eq:multiG}, that is  $\Psi_k^{\rm iso} \prec 1$, 
 via a double induction on the number $k \in \N$ of traceless matrices and the moment 
 $p \in \N$ taken of $\Psi_k^{\rm iso}$, i.e.~$\E|\Psi_k^{\rm iso}|^p$. Thereby, we crucially use
  that the $\prec$-bound is (essentially) equivalent to controlling arbitrarily high moments up to 
  an $N^\xi$-error with arbitrarily small $\xi> 0$. Afterwards, in \textbf{Part~(b)}, Section \ref{subsec:av},
   using Part~(a) as an input, we will demonstrate $\Psi_k^{\rm av} \prec 1$ (and thus conclude the 
   proof of Proposition \ref{prop:zag} for $\mathfrak{I}_{k+1}=\emptyset$~resp.~$\mathfrak{I}_k=\emptyset$) 
   for every fixed $k$ via a single induction on the moment $p$. The main reason for this order of the 
   argument is  that the one-by-one replacement in step $\gamma$ is conducted
    via resolvent expansion focusing on the differing matrix entries at positions $(i,j) = \phi^{-1}(\gamma)$ 
    and $(j,i)$, and thereby it naturally produces isotropic quantities (see Lemma \ref{lem:resolexp} below). 
    Hence, the argument for $\Psi_k^{\rm av}$ cannot be self-contained and must rely on $\Psi_k^{\rm iso}$, 
    which in fact will not involve the averaged local laws at all. 
 
We fix some further notation. We have an initial Wigner matrix $H^{(0)}:= H^{(\bm{v})}$ and iteratively define 
 \begin{equation}
 	H^{(\gamma)} := H^{(\gamma-1)} - \frac{1}{\sqrt{N}}\Delta_V^{(\gamma)} + \frac{1}{\sqrt{N}}\Delta_W^{(\gamma)},
 \end{equation}
 a sequence of Wigner matrices for $\gamma \in [\gamma(N)]$, where we denoted\footnote{Observe that in this normalization, the non-zero entries of $\Delta_V^{(\gamma)}$ and $\Delta_W^{(\gamma)}$ are of order one random variables.}
 \begin{align} \label{eq:Deltas}
 	\Delta_V^{(\gamma)}:= \sqrt{N}\frac{E^{(ij)} (H^{(\bm{v})})_{ij} + E^{(ji)}  (H^{(\bm{v})})_{ji} }{1 + \delta_{ij}} \quad \text{and} \quad \Delta_W^{(\gamma)}:= \sqrt{N}\frac{E^{(ij)}(H^{(\bm{w})})_{ij} + E^{(ji)} (H^{(\bm{w})})_{ji} }{1 + \delta_{ij}} \,. 
 \end{align}
 Here, $\phi\big((i,j)\big) = \gamma$ and $E^{(ij)}$ denotes the matrix whose matrix elements are zero everywhere except at position $(i,j)$, i.e. $(E^{(ij)})_{k\ell} = \delta_{i k} \delta_{j \ell}$. The denominator $1 + \delta_{ij}$ is introduced to 
 account for the factor of two in the numerator occurring for diagonal indices. Note that $H^{(\gamma)}$ and $H^{(\gamma-1)}$ differ only in the $(i,j)$ and $(j,i)$ matrix elements, and they can be written as
 \begin{align}
 	H^{(\gamma-1)} = \widecheck{H}^{(\gamma)} + \frac{1}{\sqrt{N}} \Delta_V^{(\gamma)} \quad \text{and} \quad H^{(\gamma)} = \widecheck{H}^{(\gamma)} + \frac{1}{\sqrt{N}} \Delta_W^{(\gamma)}
 \end{align}
 with a matrix $\widecheck{H}^{(\gamma)}$ whose matrix element is zero at the $(i,j)$ and $(j,i)$ positions. Similarly, we denote the corresponding resolvents at spectral parameter $z_j \in \C \setminus \R$ by
 \begin{equation}\label{GG}
 	G_j^{(\gamma)} := (H^{(\gamma)} - z_j)^{-1}\,, \quad G_j^{(\gamma-1)} := (H^{(\gamma-1)} - z_j)^{-1}\,, \quad \text{and} \quad  \widecheck{G}_j^{(\gamma)} := (\widecheck{H}^{(\gamma)} - z_j)^{-1}\,.
 \end{equation}
 
 Observe that, at each step $\gamma$ in the replacement procedure, the deterministic approximation to a resolvent chain involving $G^{(\gamma)}$ is the same. This is because only the first two moments of the matrix elements of $H^{(\gamma)}$ determine this approximation, symbolically denoted by $M$, via the \emph{Matrix Dyson Equation (MDE)}, see, e.g., \cite{MDEreview}. For a chain in the checked resolvents $\widecheck{G}$, the approximating $M$ is \emph{in principle} differing from the non-checked ones, simply because the self-energy operator $\widecheck{\mathcal{S}}^{(\gamma)}[R] = \E[\widecheck{H}^{(\gamma)} R \widecheck{H}^{(\gamma)}]$ associated with $\widecheck{H}^{(\gamma)}$ 
 is no longer exactly the averaged trace $\langle \cdot \rangle$. However, since this discrepancy introduces an error of size $1/N$ in the MDE, which is a stable equation, this will not be visible in the local laws \eqref{eq:multiG}. Therefore, we shall henceforth ignore this minor point
  and shall just 
  define the normalized differences 
 \begin{equation*}
 	\Psi_k^{{\rm av}, (\gamma)}\,,  \quad \widecheck{\Psi}_k^{{\rm av}, (\gamma)} \,, \quad \Psi_k^{{\rm iso}, (\gamma)} (\bm x, \bm y)\,,  \quad \text{and} \quad \widecheck{\Psi}_k^{{\rm iso}, (\gamma)} (\bm x, \bm y) \,,
 \end{equation*}
 exactly as in \eqref{eq:Psikav}--\eqref{eq:Psikiso}, but with $G_j$ replaced by $G_j^{(\gamma)}$ and $\widecheck{G}_j^{(\gamma)}$, respectively. We emphasize again that the deterministic counterparts in all of the normalized differences are the \emph{same}. 
 
 We can now turn to the actual proof.

 \subsection{Part (a): Proof of the isotropic law} \label{subsec:iso} In this first part, we exclusively work with isotropic quantities and we shall hence drop the superscript $^{\rm iso}$ in the entire  Section~\ref{subsec:iso}.
 As already mentioned above, we shall prove the claim by a \emph{double induction} on $k$ and the moment $p$ taken of $\Psi_k$, i.e.~$\E |\Psi_k|^p$.  
 
 Thereby, the primary induction parameter is $k$ and our goal is to show that, if for some $k \in \N$ we have 
 \begin{equation}
 	\label{eq:ind0}
 \max_{\gamma \le \gamma(N)}\Psi^{(\gamma)}_{k'} + \max_{\gamma \le \gamma(N)}\widecheck{\Psi}^{(\gamma)}_{k'} \prec 1\,,
 \qquad \forall \, k' \in \{0, ... , k-1\}\,,
  \end{equation}
 then also
 \begin{equation}
 	\label{eq:ind1}
 	\max_{\gamma \le \gamma(N)}\Psi^{(\gamma)}_{k} + \max_{\gamma \le \gamma(N)}\widecheck{\Psi}^{(\gamma)}_{k} \prec 1\,.
 \end{equation} 
 Within the proof of \eqref{eq:ind1}, for a fixed $k$, we will then crucially use that the $\prec$-bound is equivalent to controlling arbitrarily high moments $\E |\Psi_k|^p$ up to an $N^\xi$-error for an arbitrarily small $\xi > 0$. Therefore, we use another secondary induction on the moment $p$. More precisely, in order to establish \eqref{eq:ind1}
 from~\eqref{eq:ind0},
  our goal is to show that,  for any  fixed  $k\in \N$, \nc  if for some $p \in\N $ we have that
  $$
 \max_{\gamma \le \gamma(N)}\big\Vert \Psi^{(\gamma)}_{k} \big\Vert_{p-1} +  \max_{\gamma \le \gamma(N)}\big\Vert \widecheck{\Psi}^{(\gamma)}_{k} \big\Vert_{p-1} \lesssim N^\xi 
 $$
 for any $\xi>0$, 
 then also  
  \begin{equation}\label{eq:conclusion}
 \max_{\gamma \le \gamma(N)}\big\Vert \Psi^{(\gamma)}_{k} \big\Vert_{p} +  \max_{\gamma \le \gamma(N)}\big\Vert \widecheck{\Psi}^{(\gamma)}_{k} \big\Vert_{p} \lesssim N^\xi \,
 \end{equation}
 holds for any $\xi>0$, 
 where implicit constants  depend on $k,p$ and  $\xi$.
   Here for a random variable $X$ we used the definition 
 $\Vert X\big\Vert_{p}:=[\E |X|^p]^{1/p}$.

 To summarize, as the \emph{induction hypothesis},
  given some arbitrary fixed $p,k \in \N$, we will assume that 
 \begin{equation} \label{eq:inductionhypo}
  	\max_{\gamma \le \gamma(N)}\Psi^{(\gamma)}_{k'} + \max_{\gamma \le \gamma(N)}\widecheck{\Psi}^{(\gamma)}_{k'} \prec 1 \quad \text{and} \quad  \max_{\gamma \le \gamma(N)}\big\Vert \Psi^{(\gamma)}_{k} \big\Vert_{p-1} + 
	 \max_{\gamma \le \gamma(N)}\big\Vert \widecheck{\Psi}^{(\gamma)}_{k} \big\Vert_{p-1} \le C_{k,p,\xi} N^\xi 
 \end{equation}
 hold uniformly for all ${k' \in \{0,...,k-1\}}$ and $\xi > 0$ with an appropriate $N$-independent constant.  Then we will conclude~\eqref{eq:conclusion}.
 
 The overall \emph{base case} ($k=1$, $p=1$)  is easy to verify: it
 solely consists of the usual isotropic law (the first estimate in \eqref{eq:inductionhypo} for $k'=0$) and the trivial bound $\E | \Psi_{1}|^0 = 1$ (the second estimate in \eqref{eq:inductionhypo} for $k=1$ and $p=1$).

 We start with two arbitrary but fixed bounded deterministic vectors $\Vert \bm x \Vert, \Vert \bm y \Vert \lesssim 1$ and introduce the set 
 \begin{equation} \label{eq:Ixy}
 	I_{\bm x \bm y} := \{\bm x, \bm y\} \cup \{{\bm e}_a : a \in [N]\} \subset \C^N
 \end{equation}
 of vectors, which will naturally arise along the argument (see \eqref{eq:vectormax} below), where ${\bm e}_a$ denotes the standard basis vector in the coordinate direction $a$. Note that the cardinality of $I_{\bm x \bm y}$ is 
 $N+2$. After defining\footnote{Here, $p$ is a superscript, not a power.}  
 \begin{equation}
 	\Omega_k^p(\gamma) := \max_{\bm u, \bm v \in I_{\bm x \bm y}} \Vert {\Psi}_k^{(\gamma)}(\bm u, \bm v) \Vert_p^p\,
 \end{equation}
 (we omitted the dependence on $\bm x, \bm y$ in the notation, as they are considered fixed along the whole
 argument), 
 the principal goal of the induction step is to prove the following proposition. 
 \begin{proposition}[Gronwall estimate] \label{prop:gronwall}  Fix $p, k \in \N$ and assume~\eqref{eq:inductionhypo}
 holds.
 	Then, for any $\xi > 0$, there exist some constants $C_1, C_2 >0 $ (depending on $p$, $k$, and $\xi$, but independent of $N$, $\bm x$, and $\bm y$)  such that 
 	\begin{equation} \label{eq:gronwall}
 		\Omega_k^p(\gamma_0) \le C_1 \frac{1}{N^2} \sum_{\gamma < \gamma_0} \Omega_k^p(\gamma) + C_2 N^\xi 
 	\end{equation}
 	for every $\gamma_0 \in [\gamma(N)]$. 
 \end{proposition}
 Note that \eqref{eq:gronwall} is a discrete Gronwall inequality for $\Omega_k^p(\gamma)$. Hence, having Proposition \ref{prop:gronwall} at hand (note that, in particular, $\Omega_k^p(0) \le C_2 N^\xi$), we obtain
 \begin{equation} \label{eq:gronwallconclude}
 	\max_{\gamma \le \gamma(N)} \Omega_k^p(\gamma) \le C_2 \mathrm{e}^{C_1} N^\xi \le C_3(k,p,\xi) N^\xi\,, 
 \end{equation}
uniformly in $\bm x$ and $\bm y$ and all choices of spectral parameters and traceless deterministic matrices, which then  implies the $\Psi$-part of~\eqref{eq:conclusion}. 
    In the next subsections we present auxiliary results necessary for the proof of Proposition \ref{prop:gronwall}
    which will then be concluded in Section~\ref{sec:Propgron}. The $\widecheck{\Psi}$-part of \eqref{eq:conclusion} and thus the induction step will finally be
    completed in Section~\ref{sec:complete}.
    
  In order to simplify notation, we shall henceforth drop the subscripts for all resolvents and deterministic matrices, i.e.~write $G_j =G$ and $A_j =A$ instead. 
  
 \subsubsection{Preliminaries}
 The fundamental building block of our proof is the following elementary lemma on resolvent expansion. 
 Note that we need to express  $G^{(\gamma-1)}, G^{(\gamma)}$ in terms of the "unperturbed" resolvent
 $\widecheck{G}^{(\gamma)}$ of $\widecheck{H}^{(\gamma)}$ that has zero elements in the $\gamma$-th position,
 and conversely, we need to express $\widecheck{G}^{(\gamma)}$ in terms of both "perturbed" resolvents
 using $\Delta_V^{(\gamma)}$ and $\Delta_W^{(\gamma)}$ from \eqref{eq:Deltas} as perturbations,
 see~\eqref{GG}.
 We work with finite resolvent expansions up to some order $m$, independent of $N$,
  to be determined later. The last
 term therefore always contains the original resolvent as well and it will have to be estimated deterministically
 by its norm but if $m$ is large enough this will be affordable. 
 
 \begin{lemma}[Resolvent expansions] \label{lem:resolexp}
 	For every fixed $m\in \N$, it holds that 
 	\begin{subequations}
 		\begin{equation} \label{eq:forwardexp}
 			\widecheck{G}^{(\gamma)} =  \sum_{\ell = 0}^{m} N^{-\ell/2} \big( G^{(\gamma)} \Delta_W^{(\gamma)} \big)^\ell G^{(\gamma)} + N^{-(m+1)/2}  \big( G^{(\gamma)} \Delta_W^{(\gamma)} \big)^{m+1} \widecheck{G}^{(\gamma)}
 		\end{equation}
 		and
 		\begin{equation} \label{eq:backwardexp}
 			{G}^{(\gamma)} =  \sum_{\ell = 0}^{m} (-1)^\ell N^{-\ell/2} \big(\widecheck{G}^{(\gamma)} \Delta_W^{(\gamma)} \big)^\ell \widecheck{G}^{(\gamma)} + (-1)^{(m+1)} N^{-(m+1)/2}  \big( \widecheck{G}^{(\gamma)} \Delta_W^{(\gamma)} \big)^{m+1} {G}^{(\gamma)}\,. 
 		\end{equation}
 	\end{subequations}
 	These relations also hold verbatim when replacing $G^{(\gamma)} \to G^{(\gamma-1)}$
	 and $\Delta_W^{(\gamma)} \to \Delta_V^{(\gamma)}$. \qed
 \end{lemma}
 
 We now expand each $G^{(\gamma)}$ in 
 \begin{equation} \label{eq:gamma}
 	\big|\Psi_k^{(\gamma)}(\bm x, \bm y)\big|^p = \left(\frac{N \eta}{\rho}\right)^{p/2} N^{-pk/2}  \big| \big( (G^{(\gamma)} A)^k G^{(\gamma)} - M_{[1,k+1]}\big)_{\bm x \bm y} \big|^p
 \end{equation}
 and each $G^{(\gamma-1)}$ in 
 \begin{equation} \label{eq:gamma-1}
 	\big|\Psi_k^{(\gamma-1)}(\bm x, \bm y)\big|^p = \left(\frac{N \eta}{\rho}\right)^{p/2} N^{-pk/2}  \big| \big( (G^{(\gamma-1)} A)^k G^{(\gamma-1)} - M_{[1,k+1]}\big)_{\bm x \bm y} \big|^p
 \end{equation}
 according to \eqref{eq:backwardexp} (for some $m \ge 4$ to be determined below, depending on $p$ and $k$; see \eqref{eq:truncation}) and sort the resulting terms by their power $r = 0,1,2,...$ of $N^{-1/2}$.  
 Then we take the expectation with respect to $w_{ij}$ and $v_{ij}$, respectively (recall that $\phi\big((i,j)\big) = \gamma$),
 and  use the moment matching condition \eqref{eq:momentmatch}.
  As a result, we find that the terms with a prefactor $N^{-r/2}$ for $r = 0, 1,2,3$ are algebraically \emph{exactly the same} for both \eqref{eq:gamma} and~\eqref{eq:gamma-1}.  The conclusion of this argument
  is formalized in the following lemma. 
 \begin{lemma} \label{lem:firstthree}  For any fixed $(i,j) \in [N]^2$ with $i \le j$ and  $\gamma=\phi(i,j)$  we have that
 	\begin{align} \label{eq:gammaexp}
 		\E_{w_{ij}} \big|\Psi_k^{(\gamma)}(\bm x, \bm y)\big|^p = \sum_{r= 0}^3 N^{-r/2} \alpha_{k, r}^{(\gamma)}(\bm x, \bm y)\big|\widecheck{\Psi}_k^{(\gamma)}(\bm x, \bm y)\big|^{p-r} + \text{\rm higher order terms} \\[2mm] \label{eq:gamma-1exp}
 		\E_{v_{ij}} \big|\Psi_k^{(\gamma-1)}(\bm x, \bm y)\big|^p = \sum_{r= 0}^3 N^{-r/2} \alpha_{k, r}^{(\gamma)}(\bm x, \bm y)\big|\widecheck{\Psi}_k^{(\gamma)}(\bm x, \bm y)\big|^{p-r} + \text{\rm higher order terms}
 	\end{align}
 	for some \emph{identical} coefficients $\alpha_{k, r}^{(\gamma)}(\bm x, \bm y)$ independent of $v_{ij}$ and $w_{ij}$
 	whose precise values are (mostly) irrelevant. Here  "higher order terms" denote terms with prefactor $N^{-r/2}$ 
	with $r\ge 4$.  
 \end{lemma}
 In the following Sections \ref{subsec:fourthorder}--\ref{subsec:truncation}, preparing the conclusion of the proof of Proposition \ref{prop:gronwall} in Section~\ref{sec:Propgron}, we will 
 discuss the higher order terms in \eqref{eq:gammaexp} and \eqref{eq:gamma-1exp}. These have to be estimated  individually  by size when we will consider the difference
 of \eqref{eq:gammaexp} and \eqref{eq:gamma-1exp}.
 Recall that, we will eventually compare $\Psi_k^{(0)}(\bm x, \bm y)$ and $\Psi_k^{(\gamma(N))}(\bm x, \bm y)$ in $\gamma(N) = O(N^2)$ many steps, which is why the higher order terms must all be bounded by $1/N^2$, roughly said. More precisely, we will use the following telescopic summation: For every $\gamma_0 \in [\gamma(N)]$, it holds that 
 \begin{equation} \label{eq:telescope}
 	\left| \Vert \Psi_k^{(\gamma_0)}(\bm x, \bm y) \Vert_p^p - \Vert \Psi_k^{(0)}(\bm x, \bm y) \Vert_p^p \right| \le \sum_{1 \le \gamma \le  \gamma_0}	\left| \Vert \Psi_k^{(\gamma)}(\bm x, \bm y) \Vert_p^p - \Vert \Psi_k^{(\gamma-1)}(\bm x, \bm y) \Vert_p^p \right|\,. 
 \end{equation}

In the next Section~\ref{subsec:fourthorder},
we will explain the term with $r=4$ in Lemma \ref{lem:firstthree}, i.e.~with $N^{-2}$-prefactor, in detail. 
 All other higher order terms with $r\ge 5$  but still involving only the resolvent $\widecheck{G}^{(\gamma)}$ are completely analogous, in fact easier (see Section \ref{subsec:lowerorder} later for some detail). 
 Afterwards, in Section \ref{subsec:truncation}, we will discuss, how the maximal order $m$ of the resolvent expansion \eqref{eq:backwardexp} has to be chosen in order to accommodate the remainder term involving a non-checked resolvent $G^{(\gamma)}$ (resp.~$G^{(\gamma-1)}$). 
  
  Throughout the following argument we shall focus on the higher order terms in \eqref{eq:gammaexp}, the treatment of \eqref{eq:gamma-1exp} is exactly the~same. Whenever it does not lead to confusion, we shall henceforth drop the superscript $\gamma$.

 \subsubsection{Fourth order terms in Lemma \ref{lem:firstthree}} \label{subsec:fourthorder}
The goal of the current Section \ref{subsec:fourthorder} is to show that the terms of order $r=4$ arising in the telescopic summation \eqref{eq:telescope} can be bounded by the rhs.~of \eqref{eq:gronwall}.

 In the following, we denote (cf.~\eqref{eq:Deltas})
 \begin{equation} \label{eq:Delta}
 	\Delta = \Delta^{(\gamma)} = \frac{E^{(ij)} + E^{(ji)}}{1 + \delta_{ij}}
 \end{equation}
 and find, similarly to \eqref{eq:Deltas}, after taking the full expectation, the $r=4$ (i.e.~$1/N^{2}$) prefactor of the higher order terms in \eqref{eq:gammaexp} to be bounded by (a constant times)
 \begin{equation} \label{eq:fourthorder}
 	\E \sum_{d=1}^{4 \wedge p} \big|\widecheck{\Psi}_k(\bm x, \bm y)\big|^{p-d} \left(\frac{N \eta}{\rho}\right)^{d/2}  N^{-dk/2}  \sum_{4 \Delta \,  \leadsto \, d } \big|  \underbrace{\big(... \Delta... \Delta ...  \big)_{\bm{x} \bm{y}}\ldots \big(... \Delta... \big)_{\bm{x} \bm{y}}}_{\text{four} \,  \Delta\; \text{in a total} \; d \, \text{chains}} \big|\,. 
 \end{equation}
 Here $d$ counts the number of formerly "intact" resolvent chains $\big( (\widecheck{G} A)^k \widecheck{G} \big)_{\bm x \bm y}$, which have been `destroyed' by at least one replacement $\widecheck{G} \to \widecheck{G} \Delta \widecheck{G}$ due to the expansion \eqref{eq:backwardexp}. The symbol 
 \begin{equation} \label{eq:summation}
 	\sum_{4 \Delta \,  \leadsto \, d }
 \end{equation}
 indicates that we sum over all possibilities to destroy exactly $d$ chains by four $\Delta$'s. 
 Note that a chain may be "destroyed" by more than one $\Delta$, therefore $d$ may
 be less than four. After using the explicit form of $\Delta$, altogether we arrive at 
 a finite sum of $4+d$ chains. 
 
 \begin{example}
 	For example, for $d=1$ we have that 
 	\begin{equation} \label{eq:d=1example}
 		\begin{split}
 			&\sum_{4 \Delta \,  \leadsto \, 1 } \big|  \big(... \Delta... \Delta...\Delta...\Delta... \big)_{\bm{x} \bm{y}}\big| \\
 			&= \sum_{\substack{k_1, ... , k_5 \ge 0 : \\
 					\sum_{l} k_l = k}}\big| \big( (\widecheck{G} A)^{k_1} \widecheck{G} \Delta (\widecheck{G} A)^{k_2} \widecheck{G} \Delta (\widecheck{G} A)^{k_3} \widecheck{G} \Delta (\widecheck{G} A)^{k_4} \widecheck{G} \Delta (\widecheck{G} A)^{k_5} \widecheck{G}\big)_{\bm x \bm y} \big| \\
 			&= \sum_{\substack{k_1, ... , k_5 \ge 0 : \\
 					\sum_{l} k_l = k}}\left[\big| \big( (\widecheck{G} A)^{k_1} \widecheck{G}\big)_{\bm x \bm{e}_i}  \big( (\widecheck{G} A)^{k_2} \widecheck{G}\big)_{\bm e_j \bm{e}_j} \big( (\widecheck{G} A)^{k_3} \widecheck{G}\big)_{\bm e_i \bm{e}_i} \big( (\widecheck{G} A)^{k_4} \widecheck{G}\big)_{\bm e_j \bm{e}_j} \big( (\widecheck{G} A)^{k_5} \widecheck{G}\big)_{\bm e_i \bm{y}}  \big| + ... \right]
 		\end{split}
 	\end{equation}
 	with the neglected summands being analogous, only having different distributions of $\bm e_i$ and $\bm e_j$ occurring, which can be produced by the structure of $\Delta$. 
 \end{example}

 For general $d$, in each of the $4+d$ resolvent chains
  in the rhs.~of \eqref{eq:fourthorder}, we now add and subtract the corresponding deterministic $M$-term, $(\widecheck{G}A)^k \widecheck{G}  =  ((\widecheck{G}A)^k \widecheck{G}-M_{k+1}) + M_{k+1}$ (see also \eqref{eq:Mshorthand} below), schematically written as
 $G = (G-M) + M$. 
 In the sequel, we will distinguish the following two complementary cases: 
 \begin{itemize}
 	\item[Case (i):] At least $d$ of the $d+4$ resolvent chains are replaced by their fluctuating part, $G-M$.
 	\item[Case (ii):] At least five of the $d+4$ resolvent chains are replaced by their deterministic counterpart,  $M$. 
 \end{itemize}
 \vspace{2mm}
 \underline{Case (i):} In case (i), we first separate those possibilities from \eqref{eq:summation}, where the destruction of the $d$ chains $\big( (\widecheck{G} A)^k \widecheck{G}\big)_{\bm x \bm y}$ in fact \emph{preserves} $d$ resolvent chains each with $k$ traceless matrices $A$, but with deterministic vectors, which are not $\bm x$ and $\bm{y}$.
This happens when all four $\Delta$'s are placed at the ends of the chains. 
  For example, if $d=1$, we separate these possibilities as 
 \begin{equation} \label{eq:sustainexample}
 	\begin{split}
 		\widecheck{G}_{\bm x \bm e_i} \widecheck{G}_{\bm e_j \bm e_j} \widecheck{G}_{\bm e_i  \bm e_i} \widecheck{G}_{\bm e_j \bm e_j}  \big( (\widecheck{G} A)^k \widecheck{G} \big)_{\bm e_i \bm y} &+ ... \quad \text{or} \\
 		\widecheck{G}_{\bm x  \bm e_i} \widecheck{G}_{\bm e_j  \bm e_j}  \big( (\widecheck{G} A)^k \widecheck{G} \big)_{\bm e_i \bm e_i}  \widecheck{G}_{\bm e_j \bm e_j}  \widecheck{G}_{\bm e_i \bm y}&+ ... \,. 
 	\end{split}
 \end{equation}
 
 In the following, we shall focus on the first exemplary term in \eqref{eq:sustainexample}. Its fluctuating part 
 \begin{equation} \label{eq:G-Mfull}
 	\big( (\widecheck{G} A)^k \widecheck{G} -M_{[1,k+1]}\big)_{\bm e_i \bm y}
 \end{equation}
 can then be paired with the leftover $\big(N \eta /\rho\big)^{1/2} N^{-k/2} $ in  \eqref{eq:fourthorder} and thereby produces a further full $ \big|\widecheck{\Psi}_k^{(\gamma)}(\bm e_j, \bm y)\big|$; the remaining terms coming from a single resolvent in \eqref{eq:sustainexample} are simply estimated by one, 
 \begin{equation} \label{eq:Gbdd1GFT}
|\widecheck{G}_{\bm  u \bm v}|\prec 1\,, \qquad  \bm  u, \bm v \in I_{\bm{x}\bm{y}}   \quad \mbox{cf.~\eqref{eq:Ixy}}\,, 
 \end{equation}
 by the usual isotropic law \eqref{eq:singleGoptimal}. 
 All these terms stemming from \eqref{eq:fourthorder} and constituting a full $ \big|\widecheck{\Psi}_k^{(\gamma)}\big|$ (or $ \big|\widecheck{\Psi}_k^{(\gamma)}\big|^d$ for general $d \in [4 \wedge p]$)
 can then be estimated by 
 \begin{equation} \label{eq:vectormax}
 	\widecheck{\Omega}_k^p(\gamma) := \max_{\bm u, \bm v \in I_{\bm x \bm y}} \Vert \widecheck{\Psi}_k^{(\gamma)}(\bm u, \bm v) \Vert_p^p\,. 
 \end{equation}
 
 Now, after having separated the possibilities from \eqref{eq:summation}, where the destruction preserves $d$ resolvent chains with $k$ deterministic matrices in between, we are left with those which solely create \emph{strictly shorter} chains by the procedure $4 \Delta \leadsto d$. These terms can entirely be treated by our \emph{induction hypothesis} \eqref{eq:inductionhypo}: The power of $\widecheck{\Psi}_k^{(\gamma)}$ has been reduced by (at least) one (cf.~the second estimate in \eqref{eq:inductionhypo}) and $ \widecheck{\Psi}_{k'}^{(\gamma)} + \Psi_{k'}^{(\gamma)} \prec 1$ uniformly in $\gamma$ for \emph{strictly} shorter chains, $k' < k$, has already been shown
 (first estimate in \eqref{eq:inductionhypo}). 
 
 \begin{example}
 	Writing
 	\begin{equation} \label{eq:Mshorthand}
 		M_{j-i+1} \equiv M_{[i,j]} \quad \text{for} \quad 1 \le i < j \le k+1\,, 
 	\end{equation}
 	with a slight abuse of notation, we estimate the $d=1$ term in \eqref{eq:fourthorder} (after having split off the 
 	cases when one of the $k_l$'s equals $k$ and all others are zero in \eqref{eq:sustainexample}) as 
 	\begin{equation} \label{eq:d=1exampleCase1}
 		\begin{split}
 			\E & \big|\widecheck{\Psi}_k(\bm x, \bm y)\big|^{p-1} \left(\frac{N \eta}{\rho}\right)^{1/2} N^{-k/2}\times \\
 			& \ \times \sum_{\substack{0\le k_l \le k-1 : \\
 					\sum_{l} k_l = k}}\Big[\big| \big( (\widecheck{G} A)^{k_1} \widecheck{G} - M_{k_1+1}\big)_{\bm x \bm{e}_i}  \big(M_{k_2+1}\big)_{\bm e_j \bm{e}_j} \big( M_{k_3+1}\big)_{\bm e_i \bm{e}_i} \big( M_{k_4+1}\big)_{\bm e_j \bm{e}_j} \big( M_{k_5+1}\big)_{\bm e_i \bm{y}}  \big|  \\
 			& \quad  + \big| \big( (\widecheck{G} A)^{k_1} \widecheck{G} - M_{k_1+1}\big)_{\bm x \bm{e}_i}  \big((\widecheck{G} A)^{k_2} \widecheck{G} -M_{k_2+1}\big)_{\bm e_j \bm{e}_j} \big( M_{k_3+1}\big)_{\bm e_i \bm{e}_i} \big( M_{k_4+1}\big)_{\bm e_j \bm{e}_j} \big( M_{k_5+1}\big)_{\bm e_i \bm{y}}  \big| + 	... \Big] \\
 			\lesssim \, &N^\xi  \left(\frac{N \eta}{\rho}\right)^{1/2} N^{-k/2}\sum_{\substack{0\le k_l \le k-1 : \\
 					\sum_{l} k_l = k}} \left[\left( \frac{\rho}{N \eta} \right)^{1/2} N^{\sum_l k_l/2}+ \left( \frac{\rho}{N \eta} \right) N^{\sum_l k_l/2} + ... \right]\lesssim N^\xi\,, 
 		\end{split}
 	\end{equation}
 	where analogous summands (i.e.~having further $G-M$ factors instead of $M$, or other arrangements of standard basis vectors $\bm e_i, \bm e_j$ stemming from \eqref{eq:Delta}) are again indicated by dots. In the first estimate, we used that $\big| (M_{j+1})_{\bm u \bm v} \big| \lesssim N^{j/2}$ for all  $\bm u, \bm v \in I_{x,y}$ from Lemma \ref{lem:Mbound}~(b) together with the induction hypothesis \eqref{eq:inductionhypo}. 
 \end{example}

  In the general case, $d \ge 1$, the argument works analogously to the above example: The minimal number of $d$ fluctuating terms carrying an $(\rho/ N \eta)^{1/2}$-factor cancel the leftover $(N\eta/\rho)^{d/2}$-factor in \eqref{eq:fourthorder}. The remaining $N^{k_l/2}$-factors can then be handled by a simple power counting. 
 
 Overall, we find that, all the terms in \eqref{eq:fourthorder} summarized in Case (i), can be bounded by
 \begin{equation} \label{eq:fluctuationfinal}
 	C_1 \widecheck{\Omega}_k^p(\gamma) + C_2 N^\xi 
 \end{equation}
 for some positive constants $C_1, C_2 > 0$, which shall henceforth be used generically, 
 i.e.~their value might change from line to line (but remain uniformly bounded in $\gamma$).
 
 \vskip2mm
  \underline{Case (ii):} For the second case, we recall that all the purely deterministic terms are \emph{independent} of $\gamma$, i.e., as emphasized above, at each replacement step the deterministic approximation to a resolvent chain is the same. However, it is \emph{not} sufficient to just estimate every $M$-term blindly via 
 $\big| (M_{j+1})_{\bm u \bm v} \big| \lesssim N^{j/2}$, as done in \eqref{eq:d=1exampleCase1}. Instead, we need to \emph{gain from the summation} in \eqref{eq:telescope} over all replacement positions. 
  This is the main new element of our proof compared with previous GFT arguments. 
 
 \begin{example}
 	We again look at our $d=1$ example. Using the notation \eqref{eq:Mshorthand}, we find the \emph{trivial estimate} 
 	\begin{equation} \label{eq:d=1exampleCase2}
 		\begin{split}
 			&\E \big|\widecheck{\Psi}_k(\bm x, \bm y)\big|^{p-1} \left(\frac{N \eta}{\rho}\right)^{1/2} N^{-k/2} \sum_{\substack{0\le k_l \le k: \\
 					\sum_{l} k_l = k}}\left[\big| \big( M_{k_1+1}\big)_{\bm x \bm{e}_i}  \big(M_{k_2+1}\big)_{\bm e_j \bm{e}_j} \big( M_{k_3+1}\big)_{\bm e_i \bm{e}_i} \big( M_{k_4+1}\big)_{\bm e_j \bm{e}_j} \big( M_{k_5+1}\big)_{\bm e_i \bm{y}}  \big| + ... \right] \\
 			\lesssim \, &N^\xi  \left(\frac{N \eta}{\rho}\right)^{1/2} N^{-k/2} \sum_{\substack{0\le k_l \le k: \\
 					\sum_{l} k_l = k}} \left[N^{\sum_l k_l/2}+ ... \right]\lesssim N^\xi \left(\frac{N \eta}{\rho}\right)^{1/2}\,, 
 		\end{split}
 	\end{equation}
 	where we again used the induction hypothesis \eqref{eq:inductionhypo} and $\big| (M_{j+1})_{\bm u \bm v} \big| \lesssim N^{j/2}$. This bound is off by a factor $(N \eta/\rho)^{1/2}$, which we will now improve on. 
 	
 	Indeed, the point in \emph{gaining from the summation} is that, although at each individual step $\gamma$, the deterministic terms in \eqref{eq:d=1exampleCase2} might be large, \emph{on average}  over $\gamma$ 
 	their contribution is bounded. More precisely, fixing one constellation of $k_l$'s  in \eqref{eq:d=1exampleCase2} and using $\E \big|\widecheck{\Psi}_k\big|^{p-1} \lesssim N^\xi$ , we find the average of the first line in \eqref{eq:d=1exampleCase2} over all $i,j \in [N]$ to be bounded by (a constant times)
 	\begin{equation} \label{eq:d=1exampleCase2SUM}
 		\begin{split}
 			&N^\xi \left(\frac{N \eta}{\rho}\right)^{1/2} N^{-k/2} \frac{1}{N^2} \sum_{i,j } \left[\big| \big( M_{k_1+1}\big)_{\bm x \bm{e}_i}  \big(M_{k_2+1}\big)_{\bm e_j \bm{e}_j} \big( M_{k_3+1}\big)_{\bm e_i \bm{e}_i} \big( M_{k_4+1}\big)_{\bm e_j \bm{e}_j} \big( M_{k_5+1}\big)_{\bm e_i \bm{y}}  \big| + ... \right] \\
 			\lesssim \, &N^\xi \left(\frac{N \eta}{\rho}\right)^{1/2}  \frac{1}{N^2} \sum_{i,j } \left[ \frac{\big| (M_{k_1+1})_{\bm x \bm e_i} \big|}{N^{k_1/2}}  + ... \right] \\
 			\lesssim \, &N^\xi \left(\frac{N \eta}{\rho}\right)^{1/2}  \frac{1}{N} \sqrt{N} \left[ \frac{\sqrt{\big(|M_{k_1+1}|^2 \big)_{\bm x \bm x} } }{N^{k_1/2}}+ ... \right] 
 			\lesssim \, N^\xi \left(\frac{\eta}{\rho}\right)^{1/2} \lesssim N^\xi\,. 
 		\end{split}
 	\end{equation}
 	To go from the first to the second line, we used $\big| (M_{j+1})_{\bm u \bm v} \big| \lesssim N^{j/2}$ for all but  the first $M$ factor.
	Next, we used a Schwarz inequality for the $i$-summation, 
	which involves the off-diagonal term $(M_{k_1+1})_{\bm x \bm e_i}$:
 	\begin{equation} \label{eq:schwarzfirst}
 		\sum_i \big| (M_{k_1+1})_{\bm x \bm e_i} \big|  \le \sqrt{N} \left(\sum_i \big| (M_{k_1+1})_{\bm x \bm e_i} \big|^2 \right)^{1/2} \le \sqrt{N}\sqrt{ \big(| M_{k_1+1}|^2 \big)_{\bm x \bm x} }\,. 
 	\end{equation}
 	In the penultimate estimate, we used that 
 	\begin{equation} \label{eq:M^2estimate}
 		\sqrt{\big(|M_{j+1}|^2 \big)_{\bm u \bm u} } \lesssim N^{j/2}\,,
 	\end{equation}
 	as follows from the fact that $N^{j/2}$ is in fact the operator norm bound for $M_{j+1}$, and the final estimate in \eqref{eq:d=1exampleCase2SUM} simply used the general fact $\eta/\rho \lesssim 1$.

 	We point out that we even could have gained another $1/\sqrt{N}$-factor from the $i$-summation by not estimating $\big( M_{k_5+1}\big)_{\bm e_i \bm{y}} $ trivially by $N^{k_5/2}$ but using 
 	\begin{equation} \label{eq:schwarzfirst2}
 		\begin{split}
 			\sum_i \big| (M_{k_1+1})_{\bm x \bm e_i} \big( M_{k_5+1}\big)_{\bm e_i \bm{y}}\big| & \le \left(\sum_i \big| (M_{k_1+1})_{\bm x \bm e_i} \big|^2 \right)^{1/2} \left(\sum_i \big| (M_{k_5+1})_{\bm e_i \bm y} \big|^2 \right)^{1/2} \\
 			&\le \sqrt{ \big(| M_{k_1+1}|^2 \big)_{\bm x \bm x} } \sqrt{ \big(| M_{k_5+1}|^2 \big)_{\bm y \bm y}} \,. 
 		\end{split}
 	\end{equation}
 	instead of \eqref{eq:schwarzfirst}. However, we do not need this additional factor $1/\sqrt{N}$ here. 
 	Finally, note that the $j$-summation in \eqref{eq:d=1exampleCase2SUM} would have been useless, since the $j$-terms are diagonal. The summation gain is effective only for off-diagonal terms as in~\eqref{eq:schwarzfirst}.
 \end{example}
 The above example indicates the following general mechanism: After estimating all the $G-M$-type terms with the aid of the induction hypothesis \eqref{eq:inductionhypo}, and estimating the $M$-factors just trivially by their
 size, we are left with an  excess  $(N \eta/\rho)^{u/2}$-factor, for some $u \in [4]$. In order to remove this leftover factor, we need at least $u$ \emph{(collectively) summable bounded $M$-terms} like 
 \begin{equation} \label{eq:Msummable}
 	\frac{\big| (M_{k_1+1})_{\bm x \bm e_i} \big|}{N^{k_1/2}}
 \end{equation}
 in \eqref{eq:d=1exampleCase2SUM} (see also \eqref{eq:M^2estimate}). In fact, each of these collectively summable factors will gain one $1/\sqrt{N}$ compared to the trivial estimate, like the one in \eqref{eq:d=1exampleCase2}. Here, the notion "collective" refers to particular index structures, which allow an effective summation. Denoting terms like \eqref{eq:Msummable} symbolically by $M_{\bm x \bm e_i}$ for brevity, by \emph{(collectively) summable bounded $M$-terms} we mean the following possible index structures
 \begin{equation} \label{eq:Msums}
 	\begin{split}
 		u=1 \, :& \qquad \sum_{i,j} |M_{\bm x\bm e_i}| \quad \text{or} \quad  \sum_{i,j}  |M_{\bm e_j \bm y}| \quad \text{or} \quad ...\\
 		u=2 \, :& \qquad \sum_{i,j} |M_{\bm x\bm e_i}| |M_{\bm e_j \bm y}|\quad \text{or} \quad  \sum_{i,j} |M_{\bm x\bm e_i}| |M_{\bm e_i \bm y}| \quad \text{or} \quad ...\\
 		u=3\, :& \qquad \sum_{i,j} |M_{\bm x\bm e_i}|  |M_{\bm e_i \bm y}||M_{\bm e_j \bm y}|\quad \text{or} \quad  \sum_{i,j} |M_{\bm x\bm e_i}| |M_{\bm e_j \bm y}|^2 \quad \text{or} \quad ... \\
 		u=4\, :& \qquad \sum_{i,j} |M_{\bm x\bm e_i}| |M_{\bm x \bm e_j}| |M_{\bm e_i \bm y}||M_{\bm e_j \bm y}|\quad \text{or} \quad  \sum_{i,j} |M_{\bm x\bm e_i}|^2 |M_{\bm e_j \bm y}|^2 \quad \text{or} \quad ...
 	\end{split}
 \end{equation}
 where dots are always indicating other similar terms, obtained from trivial exchanges $\bm x \leftrightarrow \bm y$ or $i \leftrightarrow j$. 
 
 In principle, every summation over $i$ and $j$ potentially gains a full $1/N$-factor each -- provided that there are enough $M$'s with suitable indices as in  \eqref{eq:Msums}. The existence of $u$ \emph{collectively summable bounded $M$-terms} then ensures that of this potential $1/N^2$-improvement at least
  a $1/N^{u/2}$-gain is effective. 
  More precisely, as an example, for the first column of terms in \eqref{eq:Msums} we have that 
 \begin{equation} \label{eq:MsumsSchwarz} 
 	\begin{split}
 		u=1 \, :& \quad \sum_{i,j} |M_{\bm x\bm e_i}| \le N^{3/2} \left(\sum_i |M_{\bm x\bm e_i}|^2\right)^{1/2}\lesssim N^{2-1/2}\\
 		u=2 \, :& \quad \sum_{i,j} |M_{\bm x\bm e_i}| |M_{\bm e_j \bm y}| \le N \left(\sum_i |M_{\bm x\bm e_i}|^2\right)^{1/2} \left(\sum_j |M_{\bm e_j\bm y}|^2\right)^{1/2}\lesssim N^{2-2/2}\\
 		u=3\, :& \quad \sum_{i,j} |M_{\bm x\bm e_i}|  |M_{\bm e_i \bm y}||M_{\bm e_j \bm y}| \\
 		& \qquad  \quad \le N^{1/2} \left(\sum_i |M_{\bm x\bm e_i}|^2\right)^{1/2} \left(\sum_i |M_{\bm e_i \bm y}|^2\right)^{1/2} \left(\sum_j |M_{\bm e_j\bm y}|^2\right)^{1/2}\lesssim N^{2-3/2} \\
 		u=4\, :& \quad \sum_{i,j} |M_{\bm x\bm e_i}| |M_{\bm x \bm e_j}| |M_{\bm e_i \bm y}||M_{\bm e_j \bm y}| \\
 		&\qquad \quad  \le \left(\sum_i |M_{\bm x\bm e_i}|^2\right)^{1/2} \left(\sum_i |M_{\bm e_i \bm y}|^2\right)^{1/2} \left(\sum_j |M_{\bm e_j\bm y}|^2\right)^{1/2} \left(\sum_j |M_{\bm x \bm e_j}|^2\right)^{1/2}\lesssim N^{2-4/2}
 	\end{split}
 \end{equation}
 by application of Schwarz inequalities like in \eqref{eq:schwarzfirst}--\eqref{eq:schwarzfirst2} and using that $\Vert M \Vert \lesssim 1$. We point out that the $\eta/\rho\le 1$ factor within
  each excess $(N \eta/\rho)^{1/2}$ would not be able to compensate for excess $N$-factors; but 
   the \emph{gains from the summation} are obtained solely on the level of $N$'s.

 It follows from a simple counting \nc argument (or simply by considering all cases directly), that for any $u \in [4]$, we find an appropriately summable index structure within the at least five purely deterministic terms, as in \eqref{eq:Msums}--\eqref{eq:MsumsSchwarz}. Hence, we deduce that
 \begin{equation} \label{eq:fourthorderbound}
 	\eqref{eq:fourthorder} \le C_1 \widecheck{\Omega}_k^p(\gamma) + C_2 N^\xi \bigg( 1 + \sum_{u=1}^{4} \left(\frac{N \eta}{\rho}\right)^{u/2} \left|\big[u \,  \text{sum. bdd.}\, M \text{-terms}\big]^{(\gamma)}_{\bm x, \bm y}\right| \bigg)\,, 
 \end{equation}
 where 
 \begin{equation} \label{eq:Mterms}
 	\big[u\,  \text{sum. bdd.}\, M \text{-terms}\big]^{(\gamma)}_{\bm x, \bm y}
 \end{equation}
 stands symbolically for a product of \emph{$u$ collectively summable bounded deterministic terms}, like \eqref{eq:Msummable}, for which we have just shown the following.
 
  \begin{lemma} \label{lem:sumMs} It holds that
 	\begin{equation} \label{eq:Mestimate}
 		\sum_{\gamma \in [\gamma(N)]}\left|\big[u\,  \text{\rm sum. bdd.}\, M \text{\rm -terms}\big]^{(\gamma)}_{\bm x, \bm y}\right| \lesssim N^{2 -u/2}\,. 
 	\end{equation}
 \end{lemma}
 Combining \eqref{eq:fourthorderbound} with \eqref{eq:Mestimate}, this concludes the argument for the fourth order terms in \eqref{eq:gammaexp}. 
 
 \subsubsection{Further higher order terms in Lemma \ref{lem:firstthree}} \label{subsec:lowerorder}
 Just as in the previous Section \ref{subsec:fourthorder}, the goal of the current Section \ref{subsec:lowerorder} is to show that the terms of order $r\ge 5$ arising in the telescopic summation \eqref{eq:telescope} can be bounded by the rhs.~of \eqref{eq:gronwall}. 
 
 For these other higher order terms in \eqref{eq:gammaexp} with $r \ge 5$ and involving \emph{only} $\widecheck{G}$ (and not ${G}$), the two cases distinguished above for $r=4$ generalize to the following.

 \begin{itemize}
 	\item[Case (i'):] At least $d$ of the $d+r$ resolvent chains are replaced by their fluctuating part, $G-M$.
 	\item[Case (ii'):] At least $r+1$ of the $d+r$ resolvent chains are replaced by their deterministic counterpart,  $M$. 
 \end{itemize}
 For Case (i'), we separate a $1/N^2$-prefactor and find that the remaining part can be estimated by
 \begin{equation} \label{eq:higherorderboundCase1}
 	C_1 N^{-(r-4)/2}\widecheck{\Omega}_k^p(\gamma) + C_2 N^\xi  N^{-(r-4)/2} \,, 
 \end{equation}
 completely analogously to \eqref{eq:fluctuationfinal}. In fact, we gain an additional $N^{-(r-4)/2}\ll 1$ factor
 in both terms.
 This reflects the idea that more $G-M$ terms are better because their presumed bounds carry 
 a factor $(\rho/ N\eta)^{1/2}$ (encoded in 
 the prefactor $(N\eta/\rho)^{1/2}$ in the definition of $\Psi_k^{\rm iso}$ in \eqref{eq:Psikiso}).

 For Case (ii'), we include the additional $N^{-(r-4)/2}$ (after having separated a $1/N^2$-prefactor) into our counting of the leftover $(N \eta/\rho)^{u/2}$-factor (recall the discussion below \eqref{eq:Msummable}). In this way, we find that the maximal number of such leftover factors is $r-(r-4) = 4$. Hence, for every $u \in [4]$, we find an appropriately summable index structure, completely analogously to \eqref{eq:Msums}, and deduce that (leaving out the separated $1/N^2$-prefactor)
 \begin{equation} \label{eq:higherorderboundfinal}
 	\begin{split}
 		r^{\rm th} \, &\text{order term in} \, \eqref{eq:gammaexp} \\
 		&\le	C_1 N^{-(r-4)/2}\widecheck{\Omega}_k^p(\gamma) + C_2 N^\xi \bigg( N^{-(r-4)/2} + \sum_{u=1}^{4} \left(\frac{N \eta}{\rho}\right)^{u/2} \left|\big[u \,  \text{sum. bdd.}\, M \text{-terms}\big]^{(\gamma)}_{\bm x, \bm y}\right|\bigg)\,, 
 	\end{split}
 \end{equation}
 which can be directly incorporated into \eqref{eq:fourthorderbound} after adjusting the constants. 
 Note that while the contributions form Case (i') improve by larger $r$, the 
 terms from Case (ii') that carry many $M$-factors, do not. 
 
  Combining \eqref{eq:higherorderboundfinal} with \eqref{eq:Mestimate}, this concludes the argument for the higher order terms in \eqref{eq:gammaexp}.

 \subsubsection{Truncation of the resolvent expansion} \label{subsec:truncation}
 It remains to discuss the {\it truncation terms}, which involve \emph{not} only $\widecheck{G}$, but also ${G}$, i.e.~the order $m \in \N$ for the truncation of the resolvent expansion \eqref{eq:backwardexp}. Also here, our goal is to show that the contribution of these terms arising in the telescopic summation \eqref{eq:telescope} can be bounded by the rhs.~of \eqref{eq:gronwall}. After expanding each resolvent in \eqref{eq:gamma} via \eqref{eq:backwardexp}, for every fixed $q \ge 1$, we collect those terms which contain the final summand in \eqref{eq:backwardexp} (the \emph{truncation term}), and hence ${G}$ exactly $q$ times. 
For these terms with $q \ge 1$ fixed, we then proceed as follows: Estimate those chains within the truncation term in which ${G}$ appears trivially by norm, $\Vert G \Vert \le 1/\eta$ (note that there are at most $k+1$ resolvents  in such chains
 and we can afford estimating all of them by $1/\eta$ not just the last one $G$) and use $\Vert A \Vert \le \sqrt{N} \langle |A|^2 \rangle^{1/2}$ (recall that we assumed $\langle |A|^2\rangle^{1/2}=1$ around \eqref{eq:Psikav}--\eqref{eq:Psikiso}), and treat the other factors by our induction hypothesis \eqref{eq:inductionhypo} (resulting in an $N^\xi$ factor). 
 
 In this way, we conclude the estimate 
 \begin{equation} \label{eq:truncation}
 	\big[q\, \text{truncation terms}\big] \lesssim N^\xi  \frac{(N\eta/\rho)^{p/2}}{\big(N^{\frac{m+1}{2}}\big)^q} \left(\frac{N^{k/2}}{\eta^{k+1}} \right)^q =  \frac{N^\xi}{N^{2q}} \frac{1}{N^{p(q-1)/2}} \left(\frac{\eta}{\rho}\right)^{p/2}\frac{1}{(N \eta)^{(k+1)q}} \lesssim \frac{N^\xi}{N^2}
 \end{equation}
 when choosing $m = p+3k+5 $, where in the last step we used that $\eta /\rho \lesssim 1$ and $N \eta \gg 1$. We remark that $(N\eta/\rho)^{p/2}$ in \eqref{eq:truncation} comes from the prefactor of $\Psi_k$, $(N^{\frac{m+1}{2}}\big)^{-q}$ from the cumulant order of the truncation terms 
 and $\big(N^{k/2}/\eta^{k+1}\big)^{q}$ from the trivial bounds.

 \subsubsection{Proof of Proposition~\ref{prop:gronwall} }\label{sec:Propgron} 
 As mentioned above \eqref{eq:Delta}, the treatment of the higher order terms in \eqref{eq:gamma-1exp} is identical to our discussion above. Therefore, summarizing Sections \ref{subsec:fourthorder}--\ref{subsec:truncation}, we have proven the following.
 \begin{lemma} \label{lem:onestep} Fix $p,k \in \N$ and assume that the induction hypothesis \eqref{eq:inductionhypo} holds. Then, for every $\gamma \in [\gamma(N)]$, we have that
 	\begin{equation*}
 		\left| \Vert \Psi_k^{(\gamma)}(\bm x, \bm y) \Vert_p^p - \Vert \Psi_k^{(\gamma-1)}(\bm x, \bm y) \Vert_p^p \right| \le \frac{C_1}{N^2} \widecheck{\Omega}_k^p(\gamma) + C_2 \frac{N^\xi}{N^2} \bigg(1 +  \sum_{u=1}^{4} \left(\frac{N \eta}{\rho}\right)^{u/2}  \left|\big[u\,  \text{\rm sum. bdd.}\, M \text{\rm -terms}\big]^{(\gamma)}_{\bm x, \bm y}\right|\bigg)\,,
 	\end{equation*}
 	where $\big[u \,  \text{\rm sum. bdd.}\, M \text{\rm -terms}\big]^{(\gamma)}_{\bm x, \bm y}$ is understood as explained below \eqref{eq:Mterms}.
 \end{lemma}
 Next, employing the telescopic summation from \eqref{eq:telescope} we find that
 \begin{equation} \label{eq:telescopic}
 	\Vert \Psi_k^{(\gamma_0)}(\bm x, \bm y) \Vert_p^p \le  C_1 \frac{1}{N^2} \sum_{\gamma < \gamma_0} \widecheck{\Omega}_k^p(\gamma) + C_2 N^\xi + \frac{N^\xi}{N^2} \sum_{\gamma < \gamma_0} \bigg(  \sum_{u=1}^{4} \left(\frac{N \eta}{\rho}\right)^{u/2}  \left|\big[u \,  \text{\rm sum. bdd.}\, M \text{\rm -terms}\big]^{(\gamma)}_{\bm x, \bm y}\right|\bigg)
 \end{equation}
 after having absorbed $\Vert \Psi_k^{(0)}(\bm x, \bm y) \Vert_p^p$ into $C_2 N^\xi$ by our initial assumption that we have multi-resolvent local laws \eqref{eq:multiG} for the Wigner matrix $H^{(\bf v)} = H^{(0)}$. We are left with discussing the first and last term on the rhs.~of \eqref{eq:telescopic}. 
 
 For the first term, we rely on the following lemma, which says that, in particular, we can replace each $\widecheck{\Omega}_k^p(\gamma)$ in \eqref{eq:telescopic} by $\Omega_k^p(\gamma)$, absorbing the additional error into $C_2$. 
 \begin{lemma} \label{lem:checknochecksim}
 Fix $p, k \in \N$. Then, for every fixed $\gamma \in [\gamma(N)]$, the expressions (recall \eqref{eq:vectormax})
 	\begin{equation*}
 		\Omega_k^p(\gamma)\,, \quad  \Omega_k^p(\gamma-1)\,, \quad  \text{and}  \quad \widecheck{\Omega}_k^p(\gamma)
 	\end{equation*}
 	are comparable up to an additive error of order $N^\xi$ for arbitrarily small $\xi > 0$. 
 \end{lemma}
 \begin{proof}
 	We give a sketch of the simple argument based on Lemma \ref{lem:onestep} in combination with Lemma \ref{lem:firstthree}: Similarly to the proof of Lemma \ref{lem:onestep}, we first expand $G^{(\gamma)}$ (resp.~$G^{(\gamma-1)}$) in $\Vert \Psi_k^{(\gamma)}(\bm x, \bm y) \Vert_p^p$ (resp.~$\Vert \Psi_k^{(\gamma-1)}(\bm x, \bm y) \Vert_p^p$) by means of \eqref{eq:backwardexp} and realize that $\alpha_{k, 0}^{(\gamma)}(\bm x, \bm y) = 1$ in \eqref{eq:gammaexp}--\eqref{eq:gamma-1exp}. The various terms arising in the expansion (now for all $r \ge 1$ and not only for $r \ge 4$) are dealt with as explained in Sections \ref{subsec:fourthorder}--\ref{subsec:truncation}. 
 	
 	However, there is a major simplification, since we do not need to gain from the summation as in Case~(ii) in Section \ref{subsec:fourthorder}: The maximal excess power $u$ of the leftover $(N \eta/\rho)^{1/2}$-factor is bounded by the order $r$ of the expansions in \eqref{eq:gammaexp}--\eqref{eq:gamma-1exp} (simply because at order $r$, there are at most $d=r$ destroyed resolvent chains), such that the characteristic $1/N^{r/2}$-factor at order $r$ balances this excess. Finally, we take a maximum over all $\bm u, \bm v \in I_{\bm x, \bm y}$ for all $\Vert \widecheck{\Psi}_k^{(\gamma)}(\bm u, \bm v) \Vert_p^p$ arising through the expansion (see \eqref{eq:vectormax}). 
 	
 	This finishes the sketch of the proof of Lemma \ref{lem:checknochecksim}. 
 \end{proof}
 For the last term in \eqref{eq:telescopic}, we extend the summation $\sum_{\gamma < \gamma_0}$ to all indices $i,j \in [N]$; it is an upper bound as we only sum positive terms. Then, for every fixed $u\in [4]$, we need to gain from this summation of $\big[u \,  \text{\rm sum. bdd.}\, M \text{\rm -terms}\big]^{(\gamma)}_{\bm x, \bm y}$ over all $\gamma \in [\gamma(N)]$ precisely $N^{-u/2}$  compared to the naive $N^2$-size of the summation. This was achieved in Lemma \ref{lem:sumMs} by the index structure \eqref{eq:Msums} of the factors and application of several Schwarz inequalities \eqref{eq:MsumsSchwarz}. 
 
 Hence, combining \eqref{eq:telescopic} with Lemma \ref{lem:checknochecksim} and Lemma \ref{lem:sumMs}, we find that 
 \begin{equation*}
 	\Vert \Psi_k^{(\gamma_0)}(\bm x, \bm y) \Vert_p^p \le  C_1 \frac{1}{N^2} \sum_{\gamma < \gamma_0} {\Omega}_k^p(\gamma) + C_2 N^\xi \,. 
 \end{equation*}
 Since the rhs.~is independent of the elements in $I_{\bm x\bm y}$ (recall \eqref{eq:vectormax}), we can as well maximize over those on the lhs.~and arrive at Proposition \ref{prop:gronwall}.    \qed

 \subsubsection{Conclusion of the induction step}\label{sec:complete}
 Having Proposition \ref{prop:gronwall} and hence \eqref{eq:gronwallconclude} at hand, we can immediately deduce 
  \begin{equation*} 
 	\max_{\gamma \le \gamma(N)} \widecheck{\Omega}_k^p(\gamma) \lesssim N^\xi
 \end{equation*}
from Lemma \ref{lem:checknochecksim} above. This proves the $\widecheck\Psi$-part of \eqref{eq:conclusion} and thus finishes the induction step. 

Therefore, using uniformity of this bound, we conclude the proof of the isotropic multi-resolvent local laws \eqref{eq:zagmultiGISO}. 
 
 \subsection{Part~(b): Proof of the averaged law.} \label{subsec:av}The general idea of the proof of the averaged law is exactly the same as in the previous section: We replace all matrix elements one-by-one in $\gamma(N) \sim N^2$ steps and sum up the changes over all positions $\gamma\in [\gamma(N)]$ (cf.~\eqref{eq:telescope}). However, there are a several (minor) differences in the averaged case compared to Section \ref{subsec:iso}, which we will explain in the following. 
 
 Since both, averaged and isotropic normalized differences, \eqref{eq:Psikav} and \eqref{eq:Psikiso}, appear, we shall henceforth reintroduce the superscripts $^{\rm av}$ and $^{\rm iso}$. 
 Moreover, contrary to the isotropic proof, in this part it is sufficient to consider an arbitrary fixed $k \in \N$ and perform a \emph{single induction} on the moment $p$ taken of $\Psi_k^{\rm av}$, i.e.~$\E |\Psi_k^{\rm av}|^p = \Vert \Psi_k^{\rm av}\Vert_p^p$. We point out that the induction on $k$ used in the previous section is not needed, because the proof of the isotropic laws has already been concluded (see \eqref{eq:fourthorderAV} later). Hence, as the\emph{ induction hypothesis}, we will assume that
 \begin{equation} \label{eq:inductionhypoav}
 	\max_{\gamma \le \gamma(N)} \Vert \Psi_k^{\rm av, (\gamma)} \Vert_{p-1} +  \max_{\gamma \le \gamma(N)} \Vert \widecheck{\Psi}_k^{\rm av, (\gamma)} \Vert_{p-1}  \lesssim N^\xi
 \end{equation} 
 holds uniformly in traceless matrices for all $\xi > 0$, and our goal is to prove the same
 relation with $p$ replacing $p-1$. 
  The base case is thus simply the trivial bound 
 ($p=1$) given by $\E |\Psi_k^{\rm av}|^0 = 1$. 
 To ease notation, just as in Section \ref{subsec:iso}, we will drop the subscripts for all resolvents and deterministic matrices, i.e.~write $G_j =G$ and $A_j =A$ instead. Moreover, whenever it does not lead to confusion, we will drop all further sub- and superscripts.
 
 Completely analogously to Section \ref{subsec:iso}, we use resolvent expansions from Lemma \ref{lem:resolexp} to prove the exact agreement of the orders $r \in \{0,1,2,3\}$ as in Lemma \ref{lem:firstthree}. For the higher order terms (again focusing on the most critical fourth order ones, see Section \ref{subsec:fourthorder}), we argue completely analogously to \eqref{eq:fourthorder}, but now we have an additional effect: Whenever an intact averaged chain gets destroyed by a replacement $G \to G \Delta G$ from a derivative,
  we obtain (a sum of) isotropic chains with a $1/N$ prefactor  from the
 normalization of the trace, i.e.
 \begin{equation} \label{eq:avtoiso}
 	\langle (GA)^k \rangle \longrightarrow \langle G \Delta (GA)^k \rangle = \frac{1}{N} \big( (GA)^kG \big)_{\bm e_i \bm e_j} + \frac{1}{N} \big( (GA)^kG \big)_{\bm e_j \bm e_i}\,. 
 \end{equation}
 In this way, the analogue of \eqref{eq:fourthorder} reads  
 \begin{equation} \label{eq:fourthorderAV}
 	\E \sum_{d=1}^{4 \wedge p} \big|\widecheck{\Psi}_k^{\rm av}(\bm x, \bm y)\big|^{p-d} \left(\frac{N \eta}{\rho}\right)^{d/2}N^{-d(k/2-1)} \frac{1}{N^d} \sum_{(4-d) \Delta \,  \leadsto \, d } \big|  \underbrace{\big(... \Delta... \Delta ... \big)_{\bm{e}_i \bm{e}_j}\cdot ... \cdot \big(... \Delta... \big)_{\bm{e}_j \bm{e}_i}}_{(4-d)\;  \Delta \; \text{in a total} \; d \, \text{iso chains}} \big|\,,
 \end{equation}
 where the isotropic chains referred to in \eqref{eq:fourthorderAV}, are precisely those obtained in \eqref{eq:avtoiso}. In particular, one $\Delta$ has already been "used" for each destroyed averaged chain, hence only $(4-d)$ $\Delta$'s are placed in the isotropic chains (recall \eqref{eq:summation}). Observe that, after writing $N^{-d(k/2-1)} /N^d = N^{-dk/2}$, beside from the unit vectors in the isotropic chains, the structure of \eqref{eq:fourthorderAV} is exactly the same for \eqref{eq:fourthorder}. 
 
 Next, in each of the resulting four resolvent chains in the rhs.~of \eqref{eq:fourthorderAV}, as before we  add and subtract the corresponding $M$-term, again schematically written as $G = (G-M) + M$.
  Exactly as in the previous section, we have to distinguish two cases. 
 \begin{itemize}
 	\item[Case (i):] At least $d$ of the $4$ resolvent chains are replaced by their fluctuating part, $G-M$.
 	\item[Case (ii):] At least $5-d$ of the $4$ resolvent chains are replaced by their deterministic counterpart,  $M$. 
 \end{itemize}
 \vspace{2mm}
 \underline{Case (i):} First, we note that, since there are only strictly lower moments of $\Psi_k^{\rm av}$ appearing in \eqref{eq:fourthorderAV} after the resolvent expansion, we can directly employ the \emph{induction hypothesis} \eqref{eq:inductionhypoav}, i.e.~there is no possibility of preserving the destroyed chains unlike in \eqref{eq:sustainexample}. Therefore, by additionally applying the already established isotropic laws from the previous section in combination with $\big| (M_{j+1})_{\bm u \bm v} \big| \lesssim N^{j/2}$ (recall also \eqref{eq:Mshorthand}), we find that 
 \begin{equation} \label{eq:case1avfinal}
 	\text{Case (i) terms of \eqref{eq:fourthorderAV}} \lesssim N^\xi \sum_{d=1}^{4 \wedge p} \left(\frac{N \eta}{\rho}\right)^{d/2}  N^{-dk/2} \left[ N^{dk/2} \left( \frac{\rho}{N \eta}\right)^{d/2} +...\right]\lesssim N^\xi\,,
 \end{equation}
 indicating terms with more than $d$ factors of $G-M$ by dots. This concludes the discussion of Case (i). 
 \\[2mm]
 \underline{Case (ii):} For the second case, we again recall that all purely deterministic terms are independent of the replacement step $\gamma$. Moreover, completely analogously to Case (ii) in Section \ref{subsec:fourthorder}, it is not sufficient to just estimate every isotropic $M$-term blindly -- instead we again need to \emph{gain from the summation} over all replacement positions. We again illustrate this by an example. 
 \begin{example}\label{ex:ave}
 	We first consider $d=1$ and use the notation \eqref{eq:Mshorthand}. Then, by means of the induction hypothesis \eqref{eq:inductionhypoav}, we have the trivial estimate
 	\begin{equation} \label{eq:d=1exampleCase2AV}
 		\begin{split}
 			&\E \big|\widecheck{\Psi}^{\rm av}_k(\bm x, \bm y)\big|^{p-1} \left(\frac{N \eta}{\rho}\right)^{1/2} N^{-k/2}\sum_{\substack{0\le k_l \le k: \\
 					\sum_{l} k_l= k}}\left[\big| \big( M_{k_1+1}\big)_{\bm e_i \bm{e}_i}  \big(M_{k_2+1}\big)_{\bm e_j \bm{e}_j} \big( M_{k_3+1}\big)_{\bm e_i \bm{e}_i} \big( M_{k_4+1}\big)_{\bm e_j \bm{e}_j}  \big| + ... \right] \\
 			\lesssim \, &N^\xi  \left(\frac{N \eta}{\rho}\right)^{1/2} N^{-k/2}\sum_{\substack{0\le k_l \le k: \\
 					\sum_{l} k_l = k}} \left[ N^{\sum_l k_l/2} + ... \right] \lesssim N^\xi \left(\frac{N \eta}{\rho}\right)^{1/2}\,, 
 		\end{split}
 	\end{equation}
 	analogously to \eqref{eq:d=1exampleCase2}. Again, this bound is off by a factor $(N \eta/\rho)^{1/2}$, which can be improved on by averaging over all replacement position. 
 	
Compared to the isotropic case, we can no longer gain from summing over off-diagonal terms of the form $M_{\bm x \bm e_i}$. Instead, now we sum over squares of terms of the form $M_{\bm e_i \bm e_i}$ and estimate it by
  \begin{equation} \label{eq:schwarzAV}
 	\sum_{i} \big| M_{\bm e_i \bm{e}_i}\big|^2 \le \sum_{i,j} \big|M_{\bm e_i \bm{e}_j}\big|^2 \le \sum_i \big( |M|^2\big)_{\bm e_i \bm{e}_i} =  N \langle  |M|^2 \rangle\,,
 \end{equation}
 similarly to \eqref{eq:schwarzfirst}--\eqref{eq:schwarzfirst2}. 
 Note that \eqref{eq:schwarzAV} is better than the trivial bound, which would give $ N\Vert M \Vert^2$. 
	The key for exploiting this improvement is the following lemma, the proof of which is given in Appendix~\ref{sec:addtech}.  
 	\begin{lemma} \label{lem:gainAV}Using the assumptions and notations from Lemma \ref{lem:Mbound} and the normalization $\langle |A_i|^2 \rangle = 1$, we have that 
 		\begin{equation} \label{eq:M^2estimateAV}
 			\big\langle  \big|\mathcal{M}(z_1, A_1, ... , A_k, z_{k+1}; \mathfrak{I}_{k+1})\big|^2 \big\rangle \lesssim N^{k} \left(\prod_{i \in \mathfrak{I}_{k+1}} \rho_i \right)^2\left[ \left(\frac{ \max_{i\in [k+1]} \big(\rho_i + \mathbf{1}(i \notin \mathfrak{I}_{k+1})\big)}{N \ell}\right)^2 \vee \frac{1}{N} \right]\,.
 		\end{equation}
 	\end{lemma}
  Applying \eqref{eq:M^2estimateAV} for $k= k_l$ and $\mathfrak{I}_{k_l+1} = \emptyset$ (recall \eqref{eq:Mdef}, \eqref{eq:Mdefim}, and \eqref{eq:Mshorthand}), we see the bound 
\begin{equation} \label{eq:gainAVapplied}
 \langle  |M_{k_l+1}|^2 \rangle \lesssim N^{k_l} \left[\left(\frac{\rho}{N \eta}\right)^2 \vee \frac{1}{N}\right]\,. 
\end{equation}
We remark that this estimate is better by the factor $\big[\big(N \eta/\rho\big)^{-2} \vee N^{-1}\big] \ll 1$ compared to the naive norm bound $|(M_{k_l+1})_{\bm u \bm v}|^2 \le \Vert M_{k_l+1} \Vert^2 \lesssim N^{k_l}$ from Lemma \ref{lem:Mbound}~(b) employed in \eqref{eq:d=1exampleCase2AV}.
  	Hence, fixing one constellation of $k_l$'s in \eqref{eq:d=1exampleCase2AV}, we find the average of the first line in \eqref{eq:d=1exampleCase2AV} over all $i,j \in [N]$ to be bounded by 
 \begin{equation} \label{eq:d=1exampleCase2AVSUM}
 	\begin{split}
 		&N^\xi \left(\frac{N \eta}{\rho}\right)^{1/2} N^{-k/2} \frac{1}{N^2}\sum_{i,j}\left[\big| \big( M_{k_1+1}\big)_{\bm e_i \bm{e}_i}  \big(M_{k_2+1}\big)_{\bm e_j \bm{e}_j} \big( M_{k_3+1}\big)_{\bm e_i \bm{e}_i} \big( M_{k_4+1}\big)_{\bm e_j \bm{e}_j}  \big| + ... \right] \\
 		\lesssim &N^\xi\left(\frac{N \eta}{\rho}\right)^{1/2} N^{-k/2} \frac{1}{N^2} \left[\prod_{l \in [4]} \left(\sum_{i} \big|\big( M_{k_l+1}\big)_{\bm e_i \bm{e}_i}\big|^2 \right)^{1/2} + ... \right] \\
 		\lesssim &N^\xi \left(\frac{N \eta}{\rho}\right)^{1/2} N^{-k/2}\frac{1}{N^2}\left[\left(\prod_{l \in [4]}  N^{k_l+1} \left[\left(\frac{\rho}{N \eta}\right)^2 \vee \frac{1}{N}\right]\right)^{1/2}+ ... \right] \\
 		\lesssim &N^\xi \left[\left(\frac{\rho}{N\eta}\right)^{7/2} \vee \left(\frac{\eta}{\rho}\right)^{1/2} \frac{1}{N^{3/2}} \right] \lesssim N^\xi\,. 
 	\end{split}
 \end{equation}
 To go from the first to the second line, we employed a trivial Schwarz inequality. To go to the penultimate line, we used \eqref{eq:schwarzAV} with $M = M_{k_l+1}$. For the final estimate, we employed $(\prod_{l \in [4]}  N^{k_l+1})^{1/2} = N^{k/2+2}$.

 Next, we consider one example for $d=4$, where all four resolvent chains are replaced by their deterministic counterpart. In this case, the analog of \eqref{eq:d=1exampleCase2AVSUM} reads
 \begin{equation*}
 	\begin{split}
 		&N^\xi \left(\frac{N \eta}{\rho}\right)^{2} N^{-2k} \frac{1}{N^2}\sum_{i,j}\left[\big| \big( M_{k+1}\big)_{\bm e_i \bm{e}_j}  \big(M_{k+1}\big)_{\bm e_j \bm{e}_i} \big( M_{k+1}\big)_{\bm e_i \bm{e}_j} \big( M_{k+1}\big)_{\bm e_j \bm{e}_i}  \big| + ... \right] \\
 		\lesssim &N^\xi\left(\frac{N \eta}{\rho}\right)^{2} N^{-k} \frac{1}{N^2} \left[\sum_{i,j} \big|\big( M_{k+1}\big)_{\bm e_i \bm{e}_j}\big|^2 + ... \right] \\
 		\lesssim &N^\xi \left(\frac{N \eta}{\rho}\right)^{2} N^{-k}\frac{1}{N^2}\left[ N^{k+1}\left[\left(\frac{\rho}{N \eta}\right)^2 \vee \frac{1}{N}\right]+ ... \right] \lesssim N^\xi \left[\frac{1}{N} \vee \left(\frac{\eta}{\rho}\right)^2\right] \lesssim N^\xi\,. 
 	\end{split}
 \end{equation*}
 To go from the first to the second line, we estimated two factors of $M_{k+1}$ by their norm, $\big| (M_{k+1})_{\bm u \bm v}\big| \lesssim N^{k/2}$. Next, to go to the third line, we employed \eqref{eq:schwarzAV} and Lemma \ref{lem:gainAV}. The final estimate used $\eta/\rho \lesssim 1$. 
 \end{example}
 The above examples showcase the general mechanism for the terms in Case (ii): After estimating all the $(G-M)$-type terms with the aid of the induction hypothesis \eqref{eq:inductionhypoav}, we are left with an excess $(N \eta/\rho)^{u/2}$-factor, for some $u \in [4]$. Analogously to \eqref{eq:Msummable}--\eqref{eq:Msums}, this leftover factor is then controlled by \emph{gaining from the summation} like in \eqref{eq:M^2estimateAV}. We skip the simple counting argument ensuring this gain. 
 
 The treatment of the further higher order terms and the truncation of the resolvent expansion is completely analogous to Sections \ref{subsec:lowerorder} and \ref{subsec:truncation}, respectively. Therefore, by telescopic summation like in \eqref{eq:telescope}, we find that  
 \begin{equation*}
 	\max_{\gamma \le \gamma(N)} \Vert \Psi_k^{\rm av, (\gamma)} \Vert_p^p + \max_{\gamma \le \gamma(N)} \Vert \widecheck{\Psi}_k^{\rm av, (\gamma)}\Vert_p^p \lesssim  \Vert \Psi_k^{\rm av, (0)}\Vert_p^p  +  N^\xi \lesssim N^\xi
 \end{equation*}
 where in the last step
 we absorbed $\Vert \Psi_k^{\rm av, (0)}\Vert_p^p$ into $N^\xi$ by our initial assumption that we have multi-resolvent local laws \eqref{eq:multiG} for the matrix $H^{(\bm v)} = H^{(0)}$. 
 The checked version is obtained completely analogously to Lemma \ref{lem:checknochecksim}. 
 
 This completes the proof of the induction step. We have thus finished the argument for the averaged case and hence the proof of Proposition \ref{prop:zag}. \qed

\subsection{The case $\mathfrak{I}_k \neq \emptyset$ (resp.~$\mathfrak{I}_{k+1} \neq \emptyset$)} \label{subsec:withIM} In this section, we explain how to adjust the above argument for proving Proposition \ref{prop:zag} in the case that at least one of the resolvents in the chains of interests 
$$
\langle \mathcal{G}_1 A_1 ... \mathcal{G}_k A_k \rangle \qquad \text{and} \qquad \big( \mathcal{G}_1 A_1 ... \mathcal{G}_k A_k \mathcal{G}_{k+1}\big)_{\bm x \bm y}
$$
is an imaginary part, i.e.~$\mathcal{G}_i = \Im G_i$ for at least one index $i \in [k]$ (resp.~$i \in [k+1]$). Recall the local laws for the average and isotropic chain from \eqref{eq:zagmultiG} and \eqref{eq:zagmultiGISO}, respectively. Compared to the case of no imaginary parts, handled in the previous Sections \ref{subsec:iso}--\ref{subsec:av}, there are now two changes: First, the bound contains the product $\prod_{i \in \mathfrak{I}} \rho_i$ (instead of one). Second, the smallness factor $(N\eta/\rho)^{-1/2}$ from before is now replaced by $(N \ell)^{-1/2}$

For adjusting the first change, the simple but key insight is, that when applying the resolvent expansion from Lemma~\ref{lem:resolexp} to both $G$ and $G^*$ in $\Im G = \frac{1}{2 \ii }(G-G^*)$, we can always "restore"
exactly one $\Im G$ on the rhs. More precisely, taking \eqref{eq:backwardexp} for concreteness and using $\Delta = \Delta^*$, we have that 
\begin{equation*}
\begin{split}
\Im G = \frac{1}{2 \ii } \big[G - G^*\big] &= \frac{1}{2 \ii} \bigg[ \left(\widecheck{G} - N^{-1/2} \widecheck{G} \Delta \widecheck{G}+ N^{-1} \widecheck{G} \Delta\widecheck{G} \Delta \widecheck{G} + ... \right) \\
& \qquad \qquad - \left(\widecheck{G}^* - N^{-1/2} \widecheck{G}^* \Delta \widecheck{G}^*+ N^{-1} \widecheck{G}^* \Delta\widecheck{G}^* \Delta \widecheck{G} ^*+ ... \right)\bigg] \\
& = \Im \widecheck{G} - N^{-1/2} \big(\Im \widecheck{G} \Delta \widecheck{G} + \widecheck{G}^*\Delta \Im \widecheck{G}\big) \\[1mm]
& \qquad \qquad + N^{-1} \big(  \Im \widecheck{G} \Delta \widecheck{G} \Delta \widecheck{G}+ \widecheck{G}^*\Delta \Im \widecheck{G} \Delta \widecheck{G} + \widecheck{G}^*\Delta \widecheck{G}^*\Delta \Im \widecheck{G}\big) + ... 
\end{split}
\end{equation*}
In this way, the imaginary parts in the original chain are "preserved" by the resolvent expansion. Recall that $|\Im \widecheck{G}_{\bm u \bm v}(z)| \prec \rho(z)$ (as a consequence of \eqref{eq:singleGoptimal} for $N |\Im z| \rho(z) \gg 1$; recall $N \hell \gg 1$), which improves \eqref{eq:Gbdd1GFT}. In particular, using Lemma \ref{lem:Mbound}, we find that the factor $\big(\prod_{i \in \mathfrak{I}} \rho_i\big)^{-d}$ stemming from the correct normalisation of the analog of $\Psi_k^{\rm av/iso}$ in \eqref{eq:Psikav}--\eqref{eq:Psikiso} and thus appearing in the expression analogous to \eqref{eq:fourthorder} is naturally compensated by a product of $\rho$'s stemming from the destroyed chains. 

For adjusting to the second change, it suffices to replace every $\eta/\rho$ appearing in Sections \ref{subsec:iso}--\ref{subsec:av} by $\ell$ and realize that the complement of the interesting regime, i.e.~the regime $\ell\ge 1$ is already proven in Proposition~\ref{prop:initial}.

\appendix

\section{Additional technical results}
\label{sec:addtech}

In this section we prove several additional technical results which are used in the main sections.

\subsection{Bounds on the deterministic approximations}
\begin{proof}[Proofs of Lemma \ref{lem:Mbound} and the claim in Remark \ref{rmk:MHS}~(ii)]
We will first proof the following stronger bound in Lemma \ref{lem:Mboundstrong}, from which we immediately deduce Lemma \ref{lem:Mbound} and the claim in Remark \ref{rmk:MHS}~(ii). 
The proof of the following lemma is given at the end of the current section. 
\begin{lemma}\label{lem:Mboundstrong}
Fix $k \ge 1$. Consider spectral parameters $z_1, ... , z_k \in \C \setminus \R$ and traceless matrices $A_1, ... , A_k \in \C^{N \times N}$, and 
define for every $j \in [k]$
$${\eta}_j := |\Im z_j|, \qquad \rho_j := \frac{1}{\pi}|\Im m_{\rm sc}(z_j)|,  \qquad
\ell :=  \min_j\big[\eta_j (\rho_j + \mathbf{1}(j \notin \mathfrak{I}_k))\big]\,. 
$$
Then, for every $1 \le s \le \lfloor k/2 \rfloor$ and $\pi \in \mathrm{NC}([k])$ with $|\pi| = k+1-s$ \nc, it holds that
\begin{equation} \label{eq:Mboundstrong}
\left| 
\langle \mathrm{pTr}_{K(\pi)}(A_1,\ldots,A_{k-1})A_k \rangle \prod_{S\in\pi} m_\circ^{(\mathfrak{I}_k)}[S] \right| \lesssim \left(\prod_{j \in \mathfrak{I}_k} \rho_j  \right) \frac{1}{\ell^{s-1}} \prod_{\substack{S\in K(\pi) \\ |S| \ge 2}}\prod_{j\in S} \left\langle|A_j|^{|S|}\right\rangle^{\frac{1}{|S|}}   \,. 
\end{equation}
with $m_\circ^{(\mathfrak{I})}[S]$ being defined above \eqref{eq:Mdivdiff}. For $s > \lfloor k/2 \rfloor$
the lhs.~of \eqref{eq:Mboundstrong} equals zero. 
\end{lemma}

For the proof of Lemma \ref{lem:Mbound}~(a) and the claim in Remark \ref{rmk:MHS}~(ii) concerning \eqref{eq:mainAV} we use that $\langle |A|^p\rangle^{1/p} \le N^{\frac{p-2}{2p}} \langle |A|^2 \rangle^{1/2}$ for any $p \ge 2$,
and hence
$$
\text{rhs. of} \, \eqref{eq:Mboundstrong} \lesssim N^{k/2 - 1} \left(\prod_{j \in \mathfrak{I}_k} \rho_j  \right) \left(\prod_{j=1}^k \langle |A_j|^2 \rangle^{1/2} \right) \frac{1}{(N \ell)^{s-1}}\,. 
$$
This shows that, in particular, all terms with $s>1$ in \eqref{eq:Mboundstrong} are explicitly smaller than the error term in \eqref{eq:mainAV}, where we used that $N \ell \gg 1$. The $s=1$ term exactly constitutes the deterministic approximation in \eqref{eq:Mreplace}, i.e.~the sum in \eqref{eq:Mboundstrong} contains exactly one term 
$$
\sum_{\substack{\pi\in\mathrm{NC}([k]): \\
		|\pi| = k}} \langle \mathrm{pTr}_{K(\pi)}(A_1,\ldots,A_{k-1})A_k \rangle \prod_{S\in\pi} m_\circ^{(\mathfrak{I}_k)}[S] = \bigg(\prod_{j \in \mathfrak{I}_k} \Im m_j\bigg) \bigg( \prod_{j \notin \mathfrak{I}_k} m_j\bigg)\langle A_1 ... A_k \rangle \,.
$$
Here we used that $|\pi| = k$ implies that the Kreweras complement consists of the full set, $K(\pi) = [k]$. 

Finally, for the proof of Lemma \ref{lem:Mbound}~(b) and the claim in Remark \ref{rmk:MHS}~(ii) concerning \eqref{eq:mainISO} (i.e.~the corresponding isotropic bounds) we argue completely analogously to Section \ref{subsec:pureIMISO}. 
\end{proof}
It remains to  prove  Lemma \ref{lem:Mboundstrong}. 
\begin{proof}[Proof of Lemma \ref{lem:Mboundstrong}]
Fix an arbitrary non-crossing partition $\pi \in \mathrm{NC}([k])$ consisting of $|\pi| = k+1-s$ blocks. 

First, note that, in order to get a non-vanishing partial trace 
$$
\langle \mathrm{pTr}_{K(\pi)}(A_1,\ldots,A_{k-1})A_k \rangle = \prod_{S\in K(\pi)}\left\langle\prod_{j\in S}A_j\right\rangle
$$
the minimal size of a block $S \in K(\pi)$ is two (using that the $A_i$'s are traceless). 
Therefore, by application of Hölder's inequality, 
\begin{equation} \label{eq:Apartfinal}
\begin{split}
\big| \langle \mathrm{pTr}_{K(\pi)}(A_1,\ldots,A_{k-1})A_k \rangle \big| \le  \prod_{\substack{S\in K(\pi) \\ |S| \ge 2}}\prod_{j\in S} \left\langle|A_j|^{|S|}\right\rangle^{\frac{1}{|S|}} \,. 
\end{split}
\end{equation}

In order to estimate $\prod_{S\in\pi} m_\circ^{(\mathfrak{I}_k)}[S]$, we recall the Möbius inversion formula \cite[Lemma~2.16]{thermalisation} 
\begin{equation} \label{eq:Mobius}
m_\circ^{(\mathfrak{I}_k)}[S] = m^{(\mathfrak{I}_k)}[S] + \sum_{\substack{\pi \in \mathrm{NC}(S) \\ |\pi| \ge 2}} (-1)^{|\pi|-1} \left( \prod_{T \in K(\pi)} C_{|T| -1} \right) \prod_{U \in \pi} m^{(\mathfrak{I}_k)}[U]
\end{equation}
where $C_n$ is the $n^{\rm th}$ Catalan number. Hence, it suffices to bound the iterated divided differences $m^{(\mathfrak{I}_k)}[S]$ for a subset $S \subset [k]$ as  
\begin{equation} \label{eq:intrepbound}
\left| m^{(\mathfrak{I}_k)}[S] \right| \lesssim \frac{\prod_{i \in \mathfrak{I}_k \cap S} \rho_i}{\ell^{|S| - 1}}
\end{equation}
which is a direct consequence of the integral representation \eqref{eq:Mdivdiff}. Indeed, combining \eqref{eq:Mobius} with \eqref{eq:intrepbound} and using that the sum in \eqref{eq:Mobius} is restricted to partitions of $S$ with at least two blocks, we obtain
\begin{equation} \label{eq:zpartfinal}
\left| \prod_{S\in\pi} m_\circ^{(\mathfrak{I}_k)}[S] \, \right| \lesssim \left(\prod_{i \in \mathfrak{I}_k} \rho_i\right)\frac{1}{\ell^{s-1}}
\end{equation}
where we additionally used that the original non-crossing partition $\pi \in \mathrm{NC}([k])$ consists of exactly $k+1-s$ blocks. Combining \eqref{eq:zpartfinal} with \eqref{eq:Apartfinal}, we conclude the proof of \eqref{eq:Mboundstrong}. 

For $s > \lfloor k/2 \rfloor$, we note that the Kreweras complement $K(\pi)$ necessary contains singletons, and hence 
the lhs.~of \eqref{eq:Mboundstrong} vanishes since $\langle A_i\rangle=0$.
\end{proof}
We conclude this section by giving the proof of Lemma \ref{lem:gainAV}. 

\begin{proof}[Proof of Lemma \ref{lem:gainAV}]
	The principal idea of the proof is very similar to the previous ones given in this section, hence we provide only a brief argument. 
	
Recalling \eqref{eq:Mdefim}--\eqref{eq:Mdivdiff}, we have that 
\begin{equation} \label{eq:M^2AVestpf}
\big\langle \big| \mathcal{M}(z_1,A_1,\dots,A_{k},z_{k+1};\mathfrak{I}_{k+1})\big|^2 \big\rangle \lesssim \sum_{\pi\in\mathrm{NC}([k+1])} \big\langle \big| \mathrm{pTr}_{K(\pi)}(A_1,\ldots,A_{k})  \big|^2 \big\rangle \left| \prod_{S\in\pi} m_\circ^{(\mathfrak{I}_{k+1})}[S] \right|^2\,.
\end{equation}

Next, analogously to Lemma \ref{lem:Mboundstrong} above, we decompose the summation over all partitions $\pi$ into groups, where $|\pi| = k+2-s$ with $1 \le s \le \lceil (k+1)/2 \rceil$ is fixed (note that $\lfloor \cdot \rfloor$ got replaced by $\lceil \cdot \rceil$ due to the presence of a non-traceless identity matrix). Moreover, for fixed $s$ we distinguish two cases in \eqref{eq:M^2AVestpf} (recall \eqref{eq:partrdef}): For Case (i), we assume that the unique block $\mathfrak{B}(k+1) \in K(\pi)$ containing $k+1$ contains no other elements, i.e.~$\mathfrak{B}(k+1)\setminus \{k+1\} = \emptyset$. For Case (ii), we assume that $\mathfrak{B}(k+1)\setminus \{k+1\} \neq \emptyset$.  
\\[2mm]
\underline{Case~(i).} First, we note that necessarily $s \ge 2$ in this case. Then, we have that 
\begin{equation*}
\big\langle \big| \mathrm{pTr}_{K(\pi)}(A_1,\ldots,A_{k})  \big|^2 \big\rangle \le \left(\prod_{\substack{S\in K(\pi) \setminus \mathfrak{B}(k+1) \\ |S| \ge 2}}\prod_{j\in S} \left\langle|A_j|^{|S|}\right\rangle^{\frac{1}{|S|}} \right)^2 \le \left(\frac{N^{k/2}}{N^{s-1}}\right)^2\,, 
\end{equation*}
analogously to \eqref{eq:Apartfinal}. Since in Case (i), $z_1$ and $z_{k+1}$ are always together in one block $S \in \pi$ with $|\pi| = k+2-s$, we obtain 
\begin{equation*} 
	\left| \prod_{S\in\pi} m_\circ^{(\mathfrak{I}_{k+1})}[S] \, \right|^2 \lesssim \left[\frac{\left(\prod_{i \in \mathfrak{I}_{k+1}} \rho_i\right) \wedge \max_{i \in [k+1]} \rho_i}{\ell^{s-1}} \right]^2
\end{equation*}
analogously to \eqref{eq:zpartfinal} by means of \eqref{eq:Mobius} and the integral representation \eqref{eq:Mdivdiff}. The additional $\wedge \max_{i \in [k+1]} \rho_i$, which is effective only for $\mathfrak{I}_{k+1} = \emptyset$, comes from the estimate
\begin{equation*}
\int_\R \frac{\rho(x)}{|x - z_1| \, |x - z_{k+1}|} \dd x\lesssim \frac{\rho_1 \vee \rho_{k+1}}{\ell}\,, 
\end{equation*}
easily obtained by a Schwarz inequality. 
\\[2mm]
\underline{Case~(ii).} In this case, the above estimates of the two factors in \eqref{eq:M^2AVestpf} modify to
\begin{equation*}
	\big\langle \big| \mathrm{pTr}_{K(\pi)}(A_1,\ldots,A_{k})  \big|^2 \big\rangle \le \left(\prod_{\substack{S\in K(\pi)  \\ |S| \ge  2}} \ \left(\prod_{j\in S_1} \left\langle|A_j|^{2(|S_1|-1)}\right\rangle^{\frac{1}{2(|S_1|-1)}}\right) \left(\prod_{i = 2}^s \prod_{j\in S_i} \left\langle|A_j|^{|S_i|}\right\rangle^{\frac{1}{|S_i|}}\right) \right)^2\,, 
\end{equation*}
assuming that $S_1 = \mathfrak{B}(k+1)$, and 
\begin{equation*} 
	\left| \prod_{S\in\pi} m_\circ^{(\mathfrak{I}_{k+1})}[S] \, \right|^2 \lesssim \left[\frac{\prod_{i \in \mathfrak{I}_{k+1}} \rho_i}{\ell^{s-1}} \right]^2\,. 
\end{equation*}
Putting the two cases together and using $\langle |A|^p\rangle^{1/p} \le N^{\frac{p-2}{2p}} \langle |A|^2 \rangle^{1/2}$ for any $p \ge 2$ together with $N \ell > 1$ and the normalization $\langle |A_j|^2 \rangle = 1$, we find that
\begin{equation*} 
	\big\langle \big| \mathcal{M}(z_1,A_1,\dots,A_{k},z_{k+1};\mathfrak{I}_{k+1})\big|^2 \big\rangle \lesssim N^k \, \left(\prod_{i \in \mathfrak{I}_{k+1}} \rho_i\right)^2 \left[ \left(\frac{ \max_{i\in [k+1]} \big(\rho_i + \mathbf{1}(i \notin \mathfrak{I}_{k+1})\big)}{N \ell}\right)^2 + \frac{1}{N} \right]\,. 
\end{equation*}
\end{proof}
\subsection{Proof of the global law in Proposition \ref{prop:initial}} \label{app:globallaw}
We only discuss the proof of the average case \eqref{eq:maininAV}, the isotropic case \eqref{eq:maininISO} 
is analogous and hence omitted. 
Set  $d := \min_i \mathrm{dist}(z_i, [-2,2])$ and recall that $d \ge \delta \gtrsim 1$. 

The case of no $\Im G$'s, i.e.~$\mathfrak{I}_k = \emptyset$, has already been dealt with
 in \cite[Appendix~A]{A2} and yielded the bound \eqref{eq:maininAV} with a factor $d^{-(k+1)}$
  instead of $\sqrt{\max_i \rho_i / \ell}$. 
  In the $d \gtrsim 1$ regime, this bound  is in fact stronger,
  $d^{-(k+1)} \lesssim d^{-1} \lesssim\sqrt{\max_i \rho_i / \ell} $,
  since $|\rho(z)| \sim |\Im z|/\mathrm{dist}(z, [-2,2])^2$ and $\ell \sim \min_i |\Im z_i|$.

In case of $\mathfrak{I}_k \neq \emptyset$ 
we need to gain from the fact that the original chain contained $\Im G$'s. The principal idea is analogous to \cite[Appendix B]{multiG} and \cite[Appendix A]{A2}, as we employ a cumulant expansion and argue by induction on the length $k$ of the initial chain. However, in order to gain from the imaginary parts, the key observation is that within the cumulant expansion, the total number of $\Im$'s is preserved, as becomes apparent from the formula
$$\partial_{ab} \Im G=G\Delta^{ab}\Im G+\Im G\Delta^{ab}G^*$$
for the derivative of an $\Im G$ factor.
 Here, $\partial_{ab}$ denotes the partial derivative w.r.t.~the matrix entry $w_{ab}$ of the Wigner matrix $W$ and $\Delta^{ab}$ is a matrix consisting of all zeroes except for the $(a,b)$--entry which is equal to one. 
 Using the norm bounds  $\Vert \Im G_j \Vert \le |\Im z_j|/\mathrm{dist}(z_j, [-2,2])^2 \sim  \rho_j$ 
 and $\Vert G_j \Vert \le 1/d$ 
 by spectral decomposition, 
 we obtain \eqref{eq:maininAV} but with a factor $d^{k+1-|\mathfrak{I}_k|} $ instead of
  $\sqrt{\ell}$, 
  analogously to \cite[Eq.~(A.2)]{A2}. Finally, since $\sqrt{\ell} \lesssim d\lesssim d^{k+1-|\mathfrak{I}_k|} $,  this concludes the proof. \qed

\nc

\subsection{Complex moment matching} \label{app:moma}
	In order to conduct the third step of our proof, the Green function comparison (GFT) of Proposition \ref{prop:zag}, we need to guarantee the moment matching condition \eqref{eq:momentmatch} of the single entry distributions. For real random variables (or complex ones with independent real and imaginary parts), the argument ensuring this (and even an approximately matching fourth moment) is standard (see, e.g., \cite[Lemma~16.2]{EYbook}) and based on an explicit construction of a distribution supported on three points in $\R$. 	However, for general complex random variables,
this construction is not sufficient; we now present its  complex variant.

Let $Z$ be a complex random variable and denote its moments by
\begin{equation} \label{eq:moment}
	m_{i,j}=m_{i,j}(Z) := \E \big[\overline{Z}^i Z^j\big] \qquad \text{for} \qquad i,j \in \N_0\,, 
\end{equation}
and call $i+j$ the \emph{order} of $m_{i,j}$. 
Clearly $m_{0,0} = 1$ and $m_{i,j} = \overline{m}_{j,i}$, so we can focus on $m_{i,j}$ with $i\le j$. 
\begin{lemma} \label{lem:momentmatch}
	Let $m_{0,2}, m_{0,3}, m_{1,2} \in \C$ with $|m_{0,2}| \le 1$. 
	Then there exists a complex random variable $Z$ supported on at most eleven points $z_1, ... , z_{11} \in \C$, such that its moments \eqref{eq:moment} are given by 
	\begin{equation} \label{eq:momentmatch2} 
		m_{0,1}(Z) = 0\,, \quad m_{1,1}(Z) = 1\,, \quad m_{0,2}(Z) = m_{0,2}\,, \quad m_{0,3}(Z) = m_{0,3}\,, \quad \text{and} \quad m_{1,2}(Z) = m_{1,2}\,. 
	\end{equation}
\end{lemma}
\begin{remark}
	A generalized version of this problem (constructing an atomic measure with arbitrary number
	of prescribed moments), known as the \emph{truncated complex $K$-moment problem}, has been solved by Curto and Fialkow in \cite{CuFi}. To keep our result self-contained, we give a simple independent proof for the 
	special case of three moments that we need here. 
\end{remark}

Having Lemma \ref{lem:momentmatch} at hand, one can easily see that there exists a random variable that has the prescribed first three moments and it has an
 independent  Gaussian component of given variance $\gamma>0$. More precisely, given 
$m_{0,1}= 0$, $m_{1,1} = 1$, $m_{0,2}$,  $m_{0,3}$, and  $m_{1,2}$ with $|m_{0,2}| \le 1$ as the 
set of moments of $\chi_{\mathrm{od}}$,  we look for a representation of $Z$ in the form 
\begin{equation*}
	Z := (1 - \gamma)^{1/2} Z' + \gamma^{1/2} \xi_G  \quad \text{with} \quad \gamma \in (0,1) \quad \text{fixed}
\end{equation*}
with some random variable $Z'$ to be constructed, where
$\xi_G$ is a centered complex Gaussian random variable having second moments $m_{0,2}(\xi_G) = m_{0,2}$ and $m_{1,1}(\xi_G) = 1$.  The moments of $Z'$ thus satisfy the relations
\begin{equation}\label{seq}
m_{i,j} = (1 - \gamma)^{(i+j)/2} m_{i,j}(Z') + \gamma^{(i+j)/2} m_{i,j}(\xi_G) \quad \text{with} \quad 1 \le i+j \le 3.
\end{equation}
In particular, $|m_{0,2}(Z')| = |m_{0,2}|\le 1$, so the moment sequence $m_{i,j}(Z')$ from \eqref{seq} satisfy
the only nontrivial condition of Lemma \ref{lem:momentmatch}. Therefore, by Lemma \ref{lem:momentmatch}, we can construct the random variable $Z'$.
Finally, we remark that all random variables involved have arbitrarily high moments (cf.~Assumption~\ref{ass:entries}).
This moment matching argument shows how to choose the distribution of the
 initial condition $W_0$ of the Ornstein-Uhlenbeck flow
\eqref{eq:OUOUOU} so that after time $T=\gamma$ it will match with the distribution of the original matrix $W$ 
up to three moments. 

\begin{proof}[Proof of Lemma \ref{lem:momentmatch}]
	We only outline the construction of the points $z_1, ... , z_{11} \in \C$, the precise computations are a simple exercise in calculus and linear algebra and hence omitted. 
	
	We set  $z_{11} = 0$ to be the origin. The remaining ten points are then placed on five lines through the origin, carrying two points each, i.e.~we put
	\begin{equation*}
		z_j = r_j \ee^{\ii \varphi_j} \quad \text{and} \quad z_{11-j} = \hat{z}_j := - \hat{r}_j \ee^{\ii \varphi_j} \quad \text{with} \quad r_j, \hat{r}_j \ge 0\,, \varphi_j \in [0, 2 \pi) \quad \text{for} \quad j \in [5]\,. 
	\end{equation*}
	For simplicity, we can even prescribe four of the five angular variables in such a way that the corresponding points lie
	on the real and imaginary axis and the two diagonals, i.e. set~$\varphi_j := j \pi/4$ for $j \in [4]$. 
	
	We then take the law of $Z$ to be of the form 
	\begin{equation*}
		\sum_{j \in [5]} \big( p_j \delta_{z_j} + \hat{p}_j \delta_{\hat{z}_j} \big) + \bigg(1 - \sum_{j \in [5]} (p_j + \hat{p}_j)\bigg) \delta_0
	\end{equation*}
	for weights $p_j, \hat{p}_j \ge 0$ satisfying $\sum_{j \in [5]}(p_j + \hat{p}_j) \le 1$. As mentioned above, it is a simple exercise to show that the remaining parameters $r_j, \hat{r}_j, p_j, \hat{p}_j \ge 0$ for $j \in [5]$ and $\varphi_5 \in [0, 2 \pi)$ can be chosen in such a way to accommodate \eqref{eq:momentmatch2}. More precisely, taking $A_j := p_j r_j = \hat{p}_j \hat{r}_j \ge 0$ for $j \in [5]$ (this ensures $m_{0,1}(Z) = 0$), $r_5 = \hat{r}_5$, and using our choices of $\varphi_j = j \pi/4$ for $j \in [4]$, the two complex conditions $m_{0,3}(Z) = m_{0,3}$ and $m_{1,2}(Z) = m_{1,2}$ turn into four real linear equations for the variables $C_j := B_j (r_j - \hat{r}_j) \in \R$ for $j \in [4]$ with $B_j := A_j (r_j + \hat{r}_j) \ge 0$. The determinant of this linear systems can easily seen to be non-vanishing and it thus determines the difference variables $r_j - \hat{r}_j \in \R$ for $j \in [4]$. Finally, the independent variables $\varphi_5 \in [0, 2 \pi)$ and $B_j := A_j (r_j + \hat{r}_j) \ge 0$ for $j \in [5]$ can easily be chosen to satisfy $m_{1,1}(Z) = 1$ and  $m_{0,2}(Z) = m_{0,2}$. 
\end{proof}

\subsection{Additional proofs for Section \ref{sec:opAedge}}
\label{sec:addproofsec4}

\begin{proof}[Proofs of Lemmas~\ref{lem:Mcancel} and \ref{lem:cancM}]
	 
	The claim of Lemma \ref{lem:Mcancel} follows by multi-linearity from Lemma~\ref{lem:cancM}. 
	
	For the proof of Lemma \ref{lem:cancM}, we will use a \emph{tensorization argument} (or \emph{meta argument}) similar to \cite{metaargument} and \cite[Proof of Lemma~D.1]{iidpaper}. Throughout this proof the size $N$ of $W$ is fixed. For $d\in\N$ consider the $(Nd)\times (Nd)$ Wigner matrix ${\bm W}^{(d)}$, i.e. the entries of ${\bm W}^{(d)}$ have variance $1/(Nd)$. Let ${\bm W}_t^{(d)}$ be the Ornstein-Uhlenbeck flow as in \eqref{eq:OUOUOU} with initial condition ${\bm W}_0^{(d)}={\bm W}^{(d)}$, and define its resolvent ${\bm G}_{i,t}^{(d)}:=({\bm W}_t^{(d)}-z_{i,t})^{-1}$, then the deterministic approximation of the resolvent is still given by $m_1$, the Stieltjes transform of the semicircular law.
	
	We now explain that also the deterministic approximation of products of resolvents and deterministic matrices is unchanged. For $1\le i\le k$, define ${\bm A}_i^{(d)}:=A_i\otimes I_d$, with $I_d$ denoting the $d$--dimensional identity, then for ${\bm M}_{[1,k],t}^{(d)}$ defined as in \eqref{eq:Mdef} with ${\bm M}_{i,t}^{(d)}$ and ${\bm A}_i^{(d)}$ we have
	\begin{equation}
		\label{eq:usefulrel}
		{\bm M}_{[1,k],t}^{(d)}:={\bm M}^{(d)}(z_{1,t},{\bm A}_1^{(d)}, \dots, {\bm A}_{k-1}^{(d)},z_{k,t})=M(z_{1,t},A_1,\dots,A_{k-1},z_{k,t})\otimes I_d.
	\end{equation}
	
	Fix $0<s<t$, then integrating \eqref{eq:flowka}
	 for the bold faced resolvent and deterministic matrices,  in time from $s$ to $t$ and taking the expectation we obtain
	\begin{equation}
		\begin{split}
			&\langle {\bm M}_{[1,k],t}^{(d)}\bm A_k\rangle- \langle {\bm M}_{[1,k],s}^{(d)}\bm A_k\rangle \\
			=&-\E \langle (\bm{G}_{[1,k],t}-{\bm M}_{[1,k],t}^{(d)})\bm A_k\rangle+\E \langle (\bm{G}_{[1,k],s}-{\bm M}_{[1,k],s}^{(d)})\bm A_k\rangle+\frac{k}{2}\int_s^t\E\langle \bm{G}_{[1,k],r}\bm A_k\rangle\,\dd r \\
			&+ \sum_{i,j=1\atop i< j}^k\int_s^t\E\langle \bm{G}_{[i,j],r}\rangle\langle \bm{G}_{[j,i],r}\rangle\, \dd r+\sum_{i=1}^k \int_s^t\E\langle \bm G_{i,r}-m_{i,r}\rangle \langle \bm{G}^{(i)}_{[1,k],r}\bm A_k\rangle\, \dd r+\frac{\sigma}{Nd}\sum_{i,j=1\atop i\le j}^k\int_s^t \E\langle \bm{G}_{[i,j],r}\bm{G}_{[j,i],r}^\mathfrak{t}\rangle\, \dd r\,.
		\end{split}
	\end{equation}
	
	Using the global law in Proposition~\ref{prop:initial} and \eqref{eq:usefulrel}, and taking the limit $d\to \infty$, this implies that for $|\Im z_i|\gtrsim 1$ we have
	\begin{equation}
		\label{eq:almthereder}
		\langle M_{[1,k],t}A_k\rangle- \langle M_{[1,k],s}A_k\rangle=\frac{k}{2}\int_s^t\langle M_{[1,k],r}A_k\rangle\,\dd r+\sum_{i,j=1, \atop i<j}^{k-1}\int_s^t\langle M_{[i,j],r}\rangle \langle M_{[j,i],r}\rangle\, \dd r.
	\end{equation}
	Finally, dividing \eqref{eq:almthereder} by $t-s$ and taking the limit $s\to t$, we conclude the proof of Lemma \ref{lem:cancM}.
\end{proof}

\begin{proof}[Proof of Lemma \ref{lem:redinphi}]
	The proof of this lemma is very similar to \cite[Lemma 3.3]{A2}. Hence we give the argument only for the case where $k$ is even, if $k$ is odd the proof is completely analogous. Moreover, for notational simplicity we henceforth drop the time dependence and the precise indices of $G_s$'s and $A$'s, i.e.~write $\Im G_s \equiv \Im G_{i,s}$, $A \equiv A_j$, $\rho \equiv \rho_i$ and so on. Then, by application of the general bound
	$$
	| \langle B_1 B_2 B_3 B_4 \rangle | \le N \prod_{i=1}^4 \langle |B_i|^2 \rangle^{1/2} \quad \text{for all} \quad B_i \in \C^{N \times N}
	$$
	applied to $B_i = \sqrt{|\Im G_s|} A (\Im G_s A)^{k/2-1} \sqrt{|\Im G_s|}$ and with the aid of \eqref{eq:Mbound}, we find that 
	\begin{equation*}
\begin{split}
\Phi_{2k}(s) &= \frac{\sqrt{N \hell_s}}{N^{k-1} \, \rho_s^{2k} \,  \langle |A|^2 \rangle^{k} } \big| \big\langle (\Im G_s A)^{2k} - \widehat{{M}}_{[\hat{1},\widehat{2k}],s} A \big\rangle  \big| \\
&\lesssim \sqrt{N \hell_s} + \frac{\sqrt{N \hell_s}}{N^{k-1} \, \rho_s^{2k} \,  \langle |A|^2 \rangle^{k} } N \big| \big\langle (\Im G_s A)^{k}  \big\rangle  \big|^2 \\
&\prec  \sqrt{N \hell_s} + \frac{\sqrt{N \hell_s}}{N^{k-1} \, \rho_s^{2k} \,  \langle |A|^2 \rangle^{k} } N \left[ N^{k/2-1} \rho_s^k \langle |A|^2\rangle^{k/2} \left( 1 + \frac{\phi_k}{\sqrt{N \hell_s}} \right)  \right]^2 \\
&\lesssim \sqrt{N \hell_s} + \frac{\phi_k^2}{\sqrt{N \hell_s}} \,. 
\end{split}
	\end{equation*}
We remark that, in order to bound $\big\langle (\Im G_s A)^{k}  \big\rangle $ in terms of $\phi_k$, we added and subtracted the corresponding $M$-term and used the assumption that $\Phi_k(s) \prec \phi_k$. 
\end{proof}

\end{document}